\documentclass[a4paper,reqno]{amsart}
\usepackage{geometry}
\usepackage{mathrsfs}
\usepackage{amsmath, amsthm, amscd, amssymb, latexsym, eucal}
\usepackage[all]{xy}

\numberwithin{equation}{section}
\newtheorem{thm}{Theorem}[section]
\newtheorem*{main-thm}{Main Theorem}
\newtheorem{cor}[thm]{Corollary}
\newtheorem{lem}[thm]{Lemma}
\newtheorem{prop}[thm]{Proposition}
\theoremstyle{definition}
\newtheorem{defn}[thm]{Definition}
\newtheorem{rem}[thm]{Remark}
\newtheorem{exam}[thm]{Example}

\newcommand{\lxr}{\longrightarrow}

\newcommand{\A}{\mathscr A}
\newcommand{\B}{\mathscr B}
\newcommand{\C}{\mathscr C}
\newcommand{\D}{\mathscr D}
\newcommand{\F}{\mathscr F}
\newcommand{\X}{\mathcal X}

\newcommand{\perf}{{\operatorname{perf}}}
\DeclareMathOperator{\uCM}{\underline{CM}}
\DeclareMathOperator{\umod}{\underline{mod}}

\DeclareMathOperator{\inc}{inc}
\DeclareMathOperator{\Ker}{Ker}
\DeclareMathOperator{\Image}{Im}

\DeclareMathOperator{\pd}{\mathsf{pd}}
\DeclareMathOperator{\id}{\mathsf{id}}
\DeclareMathOperator{\Mod}{Mod}
\DeclareMathOperator{\fmod}{mod}
\DeclareMathOperator{\End}{End}
\DeclareMathOperator{\Inj}{Inj}
\DeclareMathOperator{\Proj}{Proj}
\DeclareMathOperator{\inj}{inj}
\DeclareMathOperator{\proj}{proj}
\DeclareMathOperator{\Hom}{Hom}
\DeclareMathOperator{\Ext}{Ext}
\DeclareMathOperator{\Tor}{Tor}
\DeclareMathOperator{\Id}{Id}

\DeclareMathOperator{\rad}{rad}

\DeclareMathOperator*{\silp}{\mathsf{silp}}
\DeclareMathOperator*{\spli}{\mathsf{spli}}
\DeclareMathOperator*{\gld}{\mathsf{gl.dim}}

\newcommand{\fcomp}{\circ}
\newcommand{\iso}{\cong}
\newcommand{\equivalence}{\simeq}
\newcommand{\Union}{\bigcup}
\newcommand{\dsum}{\oplus}
\newcommand{\Dsum}{\bigoplus}
\newcommand{\tensor}{\otimes}
\newcommand{\N}{\mathbb{N}} 
\newcommand{\No}{\N_0}      
\newcommand{\degree}[1]{\left|#1\right|}
\newcommand{\opposite}[1]{#1^{\operatorname{op}}}
\newcommand{\envalg}[1]{#1^{\operatorname{e}}}
\newcommand{\idealgenby}[1]{\langle #1 \rangle}
\newcommand{\fg}[1]{\textup{\textsf{Fg#1}}}
\newcommand{\defterm}[1]{\textbf{#1}}

\newcommand{\HH}[1]{\operatorname{HH}^{#1}}
\newcommand{\Dsg}{\mathsf{D}_\mathsf{sg}}
\newcommand{\Db}{\mathsf{D}^\mathsf{b}}
\newcommand{\Kb}{\mathsf{K}^\mathsf{b}}
\newcommand{\Derived}{\mathsf{D}}
\DeclareMathOperator{\CM}{CM}
\DeclareMathOperator{\Spec}{Spec}
\DeclareMathOperator{\coh}{coh}
\newcommand{\leS}{\preceq}

\newcommand{\recollement}[9]{\xymatrix@C=5em{
{#1} \ar[r]^{#5} &
{#2} \ar[r]^{#8} \ar@/_1.5pc/[l]_{#4} \ar@/^1.5pc/[l]^{#6} &
{#3} \ar@/_1.5pc/[l]_{#7} \ar@/^1.5pc/[l]^{#9}
}}

\begin{document}

\title[Gorenstein categories, singular equivalences and cohomology rings]
{Gorenstein categories, singular equivalences and Finite generation of cohomology rings in recollements}

\author{Chrysostomos Psaroudakis}
\author{\O{}ystein Skarts\ae{}terhagen}
\author{\O{}yvind Solberg}
\address{Institutt for matematiske fag, NTNU \\ N-7491 Trondheim \\ Norway}
\curraddr{Universit\"at Stuttgart \\ Institut f\"ur Algebra und Zahlentheorie \\
  Pfaffenwaldring 57 \\
  D-70569 Stuttgart \\
  Germany}
\email{Chrysostomos.Psaroudakis@mathematik.uni-stuttgart.de}
\address{Institutt for matematiske fag, NTNU \\ N-7491 Trondheim \\ Norway}
\email{Oystein.Skartsaterhagen@math.ntnu.no}
\address{Institutt for matematiske fag, NTNU \\ N-7491 Trondheim \\ Norway}
\email{Oyvind.Solberg@math.ntnu.no}

\date{\today}
 


\begin{abstract}
Given an artin algebra $\Lambda$ with an idempotent element $a$ we
compare the algebras $\Lambda$ and $a\Lambda a$ with respect to
Gorensteinness, singularity categories and the finite generation
condition \fg{} for the Hochschild cohomology.  In particular, we
identify assumptions on the idempotent element $a$ which ensure that
$\Lambda$ is Gorenstein if and only if $a\Lambda a$ is Gorenstein,
that the singularity categories of $\Lambda$ and $a\Lambda a$ are
equivalent and that \fg{} holds for $\Lambda$ if and only if \fg{}
holds for $a\Lambda a$.  We approach the problem by using recollements
of abelian categories and we prove the results concerning
Gorensteinness and singularity categories in this general setting.
The results are applied to stable categories of Cohen--Macaulay
modules and classes of triangular matrix algebras and quotients of
path algebras.
\end{abstract}

\maketitle

\setcounter{tocdepth}{1} \tableofcontents

\section{Introduction}

This paper deals with Gorenstein algebras/categories, singularity
categories and a finiteness condition ensuring existence of a useful
theory of support for modules over finite dimensional algebras.  First
we give some background and indicate how these subjects are linked for
us.  Then we discuss the common framework for our investigations and
give a sample of the main results in the paper.  Finally we describe
the structure of the paper.  For related work see
Green--Madsen--Marcos \cite{GMM} and Nagase \cite{N}.  In
Subsection~\ref{subsection:nagase}, we compare our results to those of
Nagase.

For a group algebra of a finite group $G$ over a field $k$ there is a
theory of support varieties of modules introduced by Jon Carlson in
the seminal paper \cite{Ca}.  This theory has proven useful and
powerful, where the support of a module is defined in terms of the
maximal ideal spectrum of the group cohomology ring $H^*(G,k)$.
Crucial facts here are that the group cohomology ring is graded
commutative and noetherian, and for any finitely generated $kG$-module
$M$, the Yoneda algebra $\Ext^*_{kG}(M,M)$ is a finitely generated
module over the group cohomology ring (see \cite{E,Go,V}).  For a
finitely generated $kG$-module $M$ the support variety is defined as
the variety associated to the annihilator ideal of the action of the
group cohomology ring $H^*(G,k)$ on $\Ext^*_{kG}(M,M)$.  This
construction is based on the Hopf algebra structure of the group
algebra $kG$, and until recently a theory of support was not available
for finite dimensional algebras in general.

Snashall and Solberg \cite{SO} have extended the theory of support
varieties from group algebras to finite dimensional algebras by
replacing the group cohomology $H^*(G,k)$ with the Hochschild
cohomology ring of the algebra.  Whenever similar properties as for
group algebras are satisfied, that is, (i) the Hochschild cohomology
ring is noetherian and (ii) all Yoneda algebras $\Ext^*_\Lambda(M,M)$
for a finitely generated $\Lambda$-module $M$ are finitely generated
modules over the Hochschild cohomology ring, then many of the same
results as for group algebras of finite groups are still true when
$\Lambda$ is a selfinjective algebra \cite{EHSST}.  The above set of
conditions is referred to as \fg{} (see \cite{EHSST, Solberg:1}).

\textit{Triangulated categories of singularities} or for simplicity
\textit{singularity categories} have been introduced and studied by
Buchweitz \cite{Buchweitz:unpublished}, under the name \textit{stable
  derived categories}, and later they have been considered by Orlov
\cite{Orlov}.  For an algebraic variety $\mathbb{X}$, Orlov introduced
the \textit{singularity category} of $\mathbb{X}$, as the Verdier
quotient $\Dsg(\mathbb{X}) = \Db(\coh \mathbb{X})/ \perf(\mathbb{X})$,
where $\Db(\coh \mathbb{X})$ is the bounded derived category of
coherent sheaves on $\mathbb{X}$ and $\perf(\mathbb{X})$ is the full
subcategory consisting of perfect complexes on $\mathbb{X}$.  
The singularity category $\Dsg(\mathbb{X})$
captures many geometric properties of $\mathbb{X}$.  For instance, if
the variety $\mathbb{X}$ is smooth, then the singularity category
$\Dsg(\mathbb{X})$ is trivial but this is not true in general
\cite{Orlov}.  It should be noted that the singularity category is not
only related to the study of the singularities of a given variety
$\mathbb{X}$ but is also related to the Homological Mirror Symmetry
Conjecture due to Kontsevich \cite{Kontsevich}.  For more information
we refer to \cite{Orlov, Orlov:2, Orlov:3}.

Similarly, the singularity category over a noetherian ring $R$ is
defined \cite{Buchweitz:unpublished} to be the Verdier quotient of the
bounded derived category $\Db(\fmod{R})$ of the finitely generated
$R$-modules by the full subcategory $\perf(R)$ of perfect complexes
and is denoted by
\[
\Dsg(R)=\Db(\fmod{R})/
\perf(R).
\]
In this case the singularity category $\Dsg(R)$ can be viewed as a
categorical measure of the singularities of the spectrum $\Spec(R)$.
Moreover, by a fundamental result of Buchweitz
\cite{Buchweitz:unpublished}, and independently by Happel
\cite{Happel:3}, the singularity category of a Gorenstein ring is
equivalent to the stable category of (maximal) Cohen--Macaulay modules
$\uCM(R)$, where the latter is well known to be a triangulated
category \cite{Happel:4}.  Note that this equivalence generalizes the
well known equivalence between the singularity category of a
selfinjective algebra and the stable module category, a result due to
Rickard \cite{Rickard}.  If there exists a triangle equivalence
between the singularity categories of two rings $R$ and $S$, then such
an equivalence is called a \textit{singular equivalence} between $R$
and $S$.  Singular equivalences were introduced by Chen, who studied
singularity categories of non-Gorenstein algebras and investigated
when there is a singular equivalence between certain extensions of
rings \cite{Chen:schurfunctors, Chen:radicalsquarezero, Chen:Singular
  equivalences, Chen:Singular equivalences-trivial extensions}.

Next, from the perspective of support varieties, we describe some
links between the above topics.  Support varieties for
$\Db(\fmod\Lambda)$ using the Hochschild cohomology ring of
$\Lambda$ were considered in \cite{Solberg:1} for a finite dimensional
algebra $\Lambda$ over a field $k$, where all the perfect complexes
$\perf(\Lambda)$ were shown to have trivial support variety.  Hence the
theory of support via the Hochschild cohomology ring naturally only
says something about the Verdier quotient
$\Db(\fmod\Lambda)/\perf(\Lambda)$ -- the singularity
category.

To have an interesting theory of support, the finiteness condition
\fg{} is pivotal.  When \fg{} is satisfied for an algebra $\Lambda$,
then $\Lambda$ is Gorenstein \cite[Proposition~1.2]{EHSST}, or
equivalently, $\fmod\Lambda$ is a Gorenstein category.

As we pointed out above, when $\Lambda$ is Gorenstein, then by
Buchweitz--Happel the singularity category
$\Db(\fmod\Lambda)/\perf(\Lambda)$ is triangle equivalent to
$\uCM(\Lambda)$, the stable category of Cohen--Macaulay modules.  When
$\Lambda$ is a selfinjective algebra, then $\envalg{\Lambda}$ is
selfinjective and $\uCM(\envalg{\Lambda})=\umod\envalg{\Lambda}$ is a
tensor triangulated category with $\Lambda$ as a tensor identity.  Let
$\mathcal{B}$ be the full subcategory of $\uCM(\envalg{\Lambda})$
consisting of all bimodules which are projective as a left and as a
right $\Lambda$-module.  Then $\mathcal{B}$ is also a tensor triangulated
category with tensor identity $\Lambda$.  The strictly positive part
of the graded endomorphism ring
\[
\End^*_{\mathcal{B}}(\Lambda) = \Dsum_{i\in \mathbb{Z}}
\Hom_{\mathcal{B}}(\Lambda, \Omega^i_{\envalg{\Lambda}}(\Lambda)),
\]
of the tensor identity $\Lambda$ in $\uCM(\envalg{\Lambda})$ 
is isomorphic to the strictly positive part $\HH{\geqslant
  1}(\Lambda)$ of the Hochschild cohomology ring of $\Lambda$.  This is
the relevant part for the theory of support varieties via the
Hochschild cohomology ring.  In addition $\mathcal{B}$ is a tensor triangulated
category acting on the triangulated category $\uCM(\Lambda)$, and we
can consider a theory of support varieties for $\uCM(\Lambda)$ using
the framework described in the forthcoming paper
\cite{BKSS}.  Therefore the singularity category of the enveloping
algebra $\envalg{\Lambda}$ encodes the geometric object for
support varieties of modules and complexes over the algebra
$\Lambda$.

Next we describe the categorical framework for our work.  There has
recently been a lot of interest around recollements of abelian (and
triangulated) categories.  These are exact sequences of abelian
categories
\[
 \xymatrix{
  0 \ar[r]^{ } & \A \ar[r]^{i} & \B \ar[r]^{e} &
  \C  \ar[r]^{ } & 0  }
\] 
where both the inclusion functor $i\colon \A\lxr \B$ and and the
quotient functor $e\colon \B\lxr \C$ have left and right adjoints.
They have been introduced by Beilinson, Bernstein and Deligne
\cite{BBD} first in the context of triangulated categories in their
study of derived categories of sheaves on singular spaces.

Properties of recollements of abelian categories were studied by
Franjou and Pirashvilli in \cite{Pira}, motivated by the
MacPherson--Vilonen construction for the category of perverse sheaves
\cite{MV}, and recently homological properties of recollements of
abelian and triangulated categories have also been studied in
\cite{Psaroud}.  Recollements of abelian categories were used by
Cline, Parshall and Scott in the context of representation theory, see
\cite{CPS, Parshall}, and later Kuhn used recollements in his study of
polynomial functors, see \cite{Kuhn}.  Recently, recollements of
triangulated categories have appeared in the work of Angeleri H\"ugel,
Koenig and Liu in connection with tilting theory, homological
conjectures and stratifications of derived categories of rings, see
\cite{AKL, AKL2, AKL3, AKLY4}.  Also, Chen and Xi have investigated
recollements in relation with tilting theory \cite{Xi:1} and algebraic
K-theory \cite{Xi:2, Xi:3}.  Furthermore, Han \cite{Han} has studied
the relations between recollements of derived categories of algebras,
smoothness and Hochschild cohomology of algebras.

It should be noted that module recollements, i.e.\ recollements of
abelian categories whose terms are categories of modules, appear quite
naturally in various settings.  For instance any idempotent element $e$
in a ring $R$ induces a recollement situation between the module
categories over the rings $R/\idealgenby{e}$, $R$ and $eRe$.  In fact
recollements of module categories are now well understood since every
such recollement is equivalent, in an appropriate sense, to one
induced by an idempotent element \cite{PsaroudVitoria}.

We want to compare the \fg{} condition for Hochschild cohomology,
Gorensteinness and the singularity categories of two algebras.  Our
aim in this paper is to present a common context where we can compare
these properties for an algebra $\Lambda$ and $a\Lambda a$, where $a$
is an idempotent of $\Lambda$.  This is achieved using recollements of
abelian categories.  To summarize our main results we introduce the
following notion.  Given a functor $f\colon \B \lxr \C$ between abelian
categories, the functor $f$ is called an \defterm{eventually
  homological isomorphism} if there is an integer $t$ such that for
every pair of objects $B$ and $B'$ in $\B$, and every $j > t$, there
is an isomorphism
\[
\Ext_\B^j(B,B') \iso \Ext_\C^j(f(B),f(B')).
\]
Our main results, stated in the context of artin algebras, are
summarized in the following theorem.  The four parts of the theorem
are proved in Corollary~\ref{cor:algebra-ext-iso},
Corollary~\ref{corsingular}, Corollary~\ref{corGorensteinArtinalg} and
Theorem~\ref{thm:main-result-fg}, respectively.  More general versions
of the first three parts, in the setting of recollements of abelian
categories, are given in Corollary~\ref{cor:recollement-e-is-iso} and
Proposition~\ref{prop:ext-iso-implies-conditions},
Theorem~\ref{thmsingular} and Theorem~\ref{thm:gorenstein}.

\begin{main-thm}
Let $\Lambda$ be an artin algebra over a commutative ring $k$ and let
$a$ be an idempotent element of $\Lambda$.  Let $e$ be the functor
$a-\colon \fmod{\Lambda}\lxr \fmod{a\Lambda a}$ given by
multiplication by $a$.  Consider the following conditions$\colon$
\begin{align*}
(\alpha)
\ \id_\Lambda \Big( \frac{\Lambda/\idealgenby{a}}
                         {\rad \Lambda/\idealgenby{a}} \Big)
&< \infty &
(\beta) \ \pd_{a{\Lambda}a} a{\Lambda} &< \infty \\
(\gamma)
\ \pd_\Lambda \Big( \frac{\Lambda/\idealgenby{a}}
                         {\rad \Lambda/\idealgenby{a}} \Big)
&< \infty &
(\delta) \ \pd_{\opposite{(a{\Lambda}a)}}{\Lambda}a &< \infty
\end{align*}
Then the following hold.
\begin{enumerate}
\item The following are equivalent$\colon$
\begin{enumerate}
\item $(\alpha)$ and $(\beta)$ hold.
\item $(\gamma)$ and $(\delta)$ hold.
\item The functor $e$ is an eventually homological isomorphism.
\end{enumerate}
\item The functor $e$ induces a singular equivalence between $\Lambda$
and $a\Lambda a$ if and only if the conditions $(\beta)$ and
$(\gamma)$ hold.
\item Assume that $e$ is an eventually homological isomorphism.  Then
$\Lambda$ is Gorenstein if and only if $a{\Lambda}a$ is Gorenstein.
\item Assume that $e$ is an eventually homological isomorphism.
Assume also that $k$ is a field and that $(\Lambda/\rad \Lambda)
\tensor_k (\opposite{\Lambda}/\rad \opposite{\Lambda})$ is a
semisimple $\envalg{\Lambda}$-module (for instance, this is true if
$k$ is algebraically closed).  Then $\Lambda$ satisfies \fg{} if and
only if $a{\Lambda}a$ satisfies \fg{}.
\end{enumerate}
\end{main-thm}

Now we describe the contents of the paper section by section.  In
Section~\ref{section:prelim}, we recall notions and results on
recollements of abelian categories and Hochschild cohomology that are
used throughout the paper.

In Section~\ref{section:extensions}, we study extension groups in a
recollement of abelian categories $(\A,\B,\C)$.  More precisely, we
investigate when the exact functor $e \colon \B \lxr \C$ is an
eventually homological isomorphism.  It turns out that the answer to
this problem is closely related to the characterization given in
\cite{Psaroud} of when the functor $e$ induces isomorphisms between
extension groups in all degrees below some bound $n$.  In
Corollary~\ref{cor:recollement-e-is-iso} and
Proposition~\ref{prop:ext-iso-implies-conditions} we give sufficient
and necessary conditions, respectively, for the functor $e$ to be an
eventually homological isomorphism.  In the setting of the Main
Theorem, we characterize when the functor $e$ is an eventually
homological isomorphism in Corollary~\ref{cor:algebra-ext-iso}.  The
results of this section are used in Section~\ref{section:Gorenstein}
and Section~\ref{section:fg-a(Lambda)a}.

In Section~\ref{section:Gorenstein}, we study Gorenstein categories,
introduced by Beligiannis and Reiten \cite{BR}.  Assuming that we have
an eventually homological isomorphism $f \colon \D \lxr \F$ between
abelian categories, we investigate when Gorensteinness is transferred
between $\D$ and $\F$.  Among other things, we prove that if $f$ is an
essentially surjective eventually homological isomorphism, then $\D$
is Gorenstein if and only if $\F$ is (see
Theorem~\ref{thm:gorenstein}).  We apply this to recollements of
abelian categories and recollements of module categories.

In Section~\ref{sectionsingular}, we investigate singularity
categories, in the sense of Buchweitz \cite{Buchweitz:unpublished} and
Orlov \cite{Orlov}, in a recollement $(\A,\B,\C)$ of abelian
categories.  In fact, we give necessary and sufficient conditions for
the quotient functor $e\colon \B\lxr \C$ to induce a triangle
equivalence between the singularity categories of $\B$ and $\C$, see
Theorem~\ref{thmsingular}.  This result generalizes earlier results by
Chen \cite{Chen:schurfunctors}.  We obtain the results of Chen in
Corollary~\ref{corsingular} by applying Theorem~\ref{thmsingular} to
rings with idempotents.  Finally, for an artin algebra $\Lambda$ with
an idempotent element $a$, we give a sufficient condition for the
stable categories of Cohen--Macaulay modules of $\Lambda$ and $a\Lambda
a$ to be triangle equivalent, see Corollary~\ref{corgorsingular}.

In Section~\ref{section:cohomology} and
Section~\ref{section:fg-a(Lambda)a}, which form a unit, we investigate
the finite generation condition \fg{} for the Hochschild cohomology
of a finite dimensional algebra over a field.  In particular, in
Section~\ref{section:cohomology} we show how we can compare the \fg{}
condition for two different algebras.  This is achieved by showing,
for two graded rings and graded modules over them, that if we have
isomorphisms in all but finitely many degrees then the noetherian
property of the rings and the finite generation of the modules is
preserved, see Proposition~\ref{prop:graded-fin-gen} and
Corollary~\ref{prop:fg-two-algebras}.  In
Section~\ref{section:fg-a(Lambda)a}, we use this result to show that
\fg{} holds for a finite dimensional algebra $\Lambda$ over a field if
and only if \fg{} holds for the algebra $a\Lambda a$, where $a$ is an
idempotent element of $\Lambda$ which satisfies certain assumptions
(see Theorem~\ref{thm:main-result-fg}).

The final Section~\ref{section:examples} is devoted to applications
and examples of our main results.  First we apply our results to
triangular matrix algebras.  For a triangular matrix algebra
$\Lambda=\bigl(\begin{smallmatrix}
\Sigma & 0 \\
{_{\Gamma}M_{\Sigma}} & \Gamma
\end{smallmatrix}\bigr)$, we compare $\Lambda$ to the algebras
$\Sigma$ and $\Gamma$ with respect to the \fg{} condition,
Gorensteinness and singularity categories.  In particular, we recover
a result by Chen \cite{Chen:schurfunctors} concerning the singularity
categories of $\Lambda$ and $\Sigma$.  Then we consider some special
cases where there are relations between the assumptions of our main
results (see $(\alpha)$--$(\delta)$ in Main Theorem) and provide an
interpretation for quotients of path algebras.  Finally, we compare
our results to those of Nagase \cite{N}.

\subsection*{Conventions and Notation}
For a ring $R$ we work usually with left $R$-modules and the
corresponding category is denoted by $\Mod R$.  The full subcategory
of finitely presented $R$-modules is denoted by $\fmod R$.  Our
additive categories are assumed to have finite direct sums and our
subcategories are assumed to be closed under isomorphisms and direct
summands.  The Jacobson radical of a ring $R$ is denoted by $\rad R$.
By a module over an artin algebra $\Lambda$, we mean a finitely
presented (generated) left $\Lambda$-module.

\subsection*{Acknowledgments}
This paper was written during a postdoc period of the first author at
the Norwegian University of Science and Technology (NTNU, Trondheim)
funded by NFR Storforsk grant no.\ 167130. The first author would
like to thank his co-authors, Idun Reiten and all the members of the
Algebra group for the warm hospitality and the excellent working
conditions. The authors are grateful for the comments from Hiroshi
Nagase on a preliminary version of this paper, which led to a much
better understanding of the conditions occurring in the Main Theorem.

\section{Preliminaries}
\label{section:prelim}

In this section we recall notions and results on recollements of
abelian categories and Hochschild cohomology.

\subsection{Recollements of Abelian Categories}

In this subsection we recall the definition of a recollement situation
in the context of abelian categories, see for instance
\cite{Pira,Happel,Kuhn}, we fix notation and recall some well known
properties of recollements which are used later in the paper.  We also
include our basic source of examples of recollements.  For an additive
functor $F\colon \A \lxr \B$ between additive categories, we denote by
$\Image F = \{B \in \B \mid B \iso F(A) \ \text{for some} \ A \in
\A\}$ the \defterm{essential image} of $F$ and by $\Ker F = \{A \in \A
\mid F(A) = 0\}$ the \defterm{kernel} of $F$.

\begin{defn}
\label{defnrec}
A \defterm{recollement situation} between 
abelian categories $\A,\B$ and $\C$ is a diagram
\[
\recollement{\A}{\B}{\C}{q}{i}{p}{l}{e}{r}
\]
henceforth denoted by $(\A,\B,\C)$, satisfying the following
conditions$\colon$
\begin{enumerate}
\item[\bf 1.] $(l,e,r)$ is an adjoint triple.
\item[\bf 2.] $(q,i,p)$ is an adjoint triple.
\item[\bf 3.] The functors $i$, $l$, and $r$ are fully faithful.
\item[\bf 4.] $\Image i = \Ker e$.
\end{enumerate}
\end{defn}

In the next result we collect some basic properties of a recollement
situation of abelian categories that can be derived easily from
Definition~\ref{defnrec}.  For more details, see \cite{Pira,Psaroud}.

\begin{prop}
\label{properties}
Let $(\A,\B,\C)$ be a recollement of abelian categories.  Then the
following hold.
\begin{enumerate}
\item The functors $i\colon \A\lxr \B$ and $e\colon \B\lxr \C$ are
exact.
\item The compositions $ei$, $ql$ and $pr$ are zero.
\item The functor $e\colon \B\lxr \C$ is essentially surjective.
\item The units of the adjoint pairs $(i,p)$ and $(l,e)$ and the
counits of the adjoint pairs $(q,i)$ and $(e,r)$ are
isomorphisms$\colon$
\[
\Id_\A \xrightarrow{\iso} pi \qquad
\Id_\C \xrightarrow{\iso} el \qquad
qi \xrightarrow{\iso} \Id_\A \qquad
er \xrightarrow{\iso} \Id_\C
\]
\item The functors $l\colon \C\lxr \B$ and $q\colon \B\lxr \A$
preserve projective objects and the functors $r\colon \C\lxr \B$ and
$p\colon \B\lxr \A$ preserve injective objects.
\item The functor $i\colon \A\lxr \B$ induces an equivalence between
$\A$ and the Serre subcategory $\Ker e = \Image i$ of $\B$.  Moreover,
$\A$ is a localizing and colocalizing subcategory of $\B$ and there is
an equivalence of categories $\B/\A \equivalence \C$.
\item For every $B$ in $\B$ there are $A$ and $A'$ in $\A$ such that
the units and counits of the adjunctions induce the following exact
sequences$\colon$
\[
 \xymatrix{
  0 \ar[r]^{ } & i(A) \ar[r]^{ } & le(B) \ar[r]^{ } &
  B  \ar[r]^{ } & iq(B) \ar[r] & 0  }
\]  
and 
\[
\xymatrix{
  0 \ar[r]^{ } & ip(B) \ar[r]^{ } & B \ar[r]^{ } &
  re(B)  \ar[r] & i(A') \ar[r] & 0  }
\]  
\end{enumerate}
\end{prop}

Throughout the paper, we apply our results to recollements of module
categories, and in particular to recollements of module categories
over artin algebras as described in the following example.

\begin{exam}
\label{exam:mod-recollements}
Let $\Lambda$ be an artin $k$-algebra, where $k$ is a commutative
artin ring, and let $a$ be an idempotent element in $\Lambda$.
\begin{enumerate}
\item We have the following recollement of abelian categories$\colon$
\[
\recollement{\fmod \Lambda/\idealgenby{a}}{\fmod \Lambda}{\fmod a \Lambda a}
            {\Lambda/\idealgenby{a} \tensor_\Lambda -}
            {\inc}
            {\Hom_\Lambda(\Lambda/\idealgenby{a},-)}
            {\Lambda a \tensor_{a \Lambda a} -}
            {e=a(-)}
            {\Hom_{a \Lambda a}(a \Lambda,-)}
\]
The functor $e\colon \fmod{\Lambda}\lxr \fmod{a\Lambda a}$ can be also
described as follows$\colon$$e=a(-)\iso \Hom_{\Lambda}(\Lambda
a,-)\iso a\Lambda\tensor_{\Lambda}-$.  We write $\idealgenby{a}$ for
the ideal of $\Lambda$ generated by the idempotent element $a$.  Then
every left $\Lambda/\idealgenby{a}$-module is annihilated by
$\idealgenby{a}$ and thus the category $\fmod{\Lambda/\idealgenby{a}}$
is the kernel of the functor $a(-)$.
\item Let $\envalg{\Lambda}=\Lambda \tensor_k \opposite{\Lambda}$ be
the enveloping algebra of $\Lambda$.  The element $\varepsilon=a\tensor
\opposite{a}$ is an idempotent element of
$\envalg{\Lambda}$.  Therefore as above we have the following
recollement of abelian categories$\colon$
\[
\recollement{\fmod \envalg{\Lambda}/\idealgenby{\varepsilon}}
            {\fmod \envalg{\Lambda}}
            {\fmod \envalg{(a \Lambda a)}}
            {\envalg{\Lambda}/\idealgenby{\varepsilon} \tensor_{\envalg{\Lambda}} -}
            {\inc}
            {\Hom_{\envalg{\Lambda}}(\envalg{\Lambda}/\idealgenby{\varepsilon},-)}
            {\envalg{\Lambda}\varepsilon \tensor_{\envalg{(a \Lambda a)}} -}
            {E = \varepsilon(-)}
            {\Hom_{\envalg{(a \Lambda a)}}(\varepsilon\envalg{\Lambda},-)}
\]
Note that $\envalg{(a \Lambda a)} \iso
\varepsilon\envalg{\Lambda}\varepsilon$ as $k$-algebras.
\end{enumerate}
\end{exam}

\begin{rem}
As in Example~\ref{exam:mod-recollements}, any idempotent element $e$
in a ring $R$ induces a recollement situation between the module
categories over the rings $R/\idealgenby{e}$, $R$ and $eRe$.  This
should be considered as the universal example for recollements of
abelian categories whose terms are categories of modules.  Indeed, in
\cite{PsaroudVitoria}, it is proved that any recollement of module
categories is equivalent, in an appropriate sense, to one induced by
an idempotent element.
\end{rem}

\subsection{Hochschild cohomology rings}
\label{subsection:hh}

We briefly explain the terminology we need regarding Hochschild
cohomology and the finite generation condition \fg{}, and recall some
important results.  For a more detailed exposition of these topics,
see sections 2--5 of \cite{Solberg:1}.

Let $\Lambda$ be an artin algebra over a commutative ring $k$.  We
define the \defterm{Hochschild cohomology ring} $\HH*(\Lambda)$ of
$\Lambda$ by
\[
\HH*(\Lambda)
 = \Ext_{\envalg{\Lambda}}^*(\Lambda, \Lambda)
 = \Dsum_{i=0}^\infty \Ext_{\envalg{\Lambda}}^i(\Lambda, \Lambda).
\]
This is a graded $k$-algebra with multiplication given by Yoneda
product.  Hochschild cohomology was originally defined by Hochschild
in \cite{Hochschild}, using the bar resolution.  It was shown in
\cite[IX, \S6]{CartanEilenberg} that our definition coincides with the
original definition when $\Lambda$ is projective over $k$.

Gerstenhaber showed in \cite{Gerstenhaber:1} that the Hochschild
cohomology ring as originally defined is graded commutative.  This
implies that the Hochschild cohomology ring as defined above is graded
commutative when $\Lambda$ is projective over $k$.  The following more
general result was shown in \cite[Theorem~1.1]{SO}  (see also
\cite{Suarez}, which proves graded commutativity of several cohomology
theories in a uniform way).

\begin{thm}
\label{thm:hh-graded-commutative}
Let $\Lambda$ be an algebra over a commutative ring $k$ such that
$\Lambda$ is flat as a module over $k$.  Then the Hochschild
cohomology ring $\HH*(\Lambda)$ is graded commutative.
\end{thm}

To describe the finite generation condition \fg{}, we first need to
define a $\HH*(\Lambda)$-module structure on the direct sum of all
extension groups of a $\Lambda$-module with itself (for more details
about this module structure, see \cite{SO}).  Assume that $\Lambda$ is
flat as $k$-module, and let $M$ be a $\Lambda$-module.  The direct sum
\[
\Ext_\Lambda^*(M, M)
 = \Dsum_{i=0}^\infty \Ext_\Lambda^i(M, M)
\]
of all extension groups of $M$ with itself is a graded $k$-algebra
with multiplication given by Yoneda product.  We give it a graded
$\HH*(\Lambda)$-module structure by the graded ring homomorphism
\[
\varphi_M \colon \HH*(\Lambda) \lxr \Ext_\Lambda^*(M, M).
\]
which is defined in the following way.  Any homogeneous element of
positive degree in $\HH*(\Lambda)$ can be represented by an exact
sequence
\[
\eta\colon
0 \lxr
\Lambda \lxr
X \lxr
P_n \lxr
\cdots \lxr
P_1 \lxr
P_0 \lxr
\Lambda \lxr
0
\]
of $\envalg{\Lambda}$-modules, where every $P_i$ is projective.
Tensoring this sequence throughout with $M$ gives an exact sequence
\[
0 \lxr
\Lambda \tensor_\Lambda M \lxr
X \tensor_\Lambda M \lxr
P_n \tensor_\Lambda M \lxr
\cdots \lxr
P_1 \tensor_\Lambda M \lxr
P_0 \tensor_\Lambda M \lxr
\Lambda \tensor_\Lambda M \lxr
0
\]
of $\Lambda$-modules (the exactness of this sequence follows from the
facts that $\Lambda$ is flat as $k$-module and that the modules $P_i$
are projective $\envalg{\Lambda}$-modules).  Using the isomorphism
$\Lambda \tensor_\Lambda M \iso M$, we get an exact sequence of
$\Lambda$-modules starting and ending in $M$; we define
$\varphi_M([\eta])$ to be the element of $\Ext_\Lambda^*(M, M)$
represented by this sequence.  For elements of degree zero in
$\HH{*}(\Lambda)$, the map $\varphi_M$ is defined by tensoring with
$M$ and using the identification $\Lambda \tensor_\Lambda M \iso M$.

In \cite{EHSST}, Erdmann--Holloway--Snashall--Solberg--Taillefer
identified certain assumptions about an algebra $\Lambda$ which are
sufficient in order for the theory of support varieties to have good
properties.  They called these assumptions \textbf{Fg1} and
\textbf{Fg2}.  We say that an algebra satisfies \fg{} if it satisfies
both \textbf{Fg1} and \textbf{Fg2}.  We use the following definition
of \fg{}, which is equivalent (by \cite[Proposition~5.7]{Solberg:1})
to the definition of \textbf{Fg1} and \textbf{Fg2} given in
\cite{EHSST}.

\begin{defn}
\label{defn:fg}
Let $\Lambda$ be an algebra over a commutative ring $k$ such that
$\Lambda$ is flat as a module over $k$.  We say that $\Lambda$
satisfies the \defterm{\fg{} condition} if the following is true:
\begin{enumerate}
\item The ring $\HH*(\Lambda)$ is noetherian.
\item The $\HH*(\Lambda)$-module $\Ext_\Lambda^*(\Lambda/\rad
  \Lambda, \Lambda/\rad \Lambda)$ is finitely generated.
\end{enumerate}
\end{defn}

The following result states that in our definition of \fg{}, we could
have replaced part (ii) by the same requirement for all
$\Lambda$-modules.  It can be proved in a similar way as
\cite[Proposition~1.4]{EHSST}.

\begin{thm}
\label{fg-implies-every-ext-fin-gen}
If an artin algebra $\Lambda$ satisfies the \fg{} condition, then
$\Ext_\Lambda^*(M, M)$ is a finitely generated $\HH*(\Lambda)$-module
for every $\Lambda$-module $M$.
\end{thm}

We end this section by describing a connection between the \fg{}
condition and Gorensteinness.

\begin{thm}\cite[Theorem~1.5~(a)]{EHSST}
\label{fg-implies-gorenstein}
If an artin algebra $\Lambda$ satisfies the \fg{} condition, then
$\Lambda$ is Gorenstein.
\end{thm}

\section{Eventually homological isomorphisms in recollements}
\label{section:extensions}

Given a functor $f\colon \D \lxr \F$ between abelian categories and an
integer $t$, the functor $f$ is called an \defterm{$t$-homological
  isomorphism} if there is an isomorphism
\[
\Ext_\B^j(B,B') \iso \Ext_\C^j(f(B),f(B'))
\]
for every pair of objects $B$ and $B'$ in $\B$, and every $j > t$.  If
$f$ is a $t$-homological isomorphism for some $t$, then it is an
\defterm{eventually homological isomorphism}.
In this section, we investigate when the functor $e$ in a recollement
\[
\recollement{\A}{\B}{\C}{q}{i}{p}{l}{e}{r}
\]
of abelian categories is an eventually homological isomorphism.

The functor $e$ induces maps
\begin{equation}
\label{eqn:ext-map}
\Ext^j_\B(X,Y) \ \lxr \ \Ext^j_\C(e(X),e(Y))
\end{equation}
of extension groups for all objects $X$ and $Y$ in $\B$ and for every
$j \ge 0$.  With one argument fixed and the other one varying over all
objects we study when these maps are isomorphisms in almost all
degrees, that is, for every degree $j$ greater than some bound $n$
(see Theorem~\ref{thm:ext-iso-proj} and
Theorem~\ref{thm:ext-iso-inj}).  We use this to find two sets of
sufficient conditions for the functor $e\colon \B \lxr \C$ to be an
eventually homological isomorphism
(Corollary~\ref{cor:recollement-e-is-iso}), and we find a partial
converse (Proposition~\ref{prop:ext-iso-implies-conditions}).
Finally, we specialize these results to artin algebras, using the
recollement $(\fmod \Lambda/\idealgenby{a}, \fmod \Lambda, \fmod
a{\Lambda}a)$ of Example~\ref{exam:mod-recollements}~(i).  In
particular, we characterize when the functor $e\colon \fmod \Lambda
\lxr \fmod a{\Lambda}a$ is an eventually homological isomorphism
(Corollary~\ref{cor:algebra-ext-iso}).

These results are used in Section~\ref{section:Gorenstein} for
comparing Gorensteinness of the categories in a recollement, and in
Section~\ref{section:fg-a(Lambda)a} for comparing the \fg{} condition
of the algebras $\Lambda$ and $a{\Lambda}a$, where $a$ is an
idempotent in $\Lambda$.

We start by fixing some notation. For an injective coresolution $0
\lxr B \lxr I^0 \lxr I^1{\lxr}\cdots $ of $B$ in $\B$, we say that the
image of the morphism $I^{n-1}\lxr I^n$ is an \defterm{\mbox{$n$-th}
  cosyzygy} of $B$, and we denote it by $\Sigma^n(B)$.  Dually, if
$\cdots \lxr P_1 \lxr P_0 \lxr B\lxr 0$ is a projective resolution of
$B$ in $\B$, then we say that the kernel of the morphism $P^{n-1}\lxr
P^{n-2}$ is an \defterm{\mbox{$n$-th} syzygy} of $B$, and we denote it
by $\Omega^n(B)$.  Also, if $\X$ is a class of objects in $\B$, then
we denote by $\X^{\bot}=\{B\in \B \mid \Hom_{\B}(X,B)=0\}$ the
\defterm{right orthogonal subcategory} of $\X$ and by
${^{\bot}\X}=\{B\in \B \mid \Hom_{\B}(B,X)=0\}$ the \defterm{left
  orthogonal subcategory} of $\X$.

We now describe precisely how the maps \eqref{eqn:ext-map} induced by
the functor $e$ in a recollement are defined.  Let $\D$ and $\F$ be
abelian categories and $f\colon \D \lxr \F$ an exact functor which has
a left and a right adjoint (for example, the functors $i$ and $e$ in a
recollement have these properties).  If
\[
\xi\colon \xymatrix{
0 \ar[r]^{} &
X_n \ar[r]^{d_n \ \ } &
X_{n-1} \ar[r]^{} &
\cdots \ar[r]^{} &
X_{1} \ar[r]^{d_1 \ } &
X_0 \ar[r]^{} &
0
}
\]
is an exact sequence in $\D$, then we denote by $f(\xi)$ the exact
sequence
\[
f(\xi)\colon \xymatrix{
0 \ar[r]^{} &
f(X_n) \ar[r]^{f(d_n) \ \ \ } &
f(X_{n-1}) \ar[r]^{} &
\cdots \ar[r]^{} &
f(X_{1}) \ar[r]^{f(d_1) \ } &
f(X_0) \ar[r]^{} &
0
}
\]
in $\F$.  It is clear that this operation commutes with Yoneda product;
that is, if $\xi$ and $\zeta$ are composable exact sequences in $\D$,
then $f(\xi\zeta) = f(\xi) \cdot f(\zeta)$.  For every pair of objects
$X$ and $Y$ in $\D$ and every nonnegative integer $j$, we define a
group homomorphism
\[
f_{X,Y}^j \ \colon \ \Ext_{\D}^j(X,Y) \ \lxr \ \Ext_{\F}^j(f(X),f(Y))
\]
by
\begin{align*}
f_{X,Y}^0(d) &= f(d) &&\text{for a morphism $d \colon X \lxr Y$;} \\
f_{X,Y}^j([\eta]) &= [f(\eta)]
 &&\text{for a $j$-fold extension $\eta$ of $X$ by $Y$, where $j>0$.}
\end{align*}

For an object $X$ in $\D$, the direct sum $\Ext^*_\D(X,X) =
\Dsum_{j=0}^\infty \Ext^j_\D(X,X)$ is a graded ring with
multiplication given by Yoneda product, and taking the maps
$f^j_{X,X}$ in all degrees $j$ gives a graded ring homomorphism
\[
f^*_{X,X}\colon \Ext^*_\D(X,X) \lxr \Ext^*_\F(f(X),f(X)).
\]

\begin{rem}
We explain briefly why the maps $f_{X,Y}^j$ and $f^*_{X,X}$ defined
above are homomorphisms.
\begin{enumerate}
\item The functor $f$ being a right and left adjoint implies that it
preserves limits and colimits and therefore it preserves pullbacks and
pushouts.  Thus the map $f_{X,Y}^j$ preserves the Baer sum between two
extensions.
\item For checking that the map $f_{X,X}^*$ is a graded ring
homomorphism, the only nontrivial case to consider is the product of a
morphism and an extension.  For this case, we again use that the
functor $f$ preserves pullbacks and pushouts.
\end{enumerate}
\end{rem}

We now consider the maps
\[
e_{B,B'}^j \colon \Ext_\B^j(B, B') \lxr \Ext_\C^j(e(B), e(B'))
\]
induced by the functor $e \colon \B \lxr \C$ in a recollement, where we
let one argument be fixed and the other vary over all objects of $\B$.
In \cite{Psaroud}, the first author studied when these maps are
isomorphisms for all degrees up to some bound $n$, that is, for $0 \le
j \le n$.  This immediately leads to a description of when these maps
are isomorphisms in all degrees, which we state as the following
theorem.

\begin{thm}\cite[Propositions~3.3 and 3.4, Theorem~3.10]{Psaroud}
\label{extthm}
Let $(\A,\B,\C)$ be a recollement of abelian categories and assume
that $\B$ and $\C$ have enough projective and injective objects.  Let
$B$ be an object in $\B$.
\begin{enumerate}
\item The following statements are equivalent$\colon$
\begin{enumerate}
\item The map $e_{B,B'}^{j} \colon \Ext_{\B}^{j}(B,B') \lxr
\Ext_{\C}^{j}(e(B),e(B'))$ is an isomorphism for every object $B'$ in
$\B$ and every nonnegative integer $j$.
\item The object $B$ has a projective resolution of the form
\[
\xymatrix{
\cdots \ar[r]^{} &
l(P_2) \ar[r]^{} &
l(P_1) \ar[r]^{} &
l(P_0) \ar[r]^{} &
B \ar[r]^{} &
0
}
\]
where $P_j$ is a projective object in $\C$.
\item $\Ext_{\B}^j(B,i(A))=0$ for every $A\in \A$ and $j\geq 0$.
\item $\Ext_{\B}^j(B,i(I))=0$ for every $I\in \Inj{\A}$ and $j\geq 0$.
\end{enumerate}
\item The following statements are equivalent$\colon$
\begin{enumerate}
\item The map $e_{B',B}^{j} \colon \Ext_{\B}^{j}(B',B) \lxr
\Ext_{\C}^{j}(e(B'),e(B))$ is an isomorphism for every object $B'$ in
$\B$ and every nonnegative integer $j$.
\item The object $B$ has an injective coresolution of the form
\[
\xymatrix{
0 \ar[r]^{} &
B \ar[r]^{} &
r(I^0) \ar[r]^{} &
r(I^1) \ar[r]^{} &
r(I^2) \ar[r]^{} &
\cdots
}
\]
where $I^j$ is an injective object in $\C$.
\item $\Ext_{\B}^j(i(A),B)=0$ for every $A\in \A$ and $j\geq 0$.
\item $\Ext_{\B}^j(i(P),B)=0$ for every $P\in \Proj{\A}$ and $j\geq 0$.
\end{enumerate}
\end{enumerate}
\end{thm}

The above theorem describes when the maps $e_{B,B'}^j$ induced by the
functor $e$ are isomorphisms in all degrees $j$.  Our aim in this
section is to give a similar description of when these maps are
isomorphisms in almost all degrees.  The basic idea is to translate
the conditions in the above theorem to similar conditions stated for
almost all degrees, and show the equivalence of these conditions by
using the above theorem and dimension shifting.  In order for this to
work, however, we need to modify the conditions somewhat.  We obtain
Theorem~\ref{thm:ext-iso-proj} which is stated below and generalizes
parts of Theorem~\ref{extthm}~(i) (and the dual
Theorem~\ref{thm:ext-iso-inj} which generalizes parts of
Theorem~\ref{extthm}~(ii)).  In order to prove the theorem, we need a
general version of dimension shifting as stated in the following
lemma.

\begin{lem}
\label{lem:dimension-shift}
Let $\A$ be an abelian category, $n$ be an integer, and let
\[
\epsilon\colon \xymatrix{
0 \ar[r]^{} &
X \ar[r]^{} &
E_{m-1} \ar[r]^{} &
\cdots \ar[r]^{} &
E_{0} \ar[r]^{} &
Y \ar[r]^{} &
0
}
\]
be an exact sequence in $\A$ with $\pd_{\A}E_i \le n$ for every $i$.
Then for every $i > n$ and $Z\in \A$, the map
\[
\xymatrix@C=0.5cm{
\epsilon^* \ \colon \ {\Ext}_{\A}^{i}(X,Z) \ \ar[rr] &&
\ \Ext^{i+m}_{\A}(Y,Z),}
\]
given by $\epsilon^*([\eta]) = [\eta\epsilon]$, is an isomorphism.
\end{lem}

Now we are ready to show our characterization of when the functor $e$
in a recollement induces isomorphisms of extension groups in almost
all degrees.

\begin{thm}
\label{thm:ext-iso-proj}
Let $(\A,\B,\C)$
be a recollement of abelian categories and assume that $\B$ and $\C$
have enough projective and injective objects.  Consider the following
statements for an object $B$ of $\B$ and two integers $n$ and $m$$\colon$
\begin{enumerate}
\item[(a)] The map $e_{B,B'}^{j} \colon \Ext_{\B}^{j}(B,B') \lxr
\Ext_{\C}^{j}(e(B),e(B'))$ is an isomorphism for every object $B'$ in
$\B$ and every integer $j > m + n$.
\item[(b)] The object $B$ has a projective resolution of the form
\[
\xymatrix{
\cdots \ar[r] &
l(Q_1) \ar[r] &
l(Q_0) \ar[r] &
P_{n-1} \ar[r] &
\cdots \ar[r] &
P_0 \ar[r] &
B \ar[r] &
0
} 
\]
where each $Q_j$ is a projective object in $\C$.
\item[(c)] $\Ext_{\B}^j(B,i(A))=0$ for every $A \in \A$ and $j > n$,
and there exists an $n$-th syzygy of $B$ lying in ${^{\bot}i(\A)}$.
\item[(d)] $\Ext_{\B}^j(B,i(I))=0$ for every $I\in \Inj{\A}$ and $j >
n$, and and there exists an $n$-th syzygy of $B$ lying in
${^{\bot}i(\Inj{\A})}$.
\end{enumerate}
We have the following relations between these statements$\colon$
\begin{enumerate}
\item $\textup{(b)} \iff \textup{(c)} \iff \textup{(d)}$.
\item If $\pd_{\C} e(P) \le m$ for every projective object $P$ in
$\B$, then $\textup{(b)} \implies \textup{(a)}$.
\end{enumerate}
\end{thm}
\begin{proof}
(i) By dimension shift, statement (c) is equivalent to
\[
\Ext_\B^j(\Omega^n(B), i(A)) = 0
\qquad\text{for every $j \ge 0$ and every $A \in \A$,}
\]
and statement (d) is equivalent to
\[
\Ext_\B^j(\Omega^n(B), i(I)) = 0
\qquad\text{for every $j \ge 0$ and every $I \in \Inj \A$,}
\]
where in both cases $\Omega^n(B)$ is a suitably chosen $n$-th syzygy
of $B$.  The equivalence of statements (b), (c) and (d) now follows
from the equivalence of (b), (c) and (d) in Theorem~\ref{extthm}~(i).

(ii) Let 
\[
\pi\colon \xymatrix{
0 \ar[r] & K \ar[r] & P_{n-1} \ar[r] & \cdots \ar[r] &
P_1 \ar[r] & P_0 \ar[r] & B \ar[r] & 0 }
\]
be the beginning of the chosen projective resolution of $B$, where $K
= \Omega^n(B)$ is the $n$-th syzygy of $B$.  Consider the following
group homomorphisms$\colon$
\begin{equation}
\label{eqn:ext-maps}
\Ext_\B^j(B,B') \xleftarrow{\pi^*}
\Ext_\B^{j-n}(K,B') \xrightarrow{e_{K,B'}^{j-n}}
\Ext_\C^{j-n}(e(K),e(B')) \xrightarrow{(e(\pi))^*}
\Ext_\C^j(e(B),e(B'))
\end{equation}
Here, the maps $\pi^*$ and $(e(\pi))^*$ are isomorphisms by
Lemma~\ref{lem:dimension-shift}.  Note that for $(e(\pi))^*$ we use the
fact that $\pd_{\C} e(P) \le n$ for every projective object $P$ in
$\B$.  The map $e_{K,B'}^{j-n}$ is an isomorphism by
Theorem~\ref{extthm}~(i).  Thus, we have an isomorphism
\[
(e(\pi))^* \fcomp e_{K,B'}^{j-n} \fcomp (\pi^*)^{-1} \colon
\Ext_\B^j(B,B') \lxr \Ext_\C^j(e(B),e(B'))
\]
for every $j\geq m+n+1$ and $B'\in \B$.  We want to show that this is
the same map as $e_{B,B'}^j$.  We consider an element $[\eta] \in
\Ext_\B^{j-n}(K,B')$, and follow it through the
maps~\eqref{eqn:ext-maps}.  We then get the following elements$\colon$
\[
\xymatrix@R=1em@C=3em{
{\Ext_\B^j(B,B')} &
{\Ext_\B^{j-n}(K,B')} \ar[l]_{\pi^*}^\iso
                        \ar[r]^-{e_{K,B'}^{j-n}}_-\iso &
{\Ext_\C^{j-n}(e(K),e(B'))} \ar[r]^-{(e(\pi))^*}_-\iso &
{\Ext_\C^j(e(B),e(B'))} \\
{[\eta\pi]} &
{[\eta]} \ar@{|->}[l] \ar@{|->}[r] &
{[e(\eta)]} \ar@{|->}[r] &
{[e(\eta) \cdot e(\pi)]} \ar@{=}[d] \\
&&& {[e(\eta\pi)]}
}
\]
This shows that our isomorphism takes any element $[\zeta] \in
\Ext_\B^j(B,B')$ to the element $[e(\zeta)] \in
\Ext_\C^j(e(B),e(B'))$.  Thus, our isomorphism is $e_{B,B'}^j$.
\end{proof}

Dually to the above theorem, we have the following generalization of
some of the implications in Theorem~\ref{extthm}~(ii).

\begin{thm}
\label{thm:ext-iso-inj}
Let $(\A,\B,\C)$
be a recollement of abelian categories and assume that $\B$ and $\C$
have enough projective and injective objects.  Consider the following
statements for an object $B$ of $\B$ and two integers $n$ and $m$$\colon$
\begin{enumerate}
\item[(a)] The map $e_{B',B}^{j} \colon \Ext_{\B}^{j}(B',B) \lxr
\Ext_{\C}^{j}(e(B'),e(B))$ is an isomorphism for every object $B'$ in
$\B$ and every integer $j > m + n$.
\item[(b)] The object $B$ has an injective coresolution of the form
\[
\xymatrix{
0 \ar[r] &
B \ar[r] &
I^0 \ar[r] &
\cdots \ar[r] &
I^{n-1} \ar[r] &
r(J^0) \ar[r] &
r(J^1) \ar[r] &
\cdots
} 
\]
where each $J^j$ is a projective object in $\C$.
\item[(c)] $\Ext_{\B}^j(i(A),B)=0$ for every $A \in \A$ and $j > n$,
and there exists an $n$-th cosyzygy of $B$ lying in $i(\A)^{\bot}$.
\item[(d)] $\Ext_{\B}^j(i(P),B)=0$ for every $P\in \Proj{\A}$ and $j >
n$, and there exists an $n$-th cosyzygy of $B$ lying in
$i(\Proj{\A})^{\bot}$.
\end{enumerate}
We have the following relations between these statements$\colon$
\begin{enumerate}
\item $\textup{(b)} \iff \textup{(c)} \iff \textup{(d)}$.
\item If $\id_{\C} e(I) \le m$ for every injective object $I$ in
$\B$, then $\textup{(b)} \implies \textup{(a)}$.
\end{enumerate}
\end{thm}

In the above results, we fixed an object $B$ of the category $\B$, and
considered the maps $e_{B,B'}^j$ or $e_{B',B}^j$ for all objects $B'$
in $\B$.  With certain conditions on the object $B$, we found that
these maps are isomorphisms for almost all degrees $j$.  We now
describe some conditions on the recollement which are sufficient to
ensure that the maps $e_{B,B'}^j$ are isomorphisms in almost all
degrees $j$ for all objects $B$ and $B'$ of $\B$, in other words, that
the functor $e$ is an eventually homological isomorphism.  These
conditions are given in the following corollary, which follows
directly from Theorem~\ref{thm:ext-iso-proj} and
Theorem~\ref{thm:ext-iso-inj}.

\begin{cor}
\label{cor:recollement-e-is-iso}
Let $(\A,\B,\C)$ be a recollement and assume that $\B$ and $\C$ have
enough projective and injective objects.  Let $m$ and $n$ be two
integers.  Assume that one of the following conditions hold$\colon$
\begin{enumerate}
\item $(\alpha')$ $\sup\{\id_{\B}i(I) \mid I\in \Inj{\A} \}<m$.

\noindent $(\epsilon)$ Every object of $\B$ has an $m$-th syzygy which
lies in ${^{\bot}i(\Inj{\A})}$.

\noindent $(\beta)$ $\sup\{\pd_{\C}e(P) \mid P\in \Proj{\B} \}\leq n$.

\item $(\gamma')$ $\sup\{\pd_{\B}i(P) \mid P\in \Proj{\A} \}<n$.

\noindent $(\opposite{\epsilon})$ Every object of $\B$ has an $n$-th
cosyzygy which lies in $i(\Proj{\A})^{\bot}$.

\noindent $(\delta)$ $\sup\{\id_{\C}e(I) \mid I\in \Inj{\B} \}\leq m$.
\end{enumerate}
Then the functor $e$ is an $(m+n)$-homological isomorphism, and in
particular the map
\[
\xymatrix@C=0.5cm{
e^j_{B,B'} \ \colon \ {\Ext}_{\B}^j(B,B') \ \ar[rr]^{ \iso } &&
\ \Ext^j_{\C}(e(B),e(B'))}
\]
is an isomorphism for all objects $B$ and $B'$ of $\B$ and for every
$j>m+n$.
\end{cor}

We now show a partial converse of the above result.

\begin{prop}
\label{prop:ext-iso-implies-conditions}
Let $(\A,\B,\C)$ be a recollement and assume that $\B$ and $\C$ have
enough projective and injective objects.  Assume that the functor $e$
is an eventually homological isomorphism.  Then the following hold$\colon$
\begin{enumerate}
\item[$(\alpha)$] $\sup \{ \id_\B i(A) \mid A \in \A \} < \infty$.
\item[$(\beta)$]  $\sup \{ \pd_\C e(P) \mid P \in \Proj \B \} < \infty$.
\item[$(\gamma)$] $\sup \{ \pd_\B i(A) \mid A \in \A \} < \infty$.
\item[$(\delta)$] $\sup \{ \id_\C e(I) \mid I \in \Inj \B \} < \infty$.
\end{enumerate}
In particular, if $e$ is an $t$-homological isomorphism for a
nonnegative integer $t$, then each of the above dimensions is bounded
by $t$.
\end{prop}
\begin{proof}
$(\alpha)$ Let $A$ be an object of $\A$.  For every $B$ in $\B$ and $j
> t$, we get
\[
\Ext_\B^j(B, i(A))
\iso \Ext_\C^j(e(B), ei(A))
\iso \Ext_\C^j(e(B), 0)
= 0,
\]
since $ei = 0$ by Proposition~\ref{properties}, and thus $\id_\B i(A)
\le t$.  The proof of $(\gamma)$ is similar.

$(\beta)$ Let $P$ be a projective object of $\B$.  For every $C$ in
$\C$ and $j > t$, we get
\[
\Ext_\C^j(e(P), C)
\iso \Ext_\C^j(e(P), el(C))
\iso \Ext_\B^j(P, l(C))
= 0,
\]
since $el \iso \Id_\C$ by Proposition~\ref{properties}, and thus
$\pd_\C e(P) \le t$.  The proof of $(\delta)$ is similar.
\end{proof}

\begin{rem}
\label{rem:relative-gldim}
Recall from \cite{Psaroud} that $\sup \{ \pd_\B i(A) \mid A \in \A \}
< \infty$, which appears in statement $(\gamma)$ above, is called the
$\A$-relative global dimension of $\B$, and denoted by $\gld_\A \B$.
\end{rem}

We close this section by interpreting Theorem~\ref{thm:ext-iso-proj},
Theorem~\ref{thm:ext-iso-inj} and
Corollary~\ref{cor:recollement-e-is-iso} for artin algebras.  To this
end, for an artin algebra $\Lambda$ and $a\in \Lambda$ an idempotent
element, we denote by
\[
e = (a\Lambda \tensor_\Lambda -) \colon
   \fmod \Lambda \lxr  \fmod a \Lambda a
\]
the quotient functor of the recollement
$(\fmod{\Lambda/\idealgenby{a}}, \fmod{\Lambda}, \fmod{a \Lambda a})$,
see Example~\ref{exam:mod-recollements}.

We first need the following well-known observation.

\begin{lem}
\label{lemHom}
Let $\Lambda$ be an artin algebra, $M$ be a $\Lambda$-module and $S$
be a simple $\Lambda$-module.  Then for every $n\geq 1$ we have$\colon$
\[
\Ext_{\Lambda}^n(M,S)\iso \Hom_{\Lambda}(\Omega^n(M),S)
\qquad\text{and}\qquad
\Ext_{\Lambda}^n(S,M)\iso \Hom_{\Lambda}(S,\Sigma^n(M)),
\]
where $\Omega^n(M)$ is the $n$-th syzygy in a minimal projective
resolution of $M$, and $\Sigma^n(M)$ is the $n$-th cosyzygy in a
minimal injective coresolution of $M$.
\end{lem}

We also need the next easy result whose proof is left to the reader.

\begin{lem}
\label{lem:algebra-dimension-inequalities}
Let $\Lambda$ be an artin algebra and $a$ an idempotent element of
$\Lambda$.
 
Then the following inequalities hold$\colon$ 
\begin{enumerate}
\item $\pd_{a \Lambda a} e(P) \le \pd_{a \Lambda a} a\Lambda$, for
every $P\in \proj{\Lambda}$.
\item $\id_{a \Lambda a} e(I) \le \pd_{\opposite{(a \Lambda a)}}
\Lambda a$, for every $I\in \inj{\Lambda}$.
\end{enumerate}
\end{lem}

The following is a consequence of Theorem~\ref{thm:ext-iso-proj} and
Theorem~\ref{thm:ext-iso-inj} for artin algebras.

\begin{cor}
\label{cor:algebra-ext-iso-onemodule}
Let $\Lambda$ be an artin algebra and $a$ an idempotent element in
$\Lambda$, and let $m$ and $n$ be integers.
\begin{enumerate}
\item Let $M$ be a $\Lambda$-module such that
$\Ext^j_{\Lambda}\big(M,(\Lambda/\idealgenby{a})/(\rad{\Lambda/
  \idealgenby{a}})\big)=0$ for every $j\geq m$.  Assume that $\pd_{a
  \Lambda a} a \Lambda \le n$.  Then the map
\[
\xymatrix@C=0.5cm{
e^j_{M,N} \ \colon \ {\Ext}_{\Lambda}^j(M,N) \ \ar[rr]^{ \iso \ } &&
\ \Ext^j_{a\Lambda a}(e(M),e(N)) }
\]
is an isomorphism for every $\Lambda$-module $N$, and for every
integer $j > m + n$.
\item Let $M$ be a $\Lambda$-module such that
$\Ext^j_{\Lambda}\big((\Lambda/\idealgenby{a})/(\rad{\Lambda/
  \idealgenby{a}}),M\big)=0$ for every $j\geq n$.  Assume that
$\pd_{\opposite{(a \Lambda a)}} \Lambda a \le m$.  Then the map
\[
\xymatrix@C=0.5cm{
e^j_{N,M} \ \colon \ {\Ext}_{\Lambda}^j(N,M) \ \ar[rr]^{ \iso \ } &&
\ \Ext^j_{a\Lambda a}(e(N),e(M)) }
\]
is an isomorphism for every $\Lambda$-module $N$, and for every
integer $j > m + n$.
\end{enumerate}
\end{cor}
\begin{proof}
(i) Consider the recollement
$(\fmod{\Lambda/\idealgenby{a}},\fmod{\Lambda},\fmod{a\Lambda a})$ of
Example~\ref{exam:mod-recollements}.  Since every simple $\Lambda/
\idealgenby{a}$-module is also simple as a $\Lambda$-module it follows
from Lemma~\ref{lemHom} that
\[
\Hom_{\Lambda}
   \big(\Omega^m(M),
        (\Lambda/\idealgenby{a})/(\rad{\Lambda/\idealgenby{a}})
   \big)
 = 0
\]
This implies that $\Hom_{\Lambda}(\Omega^m(M),N)=0$ for every
$\Lambda/\idealgenby{a}$-module $N$ since every module has a finite
composition series.  Then the result is a consequence of
Theorem~\ref{thm:ext-iso-proj}.

(ii) The result follows similarly as in (i), using
Theorem~\ref{thm:ext-iso-inj} and the second isomorphism of
Lemma~\ref{lemHom}.
\end{proof}

As an immediate consequence of the above results we have the following
characterization of when the functor $e \colon \fmod \Lambda \lxr
\fmod a{\Lambda}a$ is an eventually homological isomorphism.  This
constitutes the first part of the Main Theorem presented in the
introduction.

\begin{cor}
\label{cor:algebra-ext-iso}
Let $\Lambda$ be an artin algebra and $a$ an idempotent element in
$\Lambda$.  The following are equivalent:
\begin{enumerate}
\item There is an integer $s$ such that for every pair of
$\Lambda$-modules $M$ and $N$, and every $j > s$, the map
\[
\xymatrix@C=0.5cm{
e^j_{M,N} \ \colon \ {\Ext}_{\Lambda}^j(M,N) \ \ar[rr]^{ \iso \ } &&
\ \Ext^j_{a\Lambda a}(e(M),e(N)) }
\]
is an isomorphism.
\item The functor $e$ is an eventually homological isomorphism.
\item $(\alpha)$
$\id_{\Lambda}\big((\Lambda/\idealgenby{a})/(\rad{\Lambda/
  \idealgenby{a}})\big) < \infty$ and $(\beta)$ $\pd_{a \Lambda a} a
\Lambda < \infty$.
\item $(\gamma)$
$\pd_{\Lambda}\big((\Lambda/\idealgenby{a})/(\rad{\Lambda/
  \idealgenby{a}})\big) < \infty$ and $(\delta)$ $\pd_{\opposite{(a
    \Lambda a)}} \Lambda a < \infty$.
\end{enumerate}
In particular, if the functor $e$ is a $t$-homological isomorphism,
then each of the dimensions in \textup{(iii)} and \textup{(iv)} are at
most $t$.  The bound $s$ in \textup{(i)} is bounded by the sum of
the dimensions occurring in \textup{(iii)}, and also bounded by the
sum of the dimensions occurring in \textup{(iv)}.
\end{cor}
\begin{proof}
The implications (ii) $\implies$ (iii) and (ii) $\implies$ (iv) follow
from Proposition~\ref{prop:ext-iso-implies-conditions}.  The
implications (iii) $\implies$ (i) and (iv) $\implies$ (i) follow from
Corollary~\ref{cor:algebra-ext-iso-onemodule}.
\end{proof}

\section{Gorenstein categories and eventually homological isomorphisms}
\label{section:Gorenstein}

Our aim in this section is to study Gorenstein categories, introduced
by Beligiannis--Reiten \cite{BR}.  The main objective is to study when
a functor $f\colon \D \lxr \F$ between abelian categories preserves
Gorensteinness.  A central property here is whether the functor $f$ is
an eventually homological isomorphism.  We prove that for an
essentially surjective eventually homological isomorphism $f \colon \D
\lxr \F$, then $\D$ is Gorenstein if and only if $\F$ is.  The results
are applied to recollements of abelian categories, and recollements of
module categories.

We start by briefly recalling the notion of Gorenstein categories
introduced in \cite{BR}.
Let $\A$ be an abelian category with enough projective and injective
objects.  We consider the following invariants associated to $\A\colon$
\[
\spli{\A} = \sup\{\pd_{\A}I \mid I\in \Inj{\A} \}
\ \ \text{and} \ \
\silp{\A} = \sup\{\id_{\A}P \mid P\in \Proj{\A} \}
\]
Then we have the following notion of Gorensteinness for abelian
categories.

\begin{defn}\cite{BR}
An abelian category $\A$ with enough projective and injective objects
is called \defterm{Gorenstein} if $\spli{\A}<\infty$ and
$\silp{\A}<\infty$.
\end{defn}

Note that the above notion is a common generalization of
Gorensteinness in the commutative and in the noncommutative setting.
We refer to \cite[Chapter VII]{BR} for a thorough discussion on
Gorenstein categories and connections with Cohen--Macaulay objects and
cotorsion pairs.

We start with the following useful observation whose direct proof is
left to the reader.

\begin{lem}
\label{lemsilpspli}
Let $\A$ be an abelian category with enough projective and injective
objects and let $X$ be an object of $\A$.
\begin{enumerate}
\item If $\pd_{\A}X<\infty$, then $\id_{\A}X\le \silp{\A}$.
\item If $\id_{\A}X<\infty$, then $\pd_{\A}X\le \spli{\A}$.
\end{enumerate}
\end{lem}

In the main result of this section we study eventually homological
isomorphisms between abelian categories with enough projective and
injective objects.  In particular we show that an essentially
surjective eventually homological isomorphism preserves
Gorensteinness.  This is a general version of the third part of the
Main Theorem presented in the introduction.

\begin{thm}
\label{thm:gorenstein}
Let $f \colon \D \lxr \F$ be a functor, where $\D$ and $\F$ are
abelian categories with enough projective and injective objects, and
let $t$ be a nonnegative integer.  Consider the following four
statements.
\begin{alignat*}{2}
\text{\textup{(a)} For every $D$ in $\D$\textup{:}}
&\left\{
\begin{aligned}
\pd_\D D &\le \sup \{ \pd_\F f(D), t \} \\
\id_\D D &\le \sup \{ \id_\F f(D), t \}
\end{aligned}
\right.
\qquad\qquad
&
\textup{(c)}
&\left\{
\begin{aligned}
\spli \D &\le \sup \{ \spli \F, t \} \\
\silp \D &\le \sup \{ \silp \F, t \}
\end{aligned}
\right.
\\
\text{\textup{(b)} For every $D$ in $\D$\textup{:}}
&\left\{
\begin{aligned}
\pd_\F f(D) &\le \sup \{ \pd_\D D, t \} \\
\id_\F f(D) &\le \sup \{ \id_\D D, t \}
\end{aligned}
\right.
&
\textup{(d)}
&\left\{
\begin{aligned}
\spli \F &\le \sup \{ \spli \D, t \} \\
\silp \F &\le \sup \{ \silp \D, t \}
\end{aligned}
\right.
\end{alignat*}
We have the following.
\begin{enumerate}
\item If $f$ is a $t$-homological isomorphism, then \textup{(a)}
holds.
\item If $f$ is an essentially surjective $t$-homological isomorphism,
then \textup{(a)} and \textup{(b)} hold.
\item If \textup{(a)} and \textup{(b)} hold, then \textup{(c)} holds.
\item If \textup{(a)} and \textup{(b)} hold and $f$ is essentially
surjective, then \textup{(c)} and \textup{(d)} hold.
\end{enumerate}
In particular, we obtain the following.
\begin{enumerate}
\setcounter{enumi}{4}
\item If $f$ is an essentially surjective eventually homological
isomorphism, then $\D$ is Gorenstein if and only if $\F$ is
Gorenstein.
\item If $f$ is an eventually homological isomorphism and \textup{(b)}
holds, then $\F$ being Gorenstein implies that $\D$ is Gorenstein.
\end{enumerate}
\end{thm}
\begin{proof}
We first assume that $f$ is an essentially surjective $t$-homological
isomorphism and show the inequality $\pd_\F f(D) \le \sup \{ \pd_\D D,
t \}$; the other inequalities in parts (i) and (ii) are proved
similarly.  The inequality clearly holds if $D$ has infinite
projective dimension.  Assume that $D$ has finite projective
dimension, and let $n = \max \{ \pd_\D D, t \} + 1$.  For any object
$X$ in $\F$, there is an object $X'$ in $\D$ with $f(X') \iso X$,
since the functor $f$ is essentially surjective.  By using that $f$ is
a $t$-homological isomorphism, we get
\[
\Ext_\F^n(f(D), X)
\iso \Ext_\F^n(f(D), f(X'))
\iso \Ext_\D^n(D, X')
= 0.
\]
This means that we have $\pd_\F f(D) < n$, and therefore $\pd_\F f(D)
\le \sup \{ \pd_\D D, t \}$.

We now assume that (a) and (b) hold and $f$ is essentially surjective,
and show the inequality $\spli \F \le \sup \{ \spli \D, t \}$; the
other inequalities in parts (iii) and (iv) are proved similarly.  Let
$I$ be an injective object of $\F$.  Since $f$ is essentially
surjective, we can choose an object $D$ in $\D$ such that $f(D) \iso
I$.  By (a), the object $D$ has finite injective dimension, and then
by Lemma~\ref{lemsilpspli}, its projective dimension is at most $\spli
\D$.  Using (b), we get
\[
\pd_\F I
 \le \sup \{ \pd_\D D, t \}
 \le \sup \{ \spli \D, t \}.
\]
Since this holds for any injective object $I$ in $\F$, we have
$\spli \F \le \sup \{ \spli \D, t \}$.

Parts (v) and (vi) follow by combining parts (i)--(iv).
\end{proof}

Now we return to the setting of a recollement $(\A, \B, \C)$.  We use
Theorem~\ref{thm:gorenstein} to study the functors $i\colon \A \lxr
\B$ and $e\colon \B \lxr \C$ with respect to Gorensteinness.

\begin{cor}
\label{cor:gorenstein-recollement}
Let $(\A,\B,\C)$ be a recollement of abelian categories.
\begin{enumerate}
\item Assume that the categories $\B$ and $\C$ have enough projective
and injective objects, and that the functor $e$ is an eventually
homological isomorphism.  Then $\B$ is Gorenstein if and only if $\C$
is Gorenstein.
\item Assume that the category $\B$ has enough projective and
injective objects, and that we have either
\par\noindent\begin{minipage}{\linewidth}
\[
\left.
\begin{aligned}
\sup\{\pd_{\B}i(P) \mid P\in \Proj{\A}\} &\le 1 \\
\sup\{\id_{\B}i(I) \mid I\in \Inj{\A}\} &< \infty
\end{aligned}
\right\}
\qquad\text{or}\qquad
\left\{
\begin{aligned}
\sup\{\pd_{\B}i(P) \mid P\in \Proj{\A}\} &< \infty \\
\sup\{\id_{\B}i(I) \mid I\in \Inj{\A}\} &\le 1
\end{aligned}
\right.
\]
\end{minipage}
If $\B$ is Gorenstein, then $\A$ is Gorenstein.
\end{enumerate}
\end{cor}
\begin{proof}
Part (i) follows directly from Theorem~\ref{thm:gorenstein}~(v),
noting that $e$ is essentially surjective by
Proposition~\ref{properties}.

We now show part (ii).  By Proposition~\ref{properties}~(iv) and (v),
$\A$ has enough projective and injective objects since $\B$ does (see
\cite[Remark~2.5]{Psaroud}).

It follows from \cite[Proposition~4.15]{Psaroud} (or its
dual) that the functor $i\colon \A\lxr \B$ is a homological embedding,
i.e.\ the map $i_{X,Y}^n$ is an isomorphism for all objects $X$ and
$Y$ in $\A$ and every $n \ge 0$.  In particular, this means that $i$
is a $0$-homological isomorphism.  By
Theorem~\ref{thm:gorenstein}~(i), we have
\begin{equation}
\label{eqn:gorenstein-recollement-inequalities}
\pd_\A A \le \pd_\B i(A)
\qquad\text{and}\qquad
\id_\A A \le \id_\B i(A)
\end{equation}
for every object $A$ in $\A$.

We show that $\spli \A \le \spli \B$.  Let $I$ be an injective object
in $A$.  By assumption, we have $\id_\B i(I) < \infty$, and then by the
first inequality in \eqref{eqn:gorenstein-recollement-inequalities}
and Lemma~\ref{lemsilpspli}, we have
\[
\pd_\A I
\le \pd_\B i(I)
\le \spli \B.
\]
Hence we have $\spli \A \le \spli \B$.  By a similar argument, we have
$\silp \A \le \silp \B$.  The result follows.
\end{proof}

In a recollement $(\A, \B, \C)$ we have seen that the implications
$\B$ Gorenstein if and only if $\C$ Gorenstein and $\B$ Gorenstein
implies $\C$ Gorenstein hold under various additional assumptions.  It
is then natural to ask if the categories $\A$ and $\C$ being
Gorenstein could imply that $\B$ is Gorenstein.  The next example
shows that this is not true in general.

\begin{exam}
Let $k$ be a field and consider the algebra $k[x]/\idealgenby{x^2}$.
Then from the triangular matrix algebra
\[
\Lambda = \begin{pmatrix}
           k & k \\
           0 & k[x]/\idealgenby{x^2} \\
         \end{pmatrix}
\]
we have the recollement of module categories
$(\fmod{k[x]/\idealgenby{x^2}}, \fmod{\Lambda}, \fmod{k})$, where
$\fmod{k[x]/\idealgenby{x^2}}$ and $\fmod{k}$ are Gorenstein
categories but $\fmod{\Lambda}$ is not Gorenstein.  We refer to
\cite[Example~4.3~(2)]{Chen:schurfunctors} for more details about the
algebra $\Lambda$.
\end{exam}

Recall from \cite{BR} that a ring $R$ is called \defterm{left
  Gorenstein} if the category $\Mod{R}$ of left $R$-modules is a
Gorenstein category.  Applying
Corollary~\ref{cor:gorenstein-recollement} to the module recollement
$(\Mod{R/\idealgenby{e}},\Mod{R},\Mod{eRe})$ from
Example~\ref{exam:mod-recollements}, we have the following result.

\begin{cor} 
\label{cor:gorenstein-ring}
Let $R$ be a ring and $e$ an idempotent element of $R$.
\begin{enumerate}
\item If the functor $e-\colon \Mod R \lxr \Mod eRe$ is an eventually
homological isomorphism, then the ring $R$ is left Gorenstein if and
only if the ring $eRe$ is left Gorenstein.
\item Assume that we have either
\par\noindent\begin{minipage}{\linewidth}
\[
\left.
\begin{aligned}
\pd_R R/\idealgenby{e} &\le 1 \\
\sup\{\id_{\B}i(I) \mid I\in \Inj{R/\idealgenby{e}} \} &< \infty
\end{aligned}
\right\}
\qquad\text{or}\qquad
\left\{
\begin{aligned}
\pd_R R/\idealgenby{e} &< \infty \\
\sup\{\id_{\B}i(I) \mid I\in \Inj{R/\idealgenby{e}} \} &\le 1
\end{aligned}
\right.
\]
\end{minipage}
If the ring $R$ is left Gorenstein then the ring $R/\idealgenby{e}$ is
left Gorenstein.
\end{enumerate}
\end{cor}

Recall that an artin algebra $\Lambda$ is called \defterm{Gorenstein}
if $\id{_{\Lambda}\Lambda}<\infty$ and $\id\Lambda_{\Lambda}<\infty$
(see \cite{AR:applications, AR:cm}).  Note that $\fmod{\Lambda}$ is a
Gorenstein category if and only if $\Lambda$ is a Gorenstein algebra.
We close this section with the following consequence for artin
algebras, whose first part constitutes the third part of the Main
Theorem presented in the introduction.

\begin{cor}
\label{corGorensteinArtinalg}
Let $\Lambda$ be an artin algebra and $a$ an idempotent element of
$\Lambda$.
\begin{enumerate}
\item Assume that the functor $a-\colon \fmod \Lambda \lxr \fmod
a{\Lambda}a$ is an eventually homological isomorphism.  Then the
algebra $\Lambda$ is Gorenstein if and only if the algebra $a\Lambda
a$ is Gorenstein.
\item Assume that we have either
\[
\left.
\begin{aligned}
\pd_\Lambda \Lambda/\idealgenby{a} &\le 1 \\
\pd_{\opposite{\Lambda}} \Lambda/\idealgenby{a} &< \infty
\end{aligned}
\right\}
\qquad\text{or}\qquad
\left\{
\begin{aligned}
\pd_\Lambda \Lambda/\idealgenby{a} &< \infty \\
\pd_{\opposite{\Lambda}} \Lambda/\idealgenby{a} &\le 1
\end{aligned}
\right.
\]
If the algebra $\Lambda$ is Gorenstein, then the algebra
$\Lambda/\idealgenby{a}$ is Gorenstein.
\end{enumerate}
\end{cor}

\section{Singular equivalences in recollements}
\label{sectionsingular}

Our purpose in this section is to study singularity categories, in the
sense of Buchweitz \cite{Buchweitz:unpublished} and Orlov
\cite{Orlov}, in a recollement of abelian categories $(\A,\B,\C)$.  In
particular we are interested in finding necessary and sufficient
conditions such that the singularity categories of $\B$ and $\C$ are
triangle equivalent.  We start by recalling some well known facts about
singularity categories.

Let $\B$ be an abelian category with enough projective objects.  We
denote by $\Db(\B)$ the derived category of
bounded complexes of objects of $\B$ and by
$\Kb(\Proj{\B})$ the homotopy category of bounded
complexes of projective objects of $\B$.  Then the \defterm{singularity
  category} of $\B$ (\cite{Buchweitz:unpublished, Orlov}) is defined
to be the Verdier quotient$\colon$
\[
\Dsg(\B)
 = \Db(\B) / \Kb(\Proj{\B})
\]
See \cite{Chen:relativesing} for a discussion of more general
quotients of $\Db(\B)$ by
$\Kb(\X)$, where $\X$ is a selforthogonal
subcategory of $\B$.

It is well known that the singularity category $\Dsg(\B)$ carries a
unique triangulated structure such that the quotient functor
$Q_{\B}\colon \Db(\B)\lxr \Dsg(\B)$ is triangulated, see
\cite{Krause:Localization, Neeman:book, Verdier}.  Recall that the
objects of the singularity category $\Dsg(\B)$ are the objects of the
bounded derived category $\Db(\B)$, the morphisms between two objects
$X^{\bullet}\lxr Y^{\bullet}$ are equivalence classes of fractions
$(X^{\bullet}\leftarrow L^{\bullet}\rightarrow Y^{\bullet})$ such that
the cone of the morphism $L^{\bullet}\lxr X^{\bullet}$ belongs to
$\Kb(\Proj{\B})$ and the exact triangles in $\Dsg(\B)$ are all the
triangles which are isomorphic to images of exact triangles of
$\Db(\B)$ via the quotient functor $Q_{\B}$.  Note that a complex
$X^{\bullet}$ is zero in $\Dsg(\B)$ if and only if $X^{\bullet}\in
\Kb(\Proj{\B})$.  Following Chen \cite{Chen:Singular equivalences,
  Chen:Singular equivalences-trivial extensions}, we say that two
abelian categories $\A$ and $\B$ are \defterm{singularly equivalent}
if there is a triangle equivalence between the singularity categories
$\Dsg(\A)$ and $\Dsg(\B)$.  This triangle equivalence is called a
\defterm{singular equivalence} between $\A$ and $\B$.

To proceed further we need the following well known result for exact
triangles in derived categories.  For a complex $X^{\bullet}$ in an
abelian category $\A$ we denote by $\sigma_{>n}(X^{\bullet})$ the
truncation complex \ $\cdots\lxr 0 \lxr \Image{d^n}\lxr
X^{n+1}\stackrel{d^{n+1}}{\lxr} X^{n+2}\lxr \cdots $, and by
$H^n(X^{\bullet})$ the n-th homology of $X^{\bullet}$.

\begin{lem}
\label{lemtriangles}
Let $\A$ be an abelian category and $X^{\bullet}$ be a complex in
$\A$.  Then we have the following triangle in $\Derived(\A)\colon$
\[
\xymatrix{
H^n(X^{\bullet})[-n] \ar[r]^{} &
\sigma_{>n-1}(X^{\bullet}) \ar[r]^{} &
\sigma_{>n}(X^{\bullet}) \ar[r]^{} &
H^n(X^{\bullet})[1-n]
}
\]
\end{lem}

Now we are ready to prove the main result of this section which gives
necessary and sufficient conditions for the quotient functor $e\colon
\B\lxr \C$ to induce a triangle equivalence between the singularity
categories of $\B$ and $\C$.  This is a general version of the second
part of the Main Theorem presented in the introduction.

\begin{thm}
\label{thmsingular}
Let $(\A,\B,\C)$ be a recollement of abelian categories.  Then the
following statements are equivalent$\colon$
\begin{enumerate}
\item We have $\pd_{\B}i(A)<\infty$ and $\pd_{\C} e(P)< \infty$ for
every $A\in \A$ and $P\in \Proj{\B}$.
\item The functor $e\colon \B\lxr \C$ induces a singular equivalence
between $\B$ and $\C\colon$
\[
\xymatrix@C=0.5cm{
 \Dsg(e)\colon
 \Dsg(\B)  \ \ar[rr]^{  \ \ \ \equivalence} &&
 \ \Dsg(\C)
}
\]
\end{enumerate}
\begin{proof}
(i) $\Rightarrow$ (ii) First note that we have a well defined derived
functor $\Db(e)\colon \Db(\B)\lxr \Db(\C)$ since the quotient functor
$e\colon \B\lxr \C$ is exact.  Also the recollement situation
$(\A,\B,\C)$ implies that $0 \lxr \A \lxr \B \lxr \C \lxr 0$ is an
exact sequence of abelian categories, see Proposition~\ref{properties}.
Then it follows from \cite[Theorem~3.2]{Miyachi},
see also \cite{Keller:cyclic}, that $0 \lxr \Db_{\A}(\B) \lxr \Db(\B)
\lxr \Db(\C) \lxr 0$ is an exact sequence of triangulated categories,
where $\Db_{\A}(\B)$ is the full subcategory of $\Db(\B)$ consisting
of complexes whose homology lie in $\A$.  Hence $\Db(e)$ is a quotient
functor, i.e.\ $\Db(\B)/\Db_{\A}(\B)\equivalence \Db(\C)$.  Next we claim
that $\Db(e)(\Kb(\Proj{\B})) \subseteq \Kb(\Proj{\C})$.  Let
$P^{\bullet}\in \Kb(\Proj{\B})$.  Suppose first that $P^{\bullet}$ is
concentrated in degree zero, so we deal with a projective object $P$
of $\B$.  Since the object $e(P)$ has finite projective dimension it
follows that there is a quasi-isomorphism $Q^{\bullet}\lxr e(P)[0]$
where $Q^{\bullet}\in \Kb(\Proj{\C})$ is a projective resolution of
$e(P)$.  Then the object $e(P)$[0] is isomorphic with $Q^{\bullet}$ in
$\Db(\C)$ and therefore $e(P)\in \Kb(\Proj{\C})$.  Now let
$P^{\bullet}=(0\lxr P_0\lxr P_1\lxr 0)\in \Kb(\Proj{\B})$.  Then we
have the triangle $P_0[0]\lxr P_1[0]\lxr P^{\bullet}\lxr P_0[1]$ and
if we apply the functor $\Db(e)$ we infer that $\Db(e)(P^{\bullet})\in
\Kb(\Proj{\C})$ since $\Kb(\Proj{\C})$ is a triangulated subcategory.
Continuing inductively on the length of the complex $P^{\bullet}$ we
infer that the object $\Db(e)(P^{\bullet})$ lies in $\Kb(\Proj{\C})$
and so our claim follows.  Then since the triangulated functor
$\Db(e)\circ Q_{\C}\colon \Db(\B)\lxr \Dsg(\C)$ annihilates
$\Kb(\Proj{\B})$ it follows that $\Db(e)\circ Q_{\C}$ factors uniquely
through $Q_{\B}$ via a triangulated functor $\Dsg(e)\colon
\Dsg(\B)\lxr \Dsg(\C)$, that is the following diagram is
commutative$\colon$
\[
\xymatrix{
 \Db(\B) \ar[rr]^{Q_{\B}} \ar[d]_{\Db(e)} &&
 \Dsg(\B)  \ar[d]^{\Dsg(e)} \\
 \Db(\C)   \ar[rr]^{Q_{\C}}     &&
 \Dsg(\C)
}
\]

Next we show that $\Db_{\A}(\B)\subseteq \Kb(\Proj{\B})$ in $\Db(\B)$.
Since the projective dimension of $i(A)$ is finite for all $A$ in
$\A$, it follows that $i(\A)\subseteq \Kb(\Proj{\B})$ in $\Db(\B)$.
Let $B^{\bullet}$ be an object of $\Db_{\A}(\B)$.  Assume first that
$B^{\bullet}$ is concentrated in degree zero.  Hence we deal with an
object $B\in \B$ such that $B\iso i(A)$ for some $A\in \A$, and
therefore our claim follows.  Now consider a complex
\[
B^{\bullet}\colon\xymatrix{
0 \ar[r]^{} & B^0 \ar[r]^{d^0} & B^1 \ar[r]^{} & 0  } 
\]
where $H^0(B^{\bullet})$ and $H^1(B^{\bullet})$ lies in $\A$.  From
Lemma~\ref{lemtriangles} we have the triangles
\[
\xymatrix{
 H^0(B^{\bullet}) \ar[r]^{} &
 \sigma_{>-1}(B^{\bullet}) \ar[r]^{} &
 \sigma_{>0}(B^{\bullet}) \ar[r]^{} &
 H^0(B^{\bullet})[1]
}
\]
and
\[
\xymatrix{
 H^1(B^{\bullet})[-1] \ar[r]^{} &
 \sigma_{>0}(B^{\bullet}) \ar[r]^{} &
 \sigma_{>1}(B^{\bullet}) \ar[r]^{} &
 H^1(B^{\bullet})
}
\]
in $\Db(\B)$.  Then from the second triangle it follows that
$\sigma_{>0}(B^{\bullet})\in \Kb(\Proj{\B})$ and therefore from the
first triangle we get that
$\sigma_{>-1}(B^{\bullet})=B^{\bullet}\in\Kb(\Proj{\B})$.  Continuing
inductively on the length of the complex $B^{\bullet}$, we infer that
$\Db_{\A}(\B)\subseteq \Kb(\Proj{\B})$ in $\Db(\B)$.  Using this we
can form the quotient $\Kb(\Proj{\B})/\Db_{\A}(\B)$, and then we have
the following exact commutative diagram$\colon$
\[
\xymatrix{
 0 \ar[r] &
 \Kb(\Proj{\B})/\Db_{\A}(\B) \ar[r]^{} \ar[d]_{} &
 \Db(\B)/\Db_{\A}(\B) \ar[d]^{\equivalence} \ar[r] &
 \Dsg(\B) \ar[r]^{} \ar[d]^{\Dsg(e)} &
 0 \\
 0 \ar[r] &
 \Kb(\Proj{\C}) \ar[r]^{} &
 \Db(\C)  \ar[r] &
 \Dsg(\C)  \ar[r] &
 0 
}
\]

We show that the functor $\Kb(\Proj{\B})/\Db_{\A}(\B)\lxr
\Kb(\Proj{\C})$ is an equivalence, where we denote it by $\Kb(e)$.
First from the above commutative diagram and since there is an
equivalence $\Db(\B)/\Db_{\A}(\B)\equivalence \Db(\C)$, it follows that the
functor $\Kb(e)$ is fully faithful.  Let $P^{\bullet}\colon 0 \lxr P_n
\lxr \cdots \lxr P_1 \lxr P_0 \lxr 0 $ be an object of
$\Kb(\Proj{\C})$.  Each $P_i$ is a projective object in $\C$ and from
Proposition~\ref{properties} we have $el(P_i)\iso P_i$ with $l(P_i)\in
\Proj{\B}$.  Then the complex $l(P^{\bullet})\colon 0 \lxr l(P_n) \lxr
\cdots \lxr l(P_1) \lxr l(P_0) \lxr 0 $ is such that
$\Kb(e)(l(P^{\bullet}))=P^{\bullet}$.  This implies that the functor
$\Kb(e)$ is essentially surjective.  Hence the functor $\Kb(e)$ is an
equivalence.

In conclusion, from the above exact commutative diagram we infer that
the singularity categories of $\B$ and $\C$ are triangle equivalent.

(ii) $\Rightarrow$ (i) Suppose that there is a triangle equivalence
$\Dsg(e)\colon \Dsg(\B)\stackrel{\equivalence}{\lxr} \Dsg(\C)$.  Let $P$ be
a projective object of $\B$.  Then $P[0]\in \Kb(\Proj{\B})$ and
$\Db(e)(P[0])\in \Kb(\Proj{\C})$.  Thus the object $e(P)$ has finite
projective dimension.  Let $A\in \A$ and consider the object $i(A)$ of
$\B$.  Then from Proposition~\ref{properties} we have $e(i(A))=0$.
Since $\Dsg(e)$ is an equivalence, the object $i(A)$ is zero in
$\Dsg(\B)$, and therefore $i(A)\in \Kb(\Proj{\B})$.  We infer that
$i(A)$ has finite projective dimension.
\end{proof}
\end{thm}

\begin{rem}
\label{rem:singular}
If the functor $e \colon \B \lxr \C$ is an eventually homological
isomorphism, then statement (i) in Theorem~\ref{thmsingular} is true
by Proposition~\ref{prop:ext-iso-implies-conditions}.  Thus
Theorem~\ref{thmsingular} in particular says that if the functor $e
\colon \B \lxr \C$ in a recollement $(\A,\B,\C)$ is an eventually
homological isomorphism, then it induces a singular equivalence
between $\B$ and $\C$.

Note that statement (i) in Theorem~\ref{thmsingular} only states that
each object of the form $i(A)$ or $e(P)$ has finite projective
dimension, and not that there exists a finite bound for the projective
dimensions of all such objects.  In other words, the supremums
\[
\sup \{ \pd_\B i(A) \mid A \in \A \}
\qquad\text{and}\qquad
\sup \{ \pd_\C e(P) \mid P \in \Proj \B \}
\]
(which are used in other parts of the paper) may be infinite even if
statement (i) is true.
\end{rem}

Applying Theorem~\ref{thmsingular} to the recollement of module
categories $(\fmod{R/\idealgenby{e}},\fmod{R},\fmod{eRe})$, see
Example~\ref{exam:mod-recollements}, we have the following consequence
due to Chen, see \cite[Theorem~2.1]{Chen:schurfunctors} and
\cite[Corollary~3.3]{Chen:tworesultsofOrlov}.  Note that our version
is somewhat stronger; the difference is that Chen takes $\pd_{eRe} eR
<\infty$ as an assumption instead of including it in one of the
equivalent statements.  This result constitutes the second part of the
Main Theorem presented in the introduction.

\begin{cor}
\label{corsingular}
Let $R$ be a left Noetherian ring and $e$ an idempotent element of
$R$.  Then the following statements are equivalent$\colon$
\begin{enumerate} 
\item For every $R/\idealgenby{e}$-module $X$ we have $\pd_RX<\infty$, and
$\pd_{eRe} eR< \infty$.
\item The functor $e(-)\colon \fmod{R}\lxr \fmod{eRe}$ induces a
singular equivalence between $\fmod{R}$ and $\fmod{eRe}\colon$
\[
\xymatrix@C=0.5cm{
 \Dsg(e(-))\colon
 \Dsg(\fmod{R})  \ \ar[rr]^{  \ \ \ \ \equivalence} &&
 \ \Dsg(\fmod{eRe})
}
\]
\end{enumerate}
\end{cor}

We end this section with an application to stable categories of
Cohen--Macaulay modules.  Let $\Lambda$ be a Gorenstein artin
algebra.  We denote by $\CM(\Lambda)$ the category of (maximal)
\defterm{Cohen--Macaulay modules} defined as follows$\colon$
\[
\CM(\Lambda)
= \{ X\in \fmod{\Lambda} \mid
     \Ext_{\Lambda}^n(X,\Lambda)=0 \ \text{for all $n\geq 1$} \}
\]
Then it is known that the stable category $\uCM(\Lambda)$ modulo
projectives is a triangulated category, see \cite{Happel:4}, and
moreover there is a triangle equivalence between the singularity
category $\Dsg(\fmod{\Lambda})$ and the stable category
$\uCM(\Lambda)$, see \cite[Theorem~4.4.1]{Buchweitz:unpublished} and
\cite[Theorem~4.6]{Happel:3}.  As a consequence of
Corollary~\ref{cor:algebra-ext-iso},
Corollary~\ref{corGorensteinArtinalg} and Corollary~\ref{corsingular}
we get the following.

\begin{cor}
\label{corgorsingular}
Let $\Lambda$ be Gorenstein artin algebra and $a$ an idempotent
element of $\Lambda$.  Assume that the functor $a-\colon \fmod \Lambda
\lxr \fmod a{\Lambda}a$ is an eventually homological isomorphism.
Then there is a triangle equivalence between the stable categories of
Cohen--Macaulay modules of $\Lambda$ and $a\Lambda a$$\colon$
\[
\xymatrix@C=0.5cm{
 \uCM(\Lambda)  \ \ar[rr]^{ \equivalence \  \ } &&
 \ \uCM(a\Lambda a)  }
\]
\end{cor}

\section{Finite generation of cohomology rings}
\label{section:cohomology}

In this section, we describe a way to compare the \fg{} condition (see
Definition~\ref{defn:fg}) for two different algebras.  This is used in
the next section for the algebras $\Lambda$ and $a{\Lambda}a$, where
$\Lambda$ is a finite dimensional algebra over a field and $a$ is an
idempotent in $\Lambda$.

Let $\Lambda$ and $\Gamma$ be two artin algebras over a commutative
ring $k$, and assume that they are flat as $k$-modules.  Let $M =
\Lambda/(\rad \Lambda)$ and $N = \Gamma/(\rad \Gamma)$.  Assume that
we have graded ring isomorphisms $f$ and $g$ making the
diagram
\begin{equation}
\label{eqn:fg-diagram}
\xymatrix{
{\HH*(\Lambda)} \ar[r]^-{\varphi_M} \ar[d]_{f}^{\iso} &
{\Ext_\Lambda^*(M, M)}  \ar[d]_g^{\iso} \\
{\HH*(\Gamma)} \ar[r]_-{\varphi_{N}} &
{\Ext_{\Gamma}^*(N, N)}
}
\end{equation}
commute, where the maps $\varphi_M$ and $\varphi_N$ are defined in
Subsection~\ref{subsection:hh}.  Then it is clear that \fg{} for
$\Lambda$ is exactly the same as \fg{} for $\Gamma$, since all the
relevant data for the \fg{} condition is exactly the same for the two
algebras.

However, we can come to the same conclusion even if the homology groups for
$\Lambda$ and $\Gamma$ are different in some degrees, as long as they
are the same in all but finitely many degrees.  In other words, if the
maps $f$ and $g$ above are just graded ring homomorphisms such that
$f_n$ and $g_n$ are group isomorphisms for almost all degrees $n$,
then the \fg{} condition holds for $\Lambda$ if and only if it holds
for $\Gamma$.  The goal of this section is to show this.

We first prove the result in a more general setting, where we replace
the rings in \eqref{eqn:fg-diagram} by arbitrary graded rings
satisfying appropriate assumptions.  This is done in
Proposition~\ref{prop:graded-fin-gen}, after we have shown a part of
the result (corresponding to part (i) of the \fg{} condition) in
Proposition~\ref{prop:graded-noetherian}.  Finally, we state the
result for \fg{} in Proposition~\ref{prop:fg-two-algebras}.

We now introduce some terminology and notation which is used in this
section and the next.  By \defterm{graded ring} we always mean a ring
of the form
\[
R = \Dsum_{i=0}^\infty R_i
\]
graded over the nonnegative integers.  We denote the set of
nonnegative integers by $\No$.  If $R$ is a graded ring and $n$ a
nonnegative integer, we use the notation $R_{\ge n}$ for the graded
ideal
\[
R_{\ge n} = \Dsum_{i=n}^\infty R_i
\]
in $R$.  We use the term \defterm{rng} for a ``ring without identity'',
that is, an object which satisfies all the axioms for a ring except
having a multiplicative identity element.

We use the following characterization of noetherianness for graded
rings.

\begin{thm}
\label{thm:noetherian-equals-graded-noetherian}
Let $R$ be a graded ring.  Then $R$ is noetherian if and only if it
satisfies the ascending chain condition on graded ideals.
\end{thm}
\begin{proof}
This follows directly from \cite[Theorem~5.4.7]{gradedrings:book}.
\end{proof}

We now begin the main work of this section by showing that an
isomorphism in all but finitely many degrees between two sufficiently
nice graded rings preserves noetherianness.  This implies that such a
map between Hochschild cohomology rings preserves part (i) of the
\fg{} condition, and thus gives one half of the result we want.

\begin{prop}
\label{prop:graded-noetherian}
Let $R$ and $S$ be graded rings.  Assume that $R_0$ and $S_0$ are
noetherian, that every $R_i$ is finitely generated as left and as
right $R_0$-module, and that every $S_i$ is finitely generated as left
and as right $S_0$-module.  Let $n$ be a nonnegative integer, and
assume that there exists an isomorphism $\phi \colon R_{\ge n} \lxr S_{\ge n}$
of graded rngs.  Then $R$ is noetherian if and only if $S$ is noetherian.
\end{prop}
\begin{proof}
We prove (by showing the contrapositive) that $R$ is left noetherian
if $S$ is left noetherian.  The corresponding result with right
noetherian is proved in the same way.  This gives one of the
implications we need.  The opposite implication is proved in the same
way by interchanging $R$ and $S$ and using $\phi^{-1}$ instead of
$\phi$.

Assume that $R$ is not left noetherian.  Let
\[
\mathbb{I}\colon\quad I^{(0)} \subset I^{(1)} \subset \cdots
\]
be an infinite strictly ascending sequence of graded left ideals in
$R$ (this is possible by
Theorem~\ref{thm:noetherian-equals-graded-noetherian}).  For every
index $i$ in this sequence, we can write the ideal $I^{(i)}$ as a
direct sum
\[
I^{(i)} = \Dsum_{d \in \No} I^{(i)}_d
\]
of abelian groups, where $I^{(i)}_d \subseteq R_d$ is the degree $d$
part of $I^{(i)}$.  For any degree $d$, we can make an ascending
sequence
\[
I^{(0)}_d \subseteq I^{(1)}_d \subseteq \cdots
\]
of $R_0$-submodules of $R_d$ by taking the degree $d$ part of each
ideal in $\mathbb{I}$.  But $R_d$ is a noetherian $R_0$-module (since
$R_0$ is noetherian and $R_d$ is a finitely generated $R_0$-module),
and hence this sequence must stabilize at some point.  Let $s(d)$ be
the point where it stabilizes, that is, the smallest integer such that
$I^{(s(d))}_d = I^{(i)}_d$ for every $i > s(d)$.

We now define two functions $\sigma\colon \No \lxr \No$ and $
\delta\colon \No \lxr \No$.  For $d \in \No$, we define
\[
\sigma(d) = \max \{ s(0), s(1), \ldots, s(d) \}.
\]
For $i \in \No$, we define $\delta(i)$ as the smallest number such
that
\[
I^{(i)}_{\delta(i)} \ne I^{(i+1)}_{\delta(i)}.
\]
These functions have the following interpretation.  For a degree $d$,
the number $\sigma(d)$ is the index in the sequence $\mathbb{I}$ where
the ideals in the sequence have stabilized up to degree $d$.  For an
index $i$, the number $\delta(i)$ is the lowest degree at which there
is a difference from the ideal $I^{(i)}$ to the ideal $I^{(i+1)}$.

We now define a sequence $(i_j)_{j \in \No}$ of indices and a sequence
$(d_j)_{j \in \No}$ of degrees by
\begin{align*}
i_j &=
\left\{
\begin{array}{ll}
\sigma(n) &\text{if $j = 0$,} \\
\sigma(d_{j-1} + n) &\text{otherwise.}
\end{array}
\right. \\
d_j &= \delta(i_j)
\end{align*}
We observe that for every positive integer $j$, we have
\[
i_j > i_{j-1}
\qquad\text{and}\qquad
d_j > d_{j-1} + n.
\]

We now construct a sequence $\mathbb{J}$ of graded left ideals in $S$.
For every nonnegative integer $j$, we choose an element
\[
x_j \in I^{(i_j+1)}_{d_j} - I^{(i_j)}_{d_j}
\]
(this is possible because $d_j = \delta(i_j)$).  Note that the degree
of $x_j$ is $d_j$, which is greater than $n$.  We then define
$J^{(j)}$ to be the left ideal of $S$ generated by the set
\[
\{ \phi(x_0), \ldots, \phi(x_j) \}.
\]
We let $\mathbb{J}$ be the sequence of these ideals$\colon$
\[
\mathbb{J}\colon\quad J^{(0)} \subseteq J^{(1)} \subseteq \cdots.
\]
We want to show that each inclusion here is strict.  This means that
we must show, for every positive integer $j$, that $\phi(x_j)$ is not
an element of $J^{(j-1)}$.

We show this by contradiction.  Assume that there is a $j$ such that
$\phi(x_j) \in J^{(j-1)}$.  Then we can write $\phi(x_j)$ as a sum
\[
\phi(x_j) = \sum_{m=0}^{j-1} s_m \cdot \phi(x_m),
\]
where each $s_m$ is an element of $S$.  Since $\phi(x_j)$ and every
$\phi(x_m)$ are homogeneous elements, we can choose every $s_m$ to be
homogeneous.  For each $m$, we have that if $s_m$ is nonzero, then its
degree is
\[
\degree{s_m}
 = \degree{\phi(x_j)} - \degree{\phi(x_m)}
 = \degree{x_j} - \degree{x_m}
 = d_j - d_m
 > n.
\]
Thus $s_m$ is either zero or in the image of $\phi$.  We use this
to find corresponding elements in $R$.  Let, for each $m \in
\{1,\ldots,j-1\}$,
\[
r_m = \left\{
\begin{array}{ll}
0              &\text{if $s_m = 0$,} \\
\phi^{-1}(s_m) &\text{otherwise.}
\end{array}
\right.
\]
Now we have
\[
\phi(x_j)
 = \sum_{m=0}^{j-1} s_m \cdot \phi(x_m)
 = \phi\left( \sum_{m=0}^{j-1} r_m\cdot x_m \right).
\]
Applying $\phi^{-1}$ gives
\[
x_j = \sum_{m=0}^{j-1} r_m\cdot x_m.
\]
Since we have $x_m \in I^{(i_m+1)} \subseteq I^{(i_j)}$ for every $m$,
this means that $x_j \in I^{(i_j)}$.  This is a contradiction, since
$x_j$ is chosen so that it does not lie in $I^{(i_j)}$.

We have shown that the sequence $\mathbb{J}$ is a strictly ascending
sequence of graded left ideals in $S$.  Thus $S$ in not left
noetherian.
\end{proof}

We now complete the picture by considering two graded rings and a
graded module over each ring, and showing that isomorphisms in all but
finitely many degrees preserve both noetherianness of the rings and
finite generation of the modules (given that certain assumptions are
satisfied).

\begin{prop}
\label{prop:graded-fin-gen}
Let $R$ and $M$ be graded rings, and $\theta \colon R \lxr M$ a graded ring
homomorphism.  View $M$ as a graded left $R$-module with scalar
multiplication given by $\theta$.  Assume that $R_0$ is noetherian, that
every $R_i$ is finitely generated as left and as right $R_0$-module,
and that every $M_i$ is finitely generated as left $R_0$-module.

Similarly, let $R'$ and $M'$ be graded rings, and $\theta' \colon R' \lxr
M'$ a graded ring homomorphism.  View $M'$ as a graded left
$R'$-module with scalar multiplication given by $\theta'$.  Assume that
$R'_0$ is noetherian, that every $R'_i$ is finitely generated as left
and as right $R'_0$-module, and that every $M'_i$ is finitely
generated as left $R'_0$-module.

Assume that there are graded rng isomorphisms $\phi \colon R_{\ge n}
\lxr R'_{\ge n}$ and $\psi \colon M_{\ge n} \lxr M'_{\ge n}$ (for some
nonnegative integer $n$) such that the diagram
\[
\xymatrix@C=5em{
R_{\ge n}  \ar[r]^{\theta_{\ge n}} \ar[d]_{\phi} & M_{\ge n} \ar[d]^{\psi} \\
R'_{\ge n} \ar[r]_{\theta'_{\ge n}}              & M'_{\ge n}
}
\]
commutes.  Then the following two conditions are equivalent.
\begin{enumerate}
\item $R$ is noetherian and $M$ is finitely generated as left $R$-module.
\item $R'$ is noetherian and $M'$ is finitely generated as left $R'$-module.
\end{enumerate}
\end{prop}
\begin{proof}
We prove that condition (i) implies condition (ii).  The opposite
implication is proved in exactly the same way by using $\phi^{-1}$ and
$\psi^{-1}$ instead of $\phi$ and $\psi$.

Assume that condition (i) holds.  Then by
Proposition~\ref{prop:graded-noetherian}, $R'$ is noetherian.  We need
to show that $M'$ is finitely generated as left $R'$-module.

We begin with choosing generating sets for things we know to be
finitely generated.  Note that the ideal $R_{\ge n}$ of $R$ is
finitely generated, since $R$ is noetherian.  Let $A$ be a finite
homogeneous generating set for $R_{\ge n}$.  Let $G$ be a finite
homogeneous generating set for $M$ as left $R$-module.  For every $i$,
let $B_i$ be a finite generating set for $M'_i$ as left $R'_0$-module.

Let
\[
b_R = \max \big\{ \degree{a} \mathrel{\big|} a \in A \big\}
\qquad\text{and}\qquad
b_M = \max \big\{ \degree{g} \mathrel{\big|} g \in G \big\}
\]
be the maximal degrees of elements in our chosen generating sets for
$R$ and $M$, respectively.  Let
\[
b = b_R + b_M + n.
\]
Define the set $G'$ to be
\[
G' = \Union_{i=0}^b B_i.
\]
We want to show that $G'$ generates $M'$ as left $R'$-module.

Let $N'$ be the $R'$-submodule of $M'$ generated by $G'$.  It is clear
that $N'$ contains every homogeneous element of $M'$ with degree at
most $b$.  Let $m' \in M'$ be a homogeneous element with $\degree{m'}
> b$.  Let $m = \psi^{-1}(m')$.  We can write $m$ as a sum
\[
m = \sum_i \theta(r_i) \cdot g_i,
\]
where every $r_i$ is a homogeneous nonzero element of $R$ and every
$g_i$ is an element of the generating set $G$ for $M$.  For every
$r_i$, we have
\[
\degree{r_i}
 = \degree{m} - \degree{g_i}
 = \degree{m'} - \degree{g_i}
 > b - b_M
 = b_R + n.
\]
Thus $r_i$ lies in the ideal $R_{\ge n}$, so we can write it as a sum
\[
r_i = \sum_j u_{i,j} \cdot a_{i,j},
\]
where every $u_{i,j}$ is a homogeneous nonzero element of $R$, and
every $a_{i,j}$ is an element of the generating set $A$ for $R_{\ge
  n}$.  For every $u_{i,j}$, we have
\[
\degree{u_{i,j}}
 = \degree{r_i} - \degree{a_{i,j}}
 > (b_R + n) - b_R
 = n.
\]
Now we can write the element $m$ as
\[
m = \sum_{i,j} \theta(u_{i,j} \cdot a_{i,j}) \cdot g_i
  = \sum_{i,j} \theta(u_{i,j}) \cdot \theta(a_{i,j}) \cdot g_i
  = \sum_{i,j} \theta(u_{i,j}) \cdot (a_{i,j} \cdot g_i).
\]
If we have $a_{i,j} \cdot g_i = 0$ for some terms in the sum, we
ignore these terms.  For every pair $(i,j)$, we have
\[
\degree{\theta(u_{i,j})} = \degree{u_{i,j}} > n
\qquad\text{and}\qquad
\degree{a_{i,j} \cdot g_i} \ge \degree{a_{i,j}} \ge n.
\]
This means that when applying $\psi$ to a term in the above sum for
$m$, we have
\[
\psi(\theta(u_{i,j}) \cdot (a_{i,j} \cdot g_i))
= \psi(\theta(u_{i,j})) \cdot \psi(a_{i,j} \cdot g_i).
\]
Using this, we can write our element $m'$ of $M'$ in the following
way$\colon$
\[
m'
 = \psi(m)
 = \psi\Big( \sum_{i,j} \theta(u_{i,j}) \cdot (a_{i,j} \cdot g_i) \Big)
 = \sum_{i,j} \psi(\theta(u_{i,j})) \cdot \psi(a_{i,j} \cdot g_i)
 = \sum_{i,j} \theta'(\phi(u_{i,j})) \cdot \psi(a_{i,j} \cdot g_i).
\]
For every pair $(i,j)$, we have
\[
\degree{\psi(a_{i,j} \cdot g_i)} = \degree{a_{i,j} \cdot g_i}
 = \degree{a_{i,j}} + \degree{g_i}
 \le b_R + b_M
 \le b,
\]
so $\psi(a_{i,j} \cdot g_i)$ lies in the module $N'$ generated by
$G'$.  Thus $m'$ also lies in $N'$.  Since every homogeneous element
of $M'$ lies in $N'$, we have $M' = N'$, and hence $M'$ is finitely
generated.
\end{proof}

Finally, we apply the above result to the rings which are involved in
the \fg{} condition, and obtain the main result of this section.

\begin{prop}
\label{prop:fg-two-algebras}
Let $\Lambda$ and $\Gamma$ be artin algebras over a commutative ring
$k$, and assume that they are flat as $k$-modules.  Let $M$ and $M'$
be $\Lambda$-modules, and let $N$ and $N'$ be $\Gamma$-modules, such
that $M \iso \Lambda/(\rad \Lambda)$ and $N' \iso \Gamma/(\rad
\Gamma)$.  Let $n$ be some nonnegative integer, and assume that there
are graded rng isomorphisms $f$, $g$, $f'$ and $g'$ making the
following two diagrams commute$\colon$
\[
\vcenter{\hbox{
\xymatrix{
{\HH{\ge n}(\Lambda)} \ar[r]^-{\varphi_M^{\ge n}} \ar[d]_{f}^{\iso} &
{\Ext_\Lambda^{\ge n}(M, M)}  \ar[d]_{g}^{\iso} \\
{\HH{\ge n}(\Gamma)} \ar[r]_-{\varphi_{N}^{\ge n}} &
{\Ext_{\Gamma}^{\ge n}(N, N)}
}}}
\qquad\text{and}\qquad
\vcenter{\hbox{
\xymatrix{
{\HH{\ge n}(\Lambda)} \ar[r]^-{\varphi_{M'}^{\ge n}} \ar[d]_{f'}^{\iso} &
{\Ext_\Lambda^{\ge n}(M', M')}  \ar[d]_{g'}^{\iso} \\
{\HH{\ge n}(\Gamma)} \ar[r]_-{\varphi_{N'}^{\ge n}} &
{\Ext_{\Gamma}^{\ge n}(N', N')}
}}}
\]
Then $\Lambda$ satisfies \fg{} if and only if $\Gamma$ satisfies
\fg{}.
\end{prop}
\begin{proof}
We first check that the conditions on the graded rings in
Proposition~\ref{prop:graded-fin-gen} are satisfied in this case.  For
every degree $i$, we have that $\HH{i}(\Lambda)$,
$\Ext_\Lambda^i(M,M)$ and $\Ext_\Lambda^i(M',M')$ are finitely
generated as $k$-modules.  Therefore, they are also finitely generated
as $\HH{0}(\Lambda)$-modules.  The ring $\HH{0}(\Lambda)$ is
noetherian since it is an artin algebra.  Similarly, we see that
$\HH{i}(\Gamma)$, $\Ext_\Gamma^i(N,N)$ and $\Ext_\Gamma^i(N',N')$ are
finitely generated $\HH{0}(\Gamma)$-modules, and that the ring
$\HH{0}(\Gamma)$ is noetherian.

Assume that $\Lambda$ satisfies \fg{}.  Then $\HH*(\Lambda)$ is
noetherian, and by Theorem~\ref{fg-implies-every-ext-fin-gen},
$\Ext_\Lambda^*(M',M')$ is a finitely generated
$\HH*(\Lambda)$-module.  By applying
Proposition~\ref{prop:graded-fin-gen} to the commutative diagram with
$f'$ and $g'$, we see that $\Gamma$ satisfies \fg{}.

The opposite inclusion is proved in the same way by using the other
commutative diagram.
\end{proof}

\section{Finite generation of cohomology rings in module recollements}
\label{section:fg-a(Lambda)a}

We now investigate the relationship between the \fg{} condition (see
Definition~\ref{defn:fg}) for an algebra $\Lambda$ and the algebra
$a{\Lambda}a$, where $a$ is an idempotent of $\Lambda$.  We show that,
given some conditions on the idempotent $a$, the algebra $\Lambda$
satisfies \fg{} if and only if the algebra $a{\Lambda}a$ satisfies
\fg{}.  We prove this result only for finite-dimensional algebras over
a field, and not more general artin algebras.

Throughout this section, we let $k$ be a field, $\Lambda$ a
finite-dimensional $k$-algebra and $a$ an idempotent in $\Lambda$.  We
denote by $e$ and $E$ the exact functors
\begin{align*}
e = (a-) &\colon \fmod \Lambda \lxr \fmod a{\Lambda}a \\
E = (a-a) &\colon \fmod \envalg{\Lambda} \lxr \fmod \envalg{(a{\Lambda}a)}.
\end{align*}
These functors fit into the recollements described in
Example~\ref{exam:mod-recollements}.

For a $\Lambda$-module $M$, we can construct the diagram
\[
\xymatrix@C=4em{
{\HH*(\Lambda)}       \ar[r]^-{\varphi_M} \ar[d]_{E_{\Lambda,\Lambda}^*} &
{\Ext_\Lambda^*(M,M)} \ar[d]^{e_{M,M}^*} \\
{\HH*(a{\Lambda}a)}   \ar[r]_-{\varphi_{e(M)}} &
{\Ext_{a{\Lambda}a}^*(e(M), e(M))}
}
\]
where the maps $\varphi_M$ and $\varphi_{e(M)}$ are defined in
Subsection~\ref{subsection:hh}, and the maps $E_{\Lambda,\Lambda}^*$
and $e_{M,M}^*$ are defined in Section~\ref{section:extensions}.  We
show that this diagram commutes, and that under certain conditions on
$a$, the vertical maps are isomorphisms in almost all degrees.  We
then use Proposition~\ref{prop:fg-two-algebras} to show that $\Lambda$
satisfies \fg{} if and only if $a{\Lambda}a$ satisfies \fg{}.

Let us consider what kind of conditions we need to put on the choice
of the idempotent $a$.  From Corollary~\ref{cor:algebra-ext-iso}, we
know that the map $e_{M,M}^*$ in the above diagram is an isomorphism
in all but finitely many degrees if the two dimensions
\[
\id_\Lambda
 \Big( \frac{\Lambda/\idealgenby{a}}{\rad \Lambda/\idealgenby{a}} \Big)
\qquad\text{and}\qquad
\pd_{a{\Lambda}a} (a\Lambda)
\]
are finite, or, equivalently, that the two dimensions
\[
\pd_\Lambda
 \Big( \frac{\Lambda/\idealgenby{a}}{\rad \Lambda/\idealgenby{a}} \Big)
\qquad\text{and}\qquad
\pd_{\opposite{(a{\Lambda}a)}} ({\Lambda}a)
\]
are finite.  We show (given an additional technical assumption about
the algebra $\Lambda$) that this is in fact also sufficient for the
map $E_{\Lambda,\Lambda}^*$ to be an isomorphism in all but finitely
many degrees.

This section is structured as follows.  The first part considers the
commutativity of the above diagram, concluding with
Proposition~\ref{prop:e-gamma-commute}. The second part considers when 
the map $E_{\Lambda,\Lambda}^*$ is an isomorphism in high
degrees, concluding with Proposition~\ref{prop:lambda-ext-iso}.
Finally, the main result of this section is stated as
Theorem~\ref{thm:main-result-fg}.

We now show that the above diagram is commutative.  The maps
$\varphi_M$ and $\varphi_{e(M)}$ are defined by using tensor functors.
It is convenient to have short names for these functors.  For every
$\Lambda$-module $M$, we define $t_M$ and $T_M$ to be the tensor
functors
\begin{align*}
t_M = (- \tensor_\Lambda M) &\colon
\fmod \envalg{\Lambda} \lxr \fmod \Lambda, \\
T_M = (- \tensor_{a{\Lambda}a} aM) &\colon
\fmod \envalg{(a{\Lambda}a)} \lxr \fmod a{\Lambda}a.
\end{align*}
Together with the functors $e$ and $E$ from above, these functors fit
into the following diagram of categories and functors:
\[
\xymatrix@C=6em{
{\fmod \envalg{\Lambda}} \ar[r]^{t_M} \ar[d]_{E} &
{\fmod \Lambda} \ar[d]^{e} \\
{\fmod \envalg{(a \Lambda a)}} \ar[r]_{T_M } &
{\fmod a \Lambda a}
}
\]

We begin by showing that the two possible compositions of maps from
upper left to lower right in this diagram are related by a natural
transformation.

\begin{lem}
\label{lem:et-commutes}
For every $\Lambda$-module $M$, there is a natural transformation
$\tau^M \colon T_M \fcomp E \lxr e \fcomp t_M$.
\end{lem}
\begin{proof}
Note that we have 
\[
T_M E(N) = aNa \tensor_{a{\Lambda}a} aM
\qquad\text{and}\qquad
e t_M(N) = aN \tensor_\Lambda M
\]
for every $\envalg{\Lambda}$-module $N$.  We define the maps
$\tau^M_N$ of the natural transformation $\tau^M$ by
\[
\tau^M_N(n \tensor m) = n \tensor m
\]
for an element $n \tensor m$ of $T_M E(N)$. This gives well defined
maps since $a\Lambda a \subseteq \Lambda$. It is easy to check that
the compositions $et_M(f) \fcomp \tau^M_N$ and $\tau^M_{N'} \fcomp T_M
E(f)$ are equal for a homomorphism $f\colon N \lxr N'$ of
$\envalg{\Lambda}$-modules, so $\tau^M$ is a natural transformation.
\end{proof}

We are now able to show that the diagrams we consider are commutative.

\begin{prop}
\label{prop:e-gamma-commute}
For any $\Lambda$-module $M$, the following diagram of graded rings
commutes$\colon$
\[
\xymatrix@C=4em{
{\HH*(\Lambda)}       \ar[r]^-{\varphi_M} \ar[d]_{E_{\Lambda,\Lambda}^*} &
{\Ext_\Lambda^*(M,M)} \ar[d]^{e_{M,M}^*} \\
{\HH*(a{\Lambda}a)}   \ar[r]_-{\varphi_{e(M)}} &
{\Ext_{a{\Lambda}a}^*(e(M), e(M))}
}
\]
\end{prop}
\begin{proof}
We show that the result holds in the positive degrees of the graded
rings and graded ring homomorphisms in the diagram.  Showing that it
also holds in degree zero can be done in a similar way, by looking at
elements given by homomorphisms instead of extensions.

Let $\mu$ and $\nu$ be the natural isomorphisms
\[
\mu\colon \Lambda \tensor_\Lambda M \lxr M
\qquad\text{and}\qquad
\nu\colon a{\Lambda}a \tensor_{a{\Lambda}a} e(M) \lxr e(M)
\]
given by multiplication.

Consider, for some positive integer $i$, an element $[\eta] \in
\Ext_{\envalg{\Lambda}}^i(\Lambda,\Lambda)$ which is represented by
the exact sequence
\[
\eta\colon\quad
0 \lxr \Lambda \lxr X \lxr P_{i-2} \lxr \cdots \lxr P_0 \lxr \Lambda \lxr 0,
\]
where each $P_j$ is a projective $\envalg{\Lambda}$-module.  We apply
the compositions of maps $\varphi_{e(M)} \fcomp E_{\Lambda,\Lambda}^*$
and $e_{M,M}^* \fcomp \varphi_M$ to $[\eta]$, and show that we get the
same result in both cases.

We first consider the map $\varphi_{e(M)} \fcomp
E_{\Lambda,\Lambda}^*$.  If we apply the functor $E$
to $\eta$, then we get the exact sequence
\[
E(\eta)\colon
0 \lxr E(\Lambda) \lxr E(X) \lxr E(P_{i-2})
  \lxr \cdots \lxr E(P_0) \lxr E(\Lambda)
  \lxr 0 
\]  
of $\envalg{(a \Lambda a)}$-modules, and we have that
$E_{\Lambda,\Lambda}^*([\eta]) = [E(\eta)]$.  Since
the objects $E(P_j)$ are not necessarily projective, we may
need to find a different representative of the element $[E(\eta)]$ in
order to apply the map $\varphi_{e(M)}$.  We construct the
following commutative diagram with exact rows, where each $Q_j$ is a
projective $\envalg{(a \Lambda a)}$-module and the bottom row is
$E(\eta)$.
\[
\xymatrix{
0           \ar[r]                  &
a{\Lambda}a \ar[r] \ar@{=}[d]       &
Y           \ar[r] \ar[d]_{f_{i-1}} &
Q_{i-2}     \ar[r] \ar[d]_{f_{i-2}} &
\cdots      \ar[r]                  &
Q_0         \ar[r] \ar[d]_{f_0}     &
a{\Lambda}a \ar[r] \ar@{=}[d]       &
0                                   \\
0           \ar[r]                  &
E(\Lambda)  \ar[r]                  &
E(X)        \ar[r]                  &
E(P_{i-2})  \ar[r]                  &
\cdots      \ar[r]                  &
E(P_0)      \ar[r]                  &
E(\Lambda)  \ar[r]                  &
0
}
\]
Note that both rows represent the same element in $\Ext_{\envalg{(a
    \Lambda a)}}^i(a \Lambda a, a \Lambda a)$.  Applying the functor
$T_M$ to this diagram gives the two lower rows in the following
commutative diagram of $a \Lambda a$-modules, where the two upper rows
are exact.
\[
\xymatrix@C=1.5em{
0                \ar[r]                                  &
e(M)             \ar[r] \ar[d]_{\nu^{-1}}^{\iso}         &
T_M(Y)           \ar[r] \ar@{=}[d]                       &
T_M(Q_{i-2})     \ar[r] \ar@{=}[d]                       &
\cdots           \ar[r]                                  &
T_M(Q_0)         \ar[r] \ar@{=}[d]                       &
e(M)             \ar[r] \ar[d]_{\nu^{-1}}^{\iso}         &
0                                                        \\
0                \ar[r]                                  &
T_M(a{\Lambda}a) \ar[r] \ar@{=}[d]                       &
T_M(Y)           \ar[r] \ar[d]_{T_M(f_{i-1})}            &
T_M(Q_{i-2})     \ar[r] \ar[d]_{T_M(f_{i-2})}            &
\cdots           \ar[r]                                  &
T_M(Q_0)         \ar[r] \ar[d]_{T_M(f_0)}                &
T_M(a{\Lambda}a) \ar[r] \ar@{=}[d]                       &
0                                                        \\
                                                         &
T_M E(\Lambda)   \ar[r]                                  &
T_M E(X)         \ar[r]                                  &
T_M E(P_{i-2})   \ar[r]                                  &
\cdots           \ar[r]                                  &
T_M E(P_0)       \ar[r]                                  &
T_M E(\Lambda)                                           &
}
\]
The top row in this diagram is a representative for the element
$(\varphi_{e(M)} \fcomp E_{\Lambda,\Lambda}^*)([\eta])$.

We now consider the map $e_{M,M}^* \fcomp \varphi_M$.  Applying the
functor $e \fcomp t_M$ to the exact sequence $\eta$ gives the top row
in the following commutative diagram of $a \Lambda a$-modules with
exact rows, where the bottom row is a representative of the element
$(e_{M,M}^* \fcomp \varphi_M)([\eta])$.
\[
\xymatrix@C=1.5em{
0 \ar[r]                                      &
e t_M(\Lambda)  \ar[r] \ar[d]_{e(\mu)}^{\iso} &
e t_M(X)        \ar[r] \ar@{=}[d]             &
e t_M(P_{i-2})  \ar[r] \ar@{=}[d]             &
\cdots          \ar[r]                        &
e t_M(P_0)      \ar[r] \ar@{=}[d]             &
e t_M(\Lambda)  \ar[r] \ar[d]_{e(\mu)}^{\iso} &
0                                             \\
0 \ar[r]                                      &
e(M)            \ar[r]                        &
e t_M(X)        \ar[r]                        &
e t_M(P_{i-2})  \ar[r]                        &
\cdots          \ar[r]                        &
e t_M(P_0)      \ar[r]                        &
e(M)            \ar[r]                        &
0                                             \\
}
\]

Finally, we use the natural transformation $\tau^M$ from
Lemma~\ref{lem:et-commutes} to combine the two above diagrams into the
following commutative diagram of $a \Lambda a$-modules$\colon$
\[
\xymatrix@C=1.5em{
0 \ar[r]                                               &
e(M)             \ar[r] \ar[d]_{\nu^{-1}}^{\iso}       &
T_M(Y)           \ar[r] \ar@{=}[d]                     &
T_M(Q_{i-2})     \ar[r] \ar@{=}[d]                     &
\cdots           \ar[r]                                &
T_M(Q_0)         \ar[r] \ar@{=}[d]                     &
e(M)             \ar[r] \ar[d]_{\nu^{-1}}^{\iso}       &
0                                                      \\
0 \ar[r]                                               &
T_M(a{\Lambda}a) \ar[r] \ar@{=}[d]                     &
T_M(Y)           \ar[r] \ar[d]_{T_M(f_{i-1})}          &
T_M(Q_{i-2})     \ar[r] \ar[d]_{T_M(f_{i-2})}          &
\cdots           \ar[r]                                &
T_M(Q_0)         \ar[r] \ar[d]_{T_M(f_0)}              &
T_M(a{\Lambda}a) \ar[r] \ar@{=}[d]                     &
0                                                      \\
                                                       &
T_M E(\Lambda)   \ar[r] \ar[d]_{\tau^M_\Lambda}        &
T_M E(X)         \ar[r] \ar[d]_{\tau^M_X}              &
T_M E(P_{i-2})   \ar[r] \ar[d]_{\tau^M_{P_{i-2}}}      &
\cdots           \ar[r]                                &
T_M E(P_0)       \ar[r] \ar[d]_{\tau^M_{P_0}}          &
T_M E(\Lambda)          \ar[d]_{\tau^M_\Lambda}        &
                                                       \\
0 \ar[r]                                               &
e t_M(\Lambda)   \ar[r] \ar[d]_{e(\mu)}^{\iso}         &
e t_M(X)         \ar[r] \ar@{=}[d]                     &
e t_M(P_{i-2})   \ar[r] \ar@{=}[d]                     &
\cdots           \ar[r]                                &
e t_M(P_0)       \ar[r] \ar@{=}[d]                     &
e t_M(\Lambda)   \ar[r] \ar[d]_{e(\mu)}^{\iso}         &
0                                                      \\
0 \ar[r]                                               &
e(M)             \ar[r]                                &
e t_M(X)         \ar[r]                                &
e t_M(P_{i-2})   \ar[r]                                &
\cdots           \ar[r]                                &
e t_M(P_0)       \ar[r]                                &
e(M)             \ar[r]                                &
0                                                      \\
}
\]
It is easy to check that the composition of maps along the leftmost
column is the identity map on $e(M)$, and the same holds for the
composition of maps along the rightmost column.  Thus the top and
bottom rows in this diagram represent the same element in $\Ext_{a
  \Lambda a}^i(e(M),e(M))$.  Since the top row is a representative of
the element $(\varphi_{e(M)} \fcomp E_{\Lambda,\Lambda}^*)([\eta])$ and
the bottom row is a representative of the element $(e_{M,M}^* \fcomp
\varphi_M)([\eta])$, this means that $\varphi_{e(M)} \fcomp E_{\Lambda,\Lambda}^*
 = e_{M,M}^* \fcomp \varphi_M$.
\end{proof}

Having shown that our diagrams are commutative, we now move on
to describing when the map $E_{\Lambda,\Lambda}^*$ is an isomorphism
in almost all degrees.  For this, we use
Corollary~\ref{cor:algebra-ext-iso-onemodule}~(i) on the algebras
$\envalg{\Lambda}$ and
$(a\tensor\opposite{a})\envalg{\Lambda}(a\tensor\opposite{a})$ and the
$\envalg{\Lambda}$-module $\Lambda$.  We let $\varepsilon$ denote the
element $a\tensor\opposite{a}$ of $\envalg{\Lambda}$, so that we can
write the algebra
$(a\tensor\opposite{a})\envalg{\Lambda}(a\tensor\opposite{a})$ more
simply as $\varepsilon\envalg{\Lambda}\varepsilon$.  Note that
Corollary~\ref{cor:algebra-ext-iso-onemodule} uses a recollement
situation; in this case, the recollement is like the one in
Example~\ref{exam:mod-recollements}~(ii).

In order to use Corollary~\ref{cor:algebra-ext-iso-onemodule}~(i) in
this situation, we need to show the following$\colon$
\[
\pd_{\varepsilon\envalg{\Lambda}\varepsilon} \varepsilon\envalg{\Lambda}
< \infty
\qquad\text{and}\qquad
\Ext_{\envalg{\Lambda}}^j
\Big( \Lambda,
      \frac{\envalg{\Lambda}/\idealgenby{\varepsilon}}
           {\rad \envalg{\Lambda}/\idealgenby{\varepsilon}} \Big)
= 0
\quad\text{for $j \gg 0$.}
\]
We show the first of these conditions in Lemma~\ref{lem:pd-envalg},
and the second one in Lemma~\ref{lem:ext-lambda-simples} (here we need
an additional technical assumption on $\Lambda$ to be able to describe
the simple modules over $\envalg{\Lambda}$), and finally tie it
together in Proposition~\ref{prop:lambda-ext-iso}, where we show that
$E_{\Lambda,\Lambda}^*$ is an isomorphism in sufficiently high
degrees.

First, we show how the projective dimension of the tensor product $M
\tensor_k N$ is related to the projective dimensions of $M$ and $N$,
when $M$ and $N$ are modules over $k$-algebras.  In particular, the
following result implies that if a left and a right $\Lambda$-module
${}_\Lambda M$ and $N_\Lambda$ both have finite projective dimension,
then their tensor product $M \tensor_k N$ has finite projective
dimension as $\envalg{\Lambda}$-module.

\begin{lem}
\label{lem:tensor-preserves-fin-projdim}
Let $\Sigma$ and $\Gamma$ be $k$-algebras, and let $M$ be a
$\Sigma$-module and $N$ a $\Gamma$-module.  If $M$ has finite
projective dimension as $\Sigma$-module and $N$ has finite projective
dimension as $\Gamma$-module, then $M \tensor_k N$ has finite
projective dimension as $(\Sigma \tensor_k \Gamma)$-module, and
\[
\pd_{\Sigma \tensor_k \Gamma}(M \tensor_k N) \le
\pd_\Sigma{M} + \pd_\Gamma{N}.
\]
\end{lem}
\begin{proof}
Assume that $\pd_{\Sigma}M=m$ and $\pd_{\Gamma}N=n$.  Then we have
finite projective resolutions
\[
0 \to P_m \xrightarrow{} \cdots \xrightarrow{} P_0 \to M \to 0
\qquad\text{and}\qquad
0 \to Q_n \xrightarrow{} \cdots \xrightarrow{} Q_0 \to N \to 0
\]
of $M$ and $N$, respectively.  Let $P$ and $Q$ denote the
corresponding deleted resolutions.
Consider the tensor product 
\[
P \tensor_k Q \colon \cdots \xrightarrow{}
(P_0 \tensor_k Q_2) \dsum (P_1 \tensor_k Q_1) \dsum (P_2 \tensor_k Q_0)
   \xrightarrow{}
(P_0 \tensor_k Q_1) \dsum (P_1 \tensor_k Q_0) \xrightarrow{}
P_0 \tensor_k Q_0 \to 0
\]
of the complexes $P$ and $Q$.  This is a bounded complex of projective
($\Sigma \tensor_k \Gamma$)-modules.  We want to show that it is in
fact a deleted projective resolution of the $(\Sigma \tensor_k
\Gamma)$-module $M \tensor_k N$, which completes the proof.

We need to show that the complex $P \tensor_k Q$ is exact in all
positive degrees and has homology $M \tensor_k N$ in degree zero.  Let
us temporarily forget the $\Sigma$- and $\Gamma$-structures, and view
$P$ as a complex of right $k$-modules, $Q$ as a complex of left
$k$-modules, and $P \tensor_k Q$ as a complex of abelian groups.  Then
by the K\"unneth formula for homology, see
\cite[Corollary~11.29]{Rotman1}, we have an isomorphism
\[
\alpha \colon
\Dsum_{i+j=n} H_i(P) \tensor_k H_j(Q) \xrightarrow{\iso}
H_n(P \tensor_k Q)
\]
of abelian groups, given by $\alpha([p] \tensor [q]) = [p \tensor q]$,
for $p \in P_i$ and $q \in Q_j$.  Observe that $\alpha$ preserves
$(\Sigma \tensor_k \Gamma)$-module structure.
Thus, $\alpha$ is a $\Sigma \tensor_k \Gamma$-module isomorphism, and
we get
\[
H_n(P \tensor_k Q) \iso
\Dsum_{i+j=n} H_i(P) \tensor_k H_j(Q) \iso
\left\{
\begin{array}{ll}
M \tensor_k N & \text{if $n = 0$} \\
0             & \text{if $n > 0$}
\end{array}
\right.
\]
This means that the complex $P \tensor_k Q$ is a deleted projective
resolution of the $(\Sigma \tensor_k \Gamma)$-module $M \tensor_k N$.
Since the complex $P \tensor_k Q$ is zero in all degrees above $m+n$,
we get
\[
\pd_{\Sigma \tensor_k \Gamma} (M \tensor_k N)
 \le m + n
 = \pd_\Sigma{M} + \pd_\Gamma{N},
\]
and the proof is complete.
\end{proof}

Using the above result, we find that the assumptions we make about the
left and right $a{\Lambda}a$-modules $a\Lambda$ and ${\Lambda}a$
having finite projective dimension imply the first condition we need
for applying Corollary~\ref{cor:algebra-ext-iso-onemodule}~(i), namely
that the $\varepsilon\envalg{\Lambda}\varepsilon$-module
$\varepsilon\envalg{\Lambda}$ has finite projective dimension.  We
state this as the following result.

\begin{lem}
\label{lem:pd-envalg}
We have the following inequality$\colon$
\[
\pd_{\varepsilon\envalg{\Lambda}\varepsilon}\varepsilon\envalg{\Lambda}
\le
\pd_{a{\Lambda}a} a{\Lambda} +
\pd_{\opposite{(a{\Lambda}a)}} {\Lambda}a
\]
\end{lem}
\begin{proof}
Note that $\varepsilon \envalg{\Lambda}$ is isomorphic to
$(a\Lambda\tensor_k\Lambda a)$ as left $\envalg{(a\Lambda a)}$-modules
and that the rings $\envalg{(a{\Lambda}a)}$ and
$\varepsilon\envalg{\Lambda}\varepsilon$ are isomorphic.  By using
these isomorphisms and Lemma~\ref{lem:tensor-preserves-fin-projdim}, we
get that
\[ 
\pd_{\varepsilon \envalg{\Lambda}\varepsilon}\varepsilon\envalg{\Lambda}
 = \pd_{\envalg{(a\Lambda a)}}\varepsilon\envalg{\Lambda}
 = \pd_{\envalg{(a\Lambda a)}}(a\Lambda\tensor_k \Lambda a)
 \leq \pd_{a\Lambda a} a\Lambda+\pd_{\opposite{(a\Lambda a)}}\Lambda a.
\qedhere
\]
\end{proof}

Now we show how we get the second condition needed for applying
Corollary~\ref{cor:algebra-ext-iso-onemodule}~(i).  We begin with a
general result which relates extension groups over $\envalg{\Lambda}$
to extension groups over $\Lambda$.

\begin{lem}
\label{lem:ext-over-envalg-as-ext-over-lambda}
Let $M$ and $N$ be $\Lambda$-modules.  Let $D$ be the duality
$\Hom_k(-,k)\colon \fmod \Lambda \lxr \fmod \opposite{\Lambda}$.  Then
\[
\Ext_{\envalg{\Lambda}}^j(\Lambda, M \tensor_k D(N))
\iso
\Ext_\Lambda^j(N, M)
\]
for every nonnegative integer $j$.
\begin{proof}
This follows from \cite[Corollary~4.4, Chapter~IX]{CartanEilenberg}
by using the isomorphism $M \tensor_k D(N) \iso \Hom_k(N,M)$ of
$\envalg{\Lambda}$-modules.
\end{proof}
\end{lem}

Furthermore, we need to be able to describe the simple
$\envalg{\Lambda}$-modules in terms of simple $\Lambda$-modules.  It
is reasonable to expect that taking the tensor product
\[
(\Lambda/\rad \Lambda) \tensor_k
(\opposite{\Lambda}/\rad \opposite{\Lambda})
\]
should produce all the simple $\envalg{\Lambda}$-modules.  This is,
however, not true for all finite-dimensional algebras, as
Example~\ref{exam:not-semisimple} shows.  The following result
describes when it is true.

\begin{lem}
\label{lem:envalg-simples}
We have an isomorphism
\[
\envalg{\Lambda}/\rad \envalg{\Lambda}
\iso
(\Lambda/\rad \Lambda) \tensor_k
(\opposite{\Lambda}/\rad \opposite{\Lambda})
\]
of $\envalg{\Lambda}$-modules if and only if the
$\envalg{\Lambda}$-module
\[
(\Lambda/\rad \Lambda) \tensor_k
(\opposite{\Lambda}/\rad \opposite{\Lambda})
\]
is semisimple.
\begin{proof}
It is easy to show that
\[
(\Lambda/\rad \Lambda) \tensor_k
(\opposite{\Lambda}/\rad \opposite{\Lambda})
\iso
\frac{\envalg{\Lambda}}
     {\Lambda \tensor_k (\rad \opposite{\Lambda}) +
      (\rad \Lambda) \tensor_k \opposite{\Lambda}}
\]
as $\envalg{\Lambda}$-modules, and that the ideal
$\Lambda \tensor_k (\rad \opposite{\Lambda}) +
      (\rad \Lambda) \tensor_k \opposite{\Lambda}
$
of $\envalg{\Lambda}$ is nilpotent.  This means that if
$
(\Lambda/\rad \Lambda) \tensor_k
(\opposite{\Lambda}/\rad \opposite{\Lambda})
$ is a semisimple $\envalg{\Lambda}$-module, then it is
isomorphic to $\envalg{\Lambda}/\rad \envalg{\Lambda}$.
The opposite implication is obvious.
\end{proof}
\end{lem}

Now we give an example showing that $(\Lambda/\rad\Lambda) \tensor_k
(\opposite{\Lambda}/\rad\opposite{\Lambda})$ is not necessarily
semisimple for a finite dimensional algebra $\Lambda$ over a field
$k$.
\begin{exam}
\label{exam:not-semisimple}
Let $k=\mathbb{Z}_2(x)$ be the field of rational functions in one
indeterminant $x$ over $\mathbb{Z}_2$, and let $\Lambda$ be the
$2$-dimensional $k$-algebra $k[y]/\idealgenby{y^2 - x}$.  Then
$\Lambda$ is a field, so that $\rad\Lambda =(0)$.  The element $\alpha
= y\tensor 1 + 1\tensor y$ satisfies $\alpha^2=0$.  Hence
$\idealgenby{\alpha}$ is a nilpotent non-zero ideal in
$\envalg{\Lambda}$, and therefore $\envalg{\Lambda}$ is not
semisimple.
\end{exam}

We assume that $(\Lambda/\rad \Lambda) \tensor_k
(\opposite{\Lambda}/\rad \opposite{\Lambda})$ is semisimple whenever
we need it.  In particular, this assumption is included in the main
result at the end of this section.  Note that this assumption is
satisfied in many cases, for example if $\Lambda/\rad \Lambda$ is
separable as $k$-algebra (by \cite[Corollary~7.8~(i)]{CurtisReiner})
if $k$ is algebraically closed (this can be shown by using the
Wedderburn--Artin Theorem), or if $\Lambda$ is a quotient of a path
algebra by an admissible ideal.

Now we can show how to get the second condition we need for applying
Corollary~\ref{cor:algebra-ext-iso-onemodule}~(i).

\begin{lem}
\label{lem:ext-lambda-simples}
Assume that $(\Lambda/\rad \Lambda) \tensor_k (\opposite{\Lambda}/\rad
\opposite{\Lambda})$ is a semisimple $\envalg{\Lambda}$-module, and
that we have
\[
(\alpha)
\ \id_\Lambda \Big( \frac{\Lambda/\idealgenby{a}}{\rad \Lambda/\idealgenby{a}} \Big)
 < \infty
\qquad\text{and}\qquad
(\gamma)
\ \pd_\Lambda \Big( \frac{\Lambda/\idealgenby{a}}{\rad \Lambda/\idealgenby{a}} \Big)
 < \infty.
\]
Then
\[
\Ext_{\envalg{\Lambda}}^j
\Big( \Lambda,
      \frac{\envalg{\Lambda}/\idealgenby{\varepsilon}}
           {\rad \envalg{\Lambda}/\idealgenby{\varepsilon}} \Big)
= 0
\quad\text{for}\quad
j >
\max \Big\{
        \pd_\Lambda
          \Big( \frac{\Lambda/\idealgenby{a}}
                     {\rad \Lambda/\idealgenby{a}}
          \Big),
        \,
        \id_\Lambda
          \Big( \frac{\Lambda/\idealgenby{a}}
                     {\rad \Lambda/\idealgenby{a}}
          \Big)
     \Big\}.
\]
\end{lem}
\begin{proof}
By Lemma~\ref{lem:envalg-simples}, every simple
$\envalg{\Lambda}$-module is a direct summand of a module of the form
$S \tensor_k D(T)$ for some simple $\Lambda$-modules $S$ and $T$,
where $D$ is the duality $\Hom_k(-,k)\colon \fmod \Lambda \lxr \fmod
\opposite{\Lambda}$.  If neither of the modules $S$ or $T$ is
annihilated by the ideal $\idealgenby{a}$, then we have
\[
\idealgenby{\varepsilon} (S \tensor_k D(T)) =
 \idealgenby{a \tensor \opposite{a}} (S \tensor_k D(T))
 = (\idealgenby{a}S) \tensor_k D(\idealgenby{a}T)
 = S \tensor_k D(T),
\]
which means that no nonzero direct summand of the $\envalg{\Lambda}$-module
$S \tensor_k D(T)$ is a
$\envalg{\Lambda}/\idealgenby{\varepsilon}$-module.

Let $j$ be an integer such that
\[
j >
\max \Big\{
        \pd_\Lambda
          \Big( \frac{\Lambda/\idealgenby{a}}
                     {\rad \Lambda/\idealgenby{a}}
          \Big),
        \,
        \id_\Lambda
          \Big( \frac{\Lambda/\idealgenby{a}}
                     {\rad \Lambda/\idealgenby{a}}
          \Big)
     \Big\}.
\]
In order to prove the result, it is sufficient to show that
$\Ext_{\envalg{\Lambda}}^j(\Lambda, U) = 0$ for every simple
$\envalg{\Lambda}/\idealgenby{\varepsilon}$-module $U$.  By the
above reasoning, every such $U$ is a direct summand of a module $S \tensor_k
D(T)$ for some simple $\Lambda$-modules $S$ and $T$, where at least
one of $S$ and $T$ is annihilated by $\idealgenby{a}$ and is thus a
simple $\Lambda/\idealgenby{a}$-module.  Using
Lemma~\ref{lem:ext-over-envalg-as-ext-over-lambda}, we get
\[
\Ext_{\envalg{\Lambda}}^j(\Lambda, S \tensor_k D(T))
 \iso \Ext_\Lambda^j(T, S)
 = 0,
\]
since we have $\pd_\Lambda T < j$ or $\id_\Lambda S < j$.  It follows
that $\Ext_{\envalg{\Lambda}}^j(\Lambda, U) = 0$.
\end{proof}

The following result summarizes the above work and shows that, with
the assumptions we have indicated for the algebra $\Lambda$ and the
idempotent $a$, the functor $E$ gives isomorphisms $E_{\Lambda,\Lambda}^j \colon
\HH{j}(\Lambda) \lxr
\HH{j}(a{\Lambda}a)$ in almost all degrees $j$.

\begin{prop}
\label{prop:lambda-ext-iso}
Assume that $(\Lambda/\rad \Lambda) \tensor_k (\opposite{\Lambda}/\rad
\opposite{\Lambda})$ is a semisimple $\envalg{\Lambda}$-module, and
that the functor $e$ is an eventually homological isomorphism.  Then
the map
\[
E_{\Lambda,M}^j \colon
\Ext_{\envalg{\Lambda}}^j(\Lambda, M) \lxr
\Ext_{\envalg{(a{\Lambda}a)}}^j(E(\Lambda), E(M))
\]
is an isomorphism for every $\envalg{\Lambda}$-module $M$ and every
integer $j$ such that
\[
j >
\max \Big\{
        \pd_\Lambda
          \Big( \frac{\Lambda/\idealgenby{a}}
                     {\rad \Lambda/\idealgenby{a}}
          \Big),
        \,
        \id_\Lambda
          \Big( \frac{\Lambda/\idealgenby{a}}
                     {\rad \Lambda/\idealgenby{a}}
          \Big)
     \Big\}
+
\pd_{a{\Lambda}a} a{\Lambda}
+
\pd_{\opposite{(a{\Lambda}a)}}{\Lambda}a
+ 1
< \infty.
\]
In particular, we have isomorphisms
\[
\HH{j}(\Lambda) \iso \HH{j}(a{\Lambda}a)
\]
for almost all degrees $j$.
\end{prop}
\begin{proof}
We use Corollary~\ref{cor:algebra-ext-iso-onemodule}~(i) on the algebra
$\envalg{\Lambda}$, the idempotent $\varepsilon = a \tensor
\opposite{a}$ and the $\envalg{\Lambda}$-module $\Lambda$.  Let $m$
and $n$ be the integers
\[
m =
\max \Big\{
        \pd_\Lambda
          \Big( \frac{\Lambda/\idealgenby{a}}
                     {\rad \Lambda/\idealgenby{a}}
          \Big),
        \,
        \id_\Lambda
          \Big( \frac{\Lambda/\idealgenby{a}}
                     {\rad \Lambda/\idealgenby{a}}
          \Big)
     \Big\}
+ 1
\qquad\text{and}\qquad
n =
\pd_{a{\Lambda}a} a{\Lambda}
+
\pd_{\opposite{(a{\Lambda}a)}}{\Lambda}a.
\]
Note that $m$ and $n$ are finite by
Corollary~\ref{cor:algebra-ext-iso}.  By
Lemma~\ref{lem:ext-lambda-simples}, we have
\[
\Ext_{\envalg{\Lambda}}^j
\Big( \Lambda,
      \frac{\envalg{\Lambda}/\idealgenby{\varepsilon}}
           {\rad \envalg{\Lambda}/\idealgenby{\varepsilon}} \Big)
= 0
\quad\text{for}\quad
j \ge m,
\]
and by Lemma~\ref{lem:pd-envalg}, we have
\[
\pd_{\varepsilon\envalg{\Lambda}\varepsilon}\varepsilon\envalg{\Lambda}
\le
n.
\]
Now the result follows from
Corollary~\ref{cor:algebra-ext-iso-onemodule}~(i) by noting that
$\envalg{(a{\Lambda}a)}$ is the same algebra as
$\varepsilon\envalg{\Lambda}\varepsilon$ and that our functor $E =
a-a$ is the same as the functor $\varepsilon -$ given by left
multiplication with the idempotent $\varepsilon$.
\end{proof}

Finally, we conclude this section by showing that the assumptions we
have indicated imply that \fg{} holds for $\Lambda$ if and only if
\fg{} holds for $a{\Lambda}a$.  The following theorem is the main
result of this section and constitutes the fourth part of the Main
Theorem presented in the introduction.

\begin{thm}
\label{thm:main-result-fg}
Let $\Lambda$ be a finite dimensional algebra over a field $k$, and
let $a$ be an idempotent in $\Lambda$.  Assume that $(\Lambda/\rad
\Lambda) \tensor_k (\opposite{\Lambda}/\rad \opposite{\Lambda})$ is a
semisimple $\envalg{\Lambda}$-module, and that the functor $a-\colon
\fmod \Lambda \lxr \fmod a{\Lambda}a$ is an eventually homological
isomorphism.  Then $\Lambda$ satisfies \fg{} if and only if
$a{\Lambda}a$ satisfies \fg{}.
\end{thm}
\begin{proof}
For every $\Lambda$-module $M$, we can make a diagram
\[
\xymatrix@C=4em{
{\HH*(\Lambda)}       \ar[r]^-{\varphi_M} \ar[d]_{E_{\Lambda,\Lambda}^*} &
{\Ext_\Lambda^*(M,M)} \ar[d]^{e_{M,M}^*} \\
{\HH*(a{\Lambda}a)}   \ar[r]_-{\varphi_{e(M)}} &
{\Ext_{a{\Lambda}a}^*(e(M), e(M))}
}
\]
of graded rings and graded ring homomorphisms.  This diagram commutes
by Proposition~\ref{prop:e-gamma-commute}, and the maps
$E_{\Lambda,\Lambda}^*$ and $e_{M,M}^*$ are isomorphisms in almost all
degrees by Proposition~\ref{prop:lambda-ext-iso} and
Corollary~\ref{cor:algebra-ext-iso}, respectively.

Since we have such diagrams for every $\Lambda$-module $M$ and the
functor $e$ is essentially surjective (see
Proposition~\ref{properties}), we can make one diagram with $M =
\Lambda/\rad \Lambda$ and another with $e(M) \iso a{\Lambda}a/\rad
a{\Lambda}a$.  Then, by Proposition~\ref{prop:fg-two-algebras}, it
follows that $\Lambda$ satisfies \fg{} if and only if $a{\Lambda}a$
satisfies \fg{}.
\end{proof}

\section{Applications and Examples}
\label{section:examples}

In this section we provide applications of our Main Theorem (stated in
the Introduction), and examples illustrating its use.  For ease of
reference, we restate the Main Theorem here.

\begin{thm}
\label{thm:main-results-summary}
Let $\Lambda$ be an artin algebra over a commutative ring $k$ and let
$a$ be an idempotent element of $\Lambda$.  Let $e$ be the functor
$a-\colon \fmod{\Lambda}\lxr \fmod{a\Lambda a}$ given by
multiplication by $a$.  Consider the following conditions$\colon$
\begin{align*}
(\alpha)
\ \id_\Lambda \Big( \frac{\Lambda/\idealgenby{a}}
                         {\rad \Lambda/\idealgenby{a}} \Big)
&< \infty &
(\beta) \ \pd_{a{\Lambda}a} a{\Lambda} &< \infty \\
(\gamma)
\ \pd_\Lambda \Big( \frac{\Lambda/\idealgenby{a}}
                         {\rad \Lambda/\idealgenby{a}} \Big)
&< \infty &
(\delta) \ \pd_{\opposite{(a{\Lambda}a)}}{\Lambda}a &< \infty
\end{align*}
Then the following hold.
\begin{enumerate}
\item The following are equivalent$\colon$
\begin{enumerate}
\item $(\alpha)$ and $(\beta)$ hold.
\item $(\gamma)$ and $(\delta)$ hold.
\item The functor $e$ is an eventually homological isomorphism.
\end{enumerate}
\item The functor $a-\colon \fmod{\Lambda}\lxr \fmod{a\Lambda a}$
induces a singular equivalence between $\Lambda$ and $a\Lambda a$ if
and only if the conditions $(\beta)$ and $(\gamma)$ hold.
\item Assume that $e$ is an eventually homological isomorphism.  Then
$\Lambda$ is Gorenstein if and only if $a{\Lambda}a$ is Gorenstein.
\item Assume that $e$ is an eventually homological isomorphism, that
$k$ is a field and that $(\Lambda/\rad \Lambda) \tensor_k
(\opposite{\Lambda}/\rad \opposite{\Lambda})$ is a semisimple
$\envalg{\Lambda}$-module.  Then $\Lambda$ satisfies \fg{} if and
only if $a{\Lambda}a$ satisfies \fg{}.
\end{enumerate}
\end{thm}

This section is divided into three subsections.  In the first
subsection, we apply Theorem~\ref{thm:main-results-summary} to the
class of triangular matrix algebras.  In the second subsection, we
consider some cases where the conditions $(\alpha)$--$(\delta)$ in
Theorem~\ref{thm:main-results-summary} are related.  As a consequence,
we find sufficient conditions, stated in terms of the quiver and
relations, for applying Theorem~\ref{thm:main-results-summary} to a
quotient of a path algebra.  In the last subsection, we compare our
work to that of Nagase in \cite{N}.

\subsection{Triangular Matrix Algebras}

Let $\Sigma$ and $\Gamma$ be two artin algebras over a commutative
ring $k$, and let ${_{\Gamma}M_{\Sigma}}$ be a
$\Gamma$-$\Sigma$-bimodule such that $M$ is finitely generated over
$k$, and $k$ acts centrally on $M$.  Then we have the artin triangular
matrix algebra
\[
\Lambda = \begin{pmatrix}
           \Sigma & 0 \\
           {_\Gamma M_\Sigma} & \Gamma \\
         \end{pmatrix},
\]
where the addition and the multiplication are given by the ordinary
operations on matrices.

The module category of $\Lambda$ has a well known description, see
\cite{ARS, FGR}.  In fact, a module over $\Lambda$ is described as a
triple $(X,Y,f)$, where $X$ is a $\Sigma$-module, $Y$ is a
$\Gamma$-module and $f\colon M\tensor_{\Sigma}X\lxr Y$ is a
$\Gamma$-homomorphism.  A morphism between two triples $(X,Y,f)$ and
$(X',Y',f')$ is a pair of homomorphisms $(a,b)$, where $a\in
\Hom_{\Sigma}(X,X')$ and $b\in \Hom_{\Gamma}(Y,Y')$, such that the
following diagram commutes$\colon$
\[
\xymatrix{
 M\tensor_{\Sigma}X \ar[r]^{ \ \ \ f} \ar[d]_{\Id_M\tensor a} & Y  \ar[d]^{b} \\
 M\tensor_{\Sigma}X'\ar[r]^{ \ \ \ f'}   &  Y'  
}
\]
We define the following functors$\colon$  
\begin{enumerate}
\item The functor $T_{\Sigma}\colon \fmod{\Sigma}\lxr \fmod{\Lambda}$
is defined on $\Sigma$-modules $X$ by
$T_{\Sigma}(X)=(X,M\tensor_{\Sigma}X,\Id_{M\tensor X})$ and given a
$\Sigma$-homomorphism $a\colon X\lxr X'$ then
$T_{\Sigma}(a)=(a,\Id_{M}\tensor a)$.
\item The functor $U_{\Sigma}\colon \fmod{\Lambda}\lxr \fmod{\Sigma}$
is defined on $\Lambda$-modules $(X,Y,f)$ by $U_{\Sigma}(X,Y,f)=X$ and
given a $\Lambda$-homomorphism $(a,b)\colon (X,Y,f)\lxr (X',Y',f')$
then $U_{\Sigma}(a,b)=a$.  Similarly we define the functor
$U_{\Gamma}\colon \fmod{\Lambda}\lxr \fmod{\Gamma}$.
\item The functor $Z_{\Sigma}\colon \fmod{\Sigma}\lxr \fmod{\Lambda}$
is defined on $\Sigma$-modules $X$ by $Z_{\Sigma}(X)=(X,0,0)$ and
given a $\Sigma$-homomorphism $a\colon X\lxr X'$ then
$Z_{\Sigma}(a)=(a,0)$.  Similarly we define the functor
$Z_{\Gamma}\colon \fmod{\Gamma}\lxr \fmod{\Lambda}$.
\item The functor $H_{\Gamma}\colon\fmod{\Gamma}\lxr \fmod{\Lambda}$
is defined by $H_{\Gamma}(Y)=(\Hom_{\Gamma}(M,Y),Y,\epsilon_X)$ on
$\Gamma$-modules $Y$ and given a $\Gamma$-homomorphism $b\colon Y\lxr
Y'$ then $H_{\Gamma}(b)=(\Hom_{\Gamma}(M,b),b)$.
\end{enumerate}
Then from Example~\ref{exam:mod-recollements} (see also
\cite[Example~2.12]{Psaroud}), using the idempotent elements
$e_1=\bigl(\begin{smallmatrix}
  1_{\Sigma} & 0 \\
  0 & 0
\end{smallmatrix}\bigr)$ and $e_2=\bigl(\begin{smallmatrix}
0 & 0 \\
0 & 1_{\Gamma}
\end{smallmatrix}\bigr)$, we have the following recollements of
abelian categories$\colon$
\begin{equation}
\label{recolone}
\xymatrix@C=0.5cm{
\fmod{\Gamma} \ar[rrr]^{Z_{\Gamma}} &&& \fmod{\Lambda} \ar[rrr]^{U_{\Sigma}} \ar
@/_1.5pc/[lll]_{q}  \ar
 @/^1.5pc/[lll]^{U_{\Gamma}} &&& \fmod{\Sigma}
\ar @/_1.5pc/[lll]_{T_{\Sigma}} \ar
 @/^1.5pc/[lll]^{Z_{\Sigma}}
 } 
\end{equation}
and
\begin{equation}
\label{recoltwo}
\xymatrix@C=0.5cm{
\fmod{\Sigma} \ar[rrr]^{Z_{\Sigma}} &&& \fmod{\Lambda} \ar[rrr]^{U_{\Gamma}} \ar
@/_1.5pc/[lll]_{U_{\Sigma}}  \ar
 @/^1.5pc/[lll]^{p} &&& \fmod{\Gamma}
\ar @/_1.5pc/[lll]_{Z_{\Gamma}} \ar
 @/^1.5pc/[lll]^{H_{\Gamma}}
 } 
\end{equation}
The functors $q$ and $p$ are induced from the adjoint pairs
$(T_{\Sigma},U_{\Sigma})$ and $(U_{\Gamma},H_{\Gamma})$ respectively,
see \cite[Remark~2.3]{Psaroud} for more details.

We want to use Theorem~\ref{thm:main-results-summary} to compare the
triangular matrix algebra $\Lambda$ with the algebras $\Sigma$ and
$\Gamma$.  First consider the case where we compare $\Lambda$ with
$\Sigma$.  We then take the idempotent $a$ in the theorem to be $e_1$,
and we can reformulate the conditions $(\alpha)$, $(\beta)$,
$(\gamma)$ and $(\delta)$ as follows:
\begin{enumerate}
\item[$(\alpha)$] The functor $Z_\Gamma$ sends every $\Gamma$-module
  to a $\Lambda$-module with finite injective dimension.
\item[$(\beta)$] The functor $U_\Sigma$ sends every projective
  $\Lambda$-module to a $\Sigma$-module with finite projective
  dimension.
\item[$(\gamma)$] The functor $Z_\Gamma$ sends every $\Gamma$-module
  to a $\Lambda$-module with finite projective dimension.
\item[$(\delta)$] The functor $U_\Sigma$ sends every injective
  $\Lambda$-module to a $\Sigma$-module with finite injective
  dimension.
\end{enumerate}
By interchanging $\Sigma$ and $\Gamma$, we get a similar reformulation
of the conditions for the case where we compare $\Lambda$ with
$\Gamma$.

The next result clarifies when the above hold for the recollement
\eqref{recoltwo} of a triangular matrix algebra $\Lambda$.

\begin{lem}
\label{lemtriangular}
Let $\Lambda=\bigl(\begin{smallmatrix}
\Sigma & 0 \\
{_\Gamma M_\Sigma} & \Gamma
\end{smallmatrix}\bigr)$ be a triangular matrix algebra.  The
following hold.
\begin{enumerate}
\item If $\pd_{\Gamma}M<\infty$, then the functor $U_{\Gamma}$ sends
projective $\Lambda$-modules to $\Gamma$-modules of finite projective
dimension.
\item The functor $U_{\Gamma}$ preserves injectives.
\item Assume that $\gld{\Sigma}<\infty$.  Then
$\id_{\Lambda}Z_{\Sigma}(X)<\infty$ for every $\Sigma$-module $X$.
\item Assume that $\gld{\Sigma}<\infty$ and $\pd_{\Gamma}M<\infty$.
Then we have $\pd_{\Lambda}Z_{\Sigma}(X)<\infty$ for all
$\Sigma$-modules $X$.
\end{enumerate}
\begin{proof}
(i) It is known, see \cite{ARS}, that the indecomposable projective
$\Lambda$-modules are of the form $T_{\Sigma}(P)$, where $P$ is an
indecomposable projective $\Sigma$-module, and $Z_{\Gamma}(Q)$, where
$Q$ is an indecomposable projective $\Gamma$-module.  Hence it is
enough to consider modules of these forms.  We have
$U_{\Gamma}Z_{\Gamma}(Q)=Q$, and since $\pd_{\Gamma}M<\infty$ it
follows that
$\pd_{\Gamma}U_{\Gamma}T_{\Sigma}(P)=\pd_{\Gamma}(M\tensor_{\Sigma}P)<\infty$.

(ii) Since $(Z_{\Gamma}, U_{\Gamma})$ is an adjoint pair and
$Z_{\Gamma}$ is exact it follows that the functor $U_{\Gamma}$
preserves injectives.

(iii) Let $0\lxr X\lxr I^0\lxr \cdots\lxr I^n\lxr 0$ be a finite
injective resolution of a $\Sigma$-module $X$.  Then applying the
functor $Z_{\Sigma}$ we get the exact sequence $0\lxr
Z_{\Sigma}(X)\lxr Z_{\Sigma}(I^0)\lxr \cdots\lxr Z_{\Sigma}(I^n)\lxr
0$, where every $Z_{\Sigma}(I^i)$ is an injective $\Lambda$-module
since we have the adjoint pair $(U_{\Sigma},Z_{\Sigma})$ and
$U_{\Sigma}$ is exact.  Hence the injective dimension of
$Z_{\Sigma}(X)$ is finite.

(iv) This follows from \cite[Lemma~2.4]{triangular} since a
$\Lambda$-module $(X,Y,f)$ has finite projective dimension if and only
if the projective dimensions of $X$ and $Y$ are finite.
\end{proof}
\end{lem}

Using now the recollement \eqref{recolone} we have the following dual
result of Lemma~\ref{lemtriangular}.  The proof is left to the reader.

\begin{lem}
\label{lemtriangular2}
Let $\Lambda=\bigl(\begin{smallmatrix}
\Sigma & 0 \\
{_\Gamma M_\Sigma} & \Gamma
\end{smallmatrix}\bigr)$ be a triangular matrix algebra.  The
following hold.
\begin{enumerate}
\item The functor $U_{\Sigma}$ preserves projectives.
\item If $\pd_{\Sigma} M_{\Sigma}<\infty$, then the functor
$U_{\Sigma}$ sends injective $\Lambda$-modules to $\Sigma$-modules of
finite injective dimension.
\item Assume that $\gld{\Gamma}<\infty$.  Then
$\pd_{\Lambda}Z_{\Gamma}(Y)<\infty$ for every $\Gamma$-module $Y$.
\item Assume that $\gld{\Gamma}<\infty$ and $\pd_{\Sigma}
M_{\Sigma}<\infty$.  Then for every $\Gamma$-module $Y$ we have
$\id_{\Lambda}Z_{\Gamma}(Y)<\infty$.
\end{enumerate}
\end{lem}

As a consequence of Lemma~\ref{lemtriangular} and
Theorem~\ref{thm:main-results-summary} we have the following
result.  For similar characterizations with (ii) see
\cite{XiongZhang}.

\begin{cor}
\label{cortriang}
Let $\Lambda=\bigl(\begin{smallmatrix}
\Sigma & 0 \\
{_{\Gamma}M_{\Sigma}} & \Gamma
\end{smallmatrix}\bigr)$ be an artin triangular matrix algebra over a
commutative ring $k$ such
that $\gld{\Sigma}<\infty$ and $\pd_{\Gamma}M<\infty$.  Then the
following hold.
\begin{enumerate}
\item The singularity categories of $\Lambda$ and $\Gamma$ are triangle equivalent$\colon$
\[
\xymatrix@C=0.5cm{
 \Dsg(U_{\Gamma})\colon \Dsg(\fmod{\Lambda})  \ \ar[rr]^{  \ \ \ \ \equivalence} &&
 \ \Dsg(\fmod{\Gamma})  }
\]
\item $\Lambda$ is Gorenstein if and only if $\Gamma$ is Gorenstein.
\item Assume that $k$ is a field and that $(\Lambda/\rad \Lambda)
\tensor_k (\opposite{\Lambda}/\rad \opposite{\Lambda})$ is a
semisimple $\envalg{\Lambda}$-module.  Then $\Lambda$ satisfies
\fg{} if and only if $\Gamma$ satisfies \fg{}.
\end{enumerate}
\end{cor}

\begin{rem}
The algebra $(\Lambda/\rad \Lambda) \tensor_k (\opposite{\Lambda}/\rad
\opposite{\Lambda})$ being semisimple (as required in part (iii)
above) can be shown to be equivalent to the following three algebras
being semisimple: $(\Sigma/\rad \Sigma) \tensor_k
(\opposite{\Sigma}/\rad \opposite{\Sigma})$, $(\Sigma/\rad \Sigma)
\tensor_k (\opposite{\Gamma}/\rad \opposite{\Gamma})$ and
$(\Gamma/\rad \Gamma) \tensor_k (\opposite{\Gamma}/\rad
\opposite{\Gamma})$.
\end{rem}

We also have the following consequence, obtained now from
Lemma~\ref{lemtriangular2} and Theorem~\ref{thm:main-results-summary}.
Note that in the first statement we recover a theorem by Chen
\cite{Chen:schurfunctors}.

\begin{cor}
Let  $\Lambda=\bigl(\begin{smallmatrix}
\Sigma & 0 \\
{_{\Gamma}M_{\Sigma}} & \Gamma
\end{smallmatrix}\bigr)$ be an artin triangular matrix algebra over a
commutative ring $k$.
\begin{enumerate}
\item \cite[Theorem~4.1]{Chen:schurfunctors} Assume that
$\gld{\Gamma}<\infty$.  Then there is a triangle equivalence$\colon$
\[
\xymatrix@C=0.5cm{
   \Dsg(\fmod{\Lambda})
   \ \ar[rr]^{ \equivalence \ }_{\Dsg(U_{\Sigma}) \ } &&
   \ \Dsg(\fmod{\Sigma})  }
\]
\item Assume that $\gld{\Gamma}<\infty$ and $\pd_{\Sigma}
M_{\Sigma}<\infty$.  Then the following hold.
\begin{enumerate}
\item $\Lambda$ is Gorenstein if and only if $\Sigma$ is Gorenstein.
\item Assume that $k$ is a field and that $(\Lambda/\rad \Lambda)
\tensor_k (\opposite{\Lambda}/\rad \opposite{\Lambda})$ is a
semisimple $\envalg{\Lambda}$-module.  Then $\Lambda$ satisfies
\fg{} if and only if $\Sigma$ satisfies \fg{}.
\end{enumerate}
\end{enumerate}
\end{cor}

From the above corollaries and the classical result of
Buchweitz--Happel (see the text before Corollary~\ref{corgorsingular})
we have the following result for stable categories of Cohen--Macaulay
modules.

\begin{cor}
\label{cor:matrix-algebra-cm}
Let  $\Lambda=\bigl(\begin{smallmatrix}
\Sigma & 0 \\
{_{\Gamma}M_{\Sigma}} & \Gamma
\end{smallmatrix}\bigr)$ be an artin triangular matrix algebra.
\begin{enumerate}
\item \cite[Corollary~4.2]{Chen:schurfunctors} Assume that
$\gld{\Gamma}<\infty$ and $\Sigma$ is Gorenstein.  Then there is a
triangle equivalence$\colon$
\[
\xymatrix@C=0.5cm{
   \Dsg(\fmod{\Lambda})  \ \ar[rr]^{ \ \equivalence } &&
   \ \uCM(\Sigma) }
\]
\item Assume that $\gld{\Gamma}<\infty$ and $\pd_{\Sigma}
M_{\Sigma}<\infty$.  If $\Sigma$ is Gorenstein, then there is a
triangle equivalence between the stable categories of Cohen--Macaulay
modules of $\Lambda$ and $\Sigma$$\colon$
\[
\xymatrix@C=0.5cm{
 \uCM(\Lambda)  \ \ar[rr]^{ \equivalence} &&
 \ \uCM(\Sigma)  }
\]
\item Assume that $\gld{\Sigma}<\infty$ and $\pd_{\Gamma}M<\infty$.
If $\Gamma$ is Gorenstein, then there is a triangle equivalence
between the stable categories of Cohen--Macaulay modules of $\Lambda$
and $\Gamma$$\colon$
\[
\xymatrix@C=0.5cm{
 \uCM(\Lambda)  \ \ar[rr]^{ \equivalence} &&
 \ \uCM(\Gamma)  }
\]
\end{enumerate}
\end{cor}

\subsection{Algebras with ordered simples}

In this subsection, we apply Theorem~\ref{thm:main-results-summary} to
cases where there exists a total order $\leS$ of the simple
$\Lambda/\idealgenby{a}$-modules with the property that
\begin{equation}
\label{eqn:simple-order}
S \leS S'
\implies
\Ext_\Lambda^{>0}(S,S')=0
\end{equation}
for every pair $S$ and $S'$ of simple
$\Lambda/\idealgenby{a}$-modules.  With this assumption, we show that
we have the implications $(\alpha) \implies (\delta)$ and $(\gamma)
\implies (\beta)$ between the conditions in
Theorem~\ref{thm:main-results-summary}.  We then consider some special
cases where such orderings appear.

We need the following preliminary results.

\begin{lem}
\label{lem:simples-proj-res}
Let $\Lambda$ be an artin algebra, let $M$ be a $\Lambda$-module with
minimal projective resolution $\cdots \lxr P_1 \lxr P_0 \lxr M \lxr 0$,
and let $S$ be a simple $\Lambda$-module.  Then, for every nonnegative
integer $n$, we have $\Ext_\Lambda^n(M, S) = 0$ if and only if the
projective cover of $S$ is not a direct summand of $P_n$.
\end{lem}

\begin{lem}
\label{lem:simples-in-module-recollement}
Let $\Lambda$ be an artin algebra, and let $a$ be an idempotent in
$\Lambda$.  Let $S$ be a simple $\Lambda$-module which is not
annihilated by the ideal $\idealgenby{a}$, and let $P$ be the
projective cover of $S$.  Then $aP$ is a projective
$a{\Lambda}a$-module.
\end{lem}
\begin{proof}
We have
\[
\Hom_\Lambda({\Lambda}a, S)
\iso aS
\ne 0,
\]
so there exists a nonzero morphism $f \colon {\Lambda}a \lxr S$.
Decomposing the idempotent $a$ into a sum $a = a_1 + \cdots + a_t$ of
orthogonal primitive idempotents gives a decomposition ${\Lambda}a
\iso {\Lambda}a_1 \dsum \cdots \dsum {\Lambda}a_t$ of ${\Lambda}a$
into indecomposable projective modules.  For some $i$, we must then
have a nonzero morphism $f_i \colon {\Lambda}a_i \lxr S$.  Since $S$ is
simple, this means that ${\Lambda}a_i$ is its projective cover.  Since
$a \cdot a_i = a_i$, we get
\[
aP
\iso a{\Lambda}a_i
= (a{\Lambda}a) a_i.
\]
Therefore $aP$ is a projective $a{\Lambda}a$-module.
\end{proof}

Now we show that the conditions of
Theorem~\ref{thm:main-results-summary} are related when we have an
ordering of the simple $\Lambda/\idealgenby{a}$-modules.

\begin{prop}
\label{prop:ordered-simples}
Let $\Lambda$ be an artin algebra, and let $a$ be an idempotent in
$\Lambda$.  Assume that there is a total order $\leS$ on the simple
$\Lambda/\idealgenby{a}$-modules satisfying
condition~\eqref{eqn:simple-order}.  Then we have the following
implications between the conditions of
\textup{Theorem~\ref{thm:main-results-summary}}$\colon$
\begin{enumerate}
\item $(\alpha) \implies (\delta)$.
\item $(\gamma) \implies (\beta)$.
\end{enumerate}
\end{prop}
\begin{proof}
We show the second implication; the first can be showed in a similar
way.  Assume that $(\gamma)$ holds, that is, every
$\Lambda/\idealgenby{a}$-module has finite projective dimension as a
$\Lambda$-module.  We want to show that $(\beta)$ holds, that is, the
$a{\Lambda}a$-module $a\Lambda$ has finite projective dimension.

As in Section~\ref{section:fg-a(Lambda)a}, we let $e$ be the exact
functor $e=(a-)\colon \fmod \Lambda \lxr \fmod a{\Lambda}a$ given by
multiplication with $a$.  Then what we need to show is that
$e(\Lambda)$ has finite projective dimension as $a{\Lambda}a$-module.

Let $S_1 \leS \cdots \leS S_s$ be all the simple
$\Lambda/\idealgenby{a}$-modules (up to isomorphism), ordered by the
total order $\leS$.  Let $T_1, \ldots, T_t$ be all the other simple
$\Lambda$-modules (up to isomorphism).  Let $Q_i$ be the projective
cover of $S_i$ (considered as $\Lambda$-module) and $Q'_j$ the
projective cover of $T_j$, for every $i$ and $j$.  These are all the
indecomposable projective $\Lambda$-modules up to isomorphism, so it
is sufficient to show that $e(Q_i)$ and $e(Q'_j)$ have finite
projective dimension as $a{\Lambda}a$-modules for every $i$ and $j$.

For each of the modules $Q'_j$, we have that $e(Q'_j)$ is a projective
$a{\Lambda}a$-module by Lemma~\ref{lem:simples-in-module-recollement}.
We need to check that $e(Q_i)$ has finite projective dimension for
every $i$.

Consider the module $S_1$.  By our assumptions,
every simple $\Lambda/\idealgenby{a}$-module has finite projective
dimension over $\Lambda$.  Let
\[
\xymatrix{
0 \ar[r]^{} &
P^{(1)}_{n_1} \ar[r]^{} &
\cdots \ar[r]^{} &
P^{(1)}_2 \ar[r]^{} &
P^{(1)}_1 \ar[r]^{} &
Q_1 \ar[r]^{} &
S_1 \ar[r] &
0
}
\]
be a minimal projective resolution of $S_1$.  Applying the functor $e$
to this sequence gives the exact sequence
\begin{equation}
\label{eqn:proj-res-e(Q_1)}
\xymatrix{
0 \ar[r]^{} &
e(P^{(1)}_{n_1}) \ar[r]^{} &
\cdots \ar[r]^{} &
e(P^{(1)}_2) \ar[r]^{} &
e(P^{(1)}_1) \ar[r]^{} &
e(Q_1) \ar[r]^{} &
0
}
\end{equation}
of $a{\Lambda}a$-modules, since $e(S_1) = 0$.  Since we have
$\Ext_\Lambda^{>0}(S_1,S_i) = 0$ for every $i$, it follows from
Lemma~\ref{lem:simples-proj-res} that the only indecomposable
projective $\Lambda$-modules which can occur as direct summands of the
modules $P^{(1)}_1, \ldots, P^{(1)}_{n_1}$ are the modules $Q'_j$.
Since we know that these are mapped to projective modules by $e$, the
sequence~\eqref{eqn:proj-res-e(Q_1)} is a projective resolution of the
$a{\Lambda}a$-module $e(Q_1)$.

We continue inductively.  For every $i$, we apply the functor $e$ to a
minimal projective resolution
\[
\xymatrix{
0 \ar[r]^{} &
P^{(i)}_{n_i} \ar[r]^{} &
\cdots \ar[r]^{} &
P^{(i)}_2 \ar[r]^{} &
P^{(i)}_1 \ar[r]^{} &
Q_i \ar[r]^{} &
S_i \ar[r] &
0
}
\]
and obtain the sequence
\[
\xymatrix{
0 \ar[r]^{} &
e(P^{(i)}_{n_i}) \ar[r]^{} &
\cdots \ar[r]^{} &
e(P^{(i)}_2) \ar[r]^{} &
e(P^{(i)}_1) \ar[r]^{} &
e(Q_i) \ar[r]^{} &
0
}
\]
of $a{\Lambda}a$-modules.  Each of the modules $P^{(i)}_1, \ldots,
P^{(i)}_{n_i}$ has only the indecomposable projective modules $Q'_1,
\ldots, Q'_t, Q_1, \ldots, Q_{i-1}$ as direct summands.  Therefore (by the
induction assumption), all the modules $e(P^{(i)}_1), \ldots,
e(P^{(i)}_{n_i})$ have finite projective dimension, and thus the
module $e(Q_i)$ has finite projective dimension.
\end{proof}

The following example shows that the implications $(\alpha) \implies
(\beta)$ and $(\gamma) \implies (\delta)$ of the above proposition do
not hold in general.

\begin{exam}
Let $k$ be a field.  Let the $k$-algebra $\Lambda =
kQ/\idealgenby{\rho}$ be given by the following quiver and
relations$\colon$
\[
Q\colon
\xymatrix{
{1} \ar@/^/[r]^\alpha &
{2} \ar@/^/[l]^\beta
}
\qquad\qquad
\rho = \{ \alpha\beta \}.
\]
Let $a = e_1$.  Let $S_2$ be the simple $\Lambda$-module associated to
the vertex $2$.  Then we have $\pd_\Lambda S_2 = 2$ and $\id_\Lambda
S_2 = 2$, but $\pd_{a{\Lambda}a} a{\Lambda} = \infty$ and
$\id_{\opposite{(a{\Lambda}a)}} {\Lambda}a = \infty$.
\end{exam}

By combining Theorem~\ref{thm:main-results-summary} with
Proposition~\ref{prop:ordered-simples}, we get the following result.

\begin{cor}
\label{cor:ordered-simples}
Let $\Lambda$ be an artin algebra over a commutative ring $k$, and let
$a$ be an idempotent in $\Lambda$.  Assume that there is a total order
$\leS$ on the simple $\Lambda/\idealgenby{a}$-modules satisfying
condition~\eqref{eqn:simple-order}.  Then the following hold, where
$(\alpha)$, $(\beta)$, $(\gamma)$ and $(\delta)$ refer to the
conditions in \textup{Theorem~\ref{thm:main-results-summary}}.
\begin{enumerate}
\item The functor $a-\colon \fmod{\Lambda}\lxr \fmod{a\Lambda a}$
induces a singular equivalence between $\Lambda$ and $a\Lambda a$ if
and only if $(\gamma)$ holds.
\item Assume that $(\alpha)$ and $(\gamma)$ hold.  Then $\Lambda$ is
Gorenstein if and only if $a{\Lambda}a$ is Gorenstein.
\item Assume that $(\alpha)$ and $(\gamma)$ hold, that $k$ is a field
and $(\Lambda/\rad \Lambda) \tensor_k (\opposite{\Lambda}/\rad
\opposite{\Lambda})$ is a semisimple $\envalg{\Lambda}$-module.  Then
$\Lambda$ satisfies \fg{} if and only if $a{\Lambda}a$ satisfies
\fg{}.
\end{enumerate}
\end{cor}

We now consider special cases of the conditions $(\alpha)$ and
$(\gamma)$ where the dimensions are not only finite, but at most one.
We show that if one of these dimensions is at most one, then we have
an ordering of the simple $\Lambda/\idealgenby{a}$-modules as assumed
in Proposition~\ref{prop:ordered-simples} and
Corollary~\ref{cor:ordered-simples}.

\begin{lem}
\label{lem:pd,id<=1}
Let $\Lambda$ be an artin algebra, and let $a$ be an idempotent in
$\Lambda$.  Assume that we have either
\[
(\alpha_1)
\ \id_\Lambda \Big( \frac{\Lambda/\idealgenby{a}}{\rad \Lambda/\idealgenby{a}} \Big)
 \le 1
\qquad\text{or}\qquad
(\gamma_1)
\ \pd_\Lambda \Big( \frac{\Lambda/\idealgenby{a}}{\rad \Lambda/\idealgenby{a}} \Big)
 \le 1
\]
Then there exists a total order $\leS$ on the simple
$\Lambda/\idealgenby{a}$-modules satisfying
condition~\eqref{eqn:simple-order}.
\end{lem}
\begin{proof}
Assume that $(\gamma_1)$ holds (the proof using $(\alpha_1)$ is similar).
Let $S_1, \ldots, S_s$ be all the simple
$\Lambda/\idealgenby{a}$-modules (up to isomorphism), and let $P_1,
\ldots, P_s$ be their projective covers as $\Lambda$-modules, such
that $P_i/((\rad \Lambda)P_i) \iso S_i$ for every $i$.  Assume that we
have ordered these by increasing length of the projective covers, that
is,
\[
\operatorname{length}(P_1) \le
\operatorname{length}(P_2) \le
\cdots \le
\operatorname{length}(P_s).
\]
For any $i$, the module $S_i$ has a projective resolution of the form
\[
\xymatrix{
0 \ar[r]^{} & Q \ar[r]^{} & P_i \ar[r]^{} & S_i \ar[r]^{} & 0  }
\]
Since the module $Q$ has shorter length than the module $P_i$, it can
not have any of the modules $P_i, \ldots, P_s$ as direct summands.  Then
Lemma~\ref{lem:simples-proj-res} implies that
$\Ext_\Lambda^{>0}(S_i, S_j) = 0$ for $i \le j$.
\end{proof}

By using Proposition~\ref{prop:ordered-simples},
Lemma~\ref{lem:pd,id<=1} and Theorem~\ref{thm:main-results-summary},
we have the following.

\begin{cor}
\label{cor:pd,id<=1}
Let $\Lambda$ be an artin algebra over a commutative ring $k$, and let
$a$ be an idempotent in $\Lambda$.  Then the following hold, where
$(\alpha)$, $(\beta)$, $(\gamma)$ and $(\delta)$ refer to the
conditions in \textup{Theorem~\ref{thm:main-results-summary}}, and $(\alpha_1)$
and $(\gamma_1)$ refer to the conditions in \textup{Lemma~\ref{lem:pd,id<=1}}.
\begin{enumerate}
\item If $(\gamma_1)$ holds, then the singularity categories of
$\Lambda$ and $a{\Lambda}a$ are triangle equivalent.
\item Assume either that $(\alpha_1)$ and $(\gamma)$ hold, or that
$(\alpha)$ and $(\gamma_1)$ hold.  Then $\Lambda$ is Gorenstein if and
only if $a{\Lambda}a$ is Gorenstein.
\item Assume either that $(\alpha_1)$ and $(\gamma)$ hold, or that
$(\alpha)$ and $(\gamma_1)$ hold.  Furthermore, assume that $k$ is a
field and $(\Lambda/\rad \Lambda) \tensor_k (\opposite{\Lambda}/\rad
\opposite{\Lambda})$ is a semisimple $\envalg{\Lambda}$-module.  Then
$\Lambda$ satisfies \fg{} if and only if $a{\Lambda}a$ satisfies
\fg{}.
\end{enumerate}
\end{cor}

For the following results, we let $\Lambda = kQ/\idealgenby{\rho}$ be
a quotient of a path algebra, where $k$ is a field, $Q$ is a quiver,
and $\rho$ a minimal set of relations in $kQ$ generating an admissible
ideal $\idealgenby{\rho}$.

First we describe how the conditions $(\alpha_1)$ and $(\gamma_1)$ can
be interpreted for quotients of path algebras.  The result follows
directly from \cite[Corollary, Section~1.1]{Bongartz}.

\begin{lem}
\label{lem:relations-dimensions}
Let $S$ be the simple $\Lambda$-module corresponding to a vertex $v$
in the quiver $Q$.
\begin{enumerate}
\item We have $\pd_\Lambda S \le 1$ if and only if no relation starts
in the vertex $v$.
\item We have $\id_\Lambda S \le 1$ if and only if no relation ends in
the vertex $v$.
\end{enumerate}
\end{lem}

As a consequence of Lemma~\ref{lem:relations-dimensions} and
Corollary~\ref{cor:pd,id<=1}, we get the following results for path
algebras.

\begin{cor}
\label{cor:removing-vertices-1}
Let $\Lambda = kQ/\idealgenby{\rho}$ be a quotient of a path algebra
as above.  Choose some vertices in $Q$ where no relations start, and
let $a$ be the sum of all vertices except these.  Then the functor
$a-\colon \fmod \Lambda \lxr \fmod a{\Lambda}a$ induces a singular
equivalence between $\Lambda$ and $a{\Lambda}a\colon$
\[
\Dsg(a-)\colon
 \Dsg(\fmod \Lambda)
 \stackrel{\equivalence}{\lxr}
 \Dsg(\fmod a{\Lambda}a)
\]
\end{cor}

\begin{cor}
\label{cor:removing-vertices-2}
Let $\Lambda = kQ/\idealgenby{\rho}$ be a quotient of a path algebra
as above.  Choose some vertices in $Q$ where no relations start and no
relations end, and let $a$ be the sum of all vertices except these.
Then the following hold$\colon$
\begin{enumerate}
\item $\Lambda$ is Gorenstein if and only if $a{\Lambda}a$ is
Gorenstein.
\item $\Lambda$ satisfies \fg{} if and only if $a{\Lambda}a$ satisfies
\fg{}.
\end{enumerate}
\end{cor}

We apply the above result in the following example.

\begin{exam}
Let $Q$ be the quiver with relations $\rho$ given by
\[
Q\colon
\xymatrix@C=3em{
1 \ar[r]^{\alpha_1} &
2 \ar[r]^-{\alpha_2} &
{\cdots} \ar[r]^{\alpha_{m-1}} &
m \ar@(dl,dr)[lll]^{\alpha_m}
}
\qquad\text{and}\qquad
\rho=\{(\alpha_m \cdots \alpha_1)^n\},
\]
for some integers $m \ge 2$ and $n \ge 2$.  Let $\Lambda =
kQ/\idealgenby{\rho}$, and let $a = e_1$ (the only vertex where a
relation starts and ends).  Then $a{\Lambda}a \iso
k[x]/\idealgenby{x^n}$, so $a{\Lambda}a$ satisfies \fg{} by
\cite{EH,EHS}.
By Corollary~\ref{cor:removing-vertices-2}, the algebra $\Lambda$ also
satisfies \fg{}.  By Corollary~\ref{cor:removing-vertices-1}, the
algebras $\Lambda$ and $k[x]/\idealgenby{x^n}$ are singularly
equivalent.  See \cite{SO} for a general discussion of the Hochschild
cohomology ring of the path algebra $kQ$ modulo one relation.
\end{exam}

\subsection{Comparison to work by Nagase}
\label{subsection:nagase}

In this subsection we recall a result of Hiroshi Nagase \cite{N} and
relate his set of assumptions to ours.

In \cite{N} Hiroshi Nagase proves the following result.
\begin{prop}
  Let $\Lambda$ be a finite dimensional algebra over an algebraically
  closed field with a stratifying ideal $\idealgenby{a}$ for an
  idempotent $a$ in $\Lambda$.  Suppose $\pd_{\envalg{\Lambda}}
  \Lambda/\idealgenby{a} < \infty$.  Then we have
\begin{enumerate}
\item[(1)] $\HH{\ge n}(\Lambda)\iso \HH{\ge n}(a\Lambda a)$ as graded
algebras, where $n = \pd_{\envalg{\Lambda}} \Lambda/\idealgenby{a} +
1$.
\item[(2)] $\Lambda$ satisfies \fg{} if and only if so does $a\Lambda
  a$.
\item[(3)] $\Lambda$ is Gorenstein if and only if so is $a\Lambda a$.
\end{enumerate}
\end{prop}
This work is based on the paper \cite{Kon-Nag}, where stratifying
ideals $\idealgenby{a}$ in a finite dimensional algebra $\Lambda$
were used to show that the Hochschild cohomology groups of $\Lambda$
and $a\Lambda a$ are isomorphic in almost all degrees.

We start by giving an example of a recollement $(\fmod
\Lambda/\idealgenby{a}, \fmod \Lambda, \fmod a\Lambda a)$, where the
ideal $\idealgenby{a}$ is not a stratifying ideal but it satisfies our
conditions from Theorem \ref{thm:main-result-fg}.

\begin{exam}
Let $Q$ be the quiver with relations $\rho$ given by
\[\xymatrix@R=5pt{
  & 2\ar[dd]^\gamma \\
1\ar@(ul,dl)_{\alpha}\ar[ur]^\beta& \\
  & 3\ar[ul]^\delta }\]
and $\rho=\{\alpha^2, \gamma\beta, \beta\alpha\delta\}$.  Let $\Lambda
= kQ/\idealgenby{\rho}$ for some field $k$, and let $a = e_1$. We
want to study the relationship between $\Lambda$ and $a\Lambda a$.  Let
$S_i$ denote the simple $\Lambda$-module associated to the vertex $i$
for $i=1,2,3$.  Then $\pd_\Lambda S_2 = 1$, $\pd_\Lambda S_3 = 3$,
$\id_\Lambda S_2 = 2$ and $\id_\Lambda S_3 = 3$.  Furthermore, the left
and right $a\Lambda a$-module $a\Lambda$ and $\Lambda a$ have finite
projective dimension (they are projective) as $a\Lambda
a$-modules.  Hence, according to Theorem \ref{thm:main-result-fg}
$\Lambda$ satisfies \fg{} if and only if $a\Lambda a\iso k[x]/\idealgenby{x^2}$
does.  We infer from this that $\Lambda$ satisfies \fg{}. Moreover, the
Hochschild cohomology groups of $\Lambda$ and $a\Lambda a$ are
isomorphic in almost all degrees by Proposition
\ref{prop:lambda-ext-iso}.

We claim that $\idealgenby{a}$ is not a stratifying ideal. Recall that
$\idealgenby{a}$ is stratifying if (i) the multiplication map $\Lambda
a\tensor_{a\Lambda a} a\Lambda \lxr\Lambda a \Lambda$ is an isomorphism
and (ii) $\Tor^{a\Lambda a}_i(\Lambda a,a\Lambda) = (0)$ for $i>0$.
Using that $(1-a)\Lambda a\iso a\Lambda a$ as a right $a\Lambda
a$-module, direct computations show that $\Lambda a\tensor_{a\Lambda
  a} a\Lambda$ has dimension $12$, while $\idealgenby{a}$ has
dimension $10$.  Consequently $\idealgenby{a}$ is not a stratifying
ideal in $\Lambda$.  However, the condition (ii) is satisfied since
$\Lambda a$ is a projective $a\Lambda a$-module.
\end{exam}

Next we show that, when $\idealgenby{a}$ is a stratifying ideal, then
the property $\pd_{\envalg{\Lambda}} \Lambda/\idealgenby{a}< \infty$
is equivalent to the functor $e \colon \fmod \Lambda \lxr \fmod
a{\Lambda}a$ being an eventually homological isomorphism.  We thank
Hiroshi Nagase for pointing out that (a) implies (b) in the second
part of the following result.  This led to a much better understanding
of the conditions occurring in the main results.

\begin{lem}
\label{lem:stratifying}
Let $\Lambda$ be a finite dimensional algebra over an algebraically
closed field $k$.
\begin{enumerate}
\item\label{lem:stratifying:pd} Assume that $(\alpha)\ \id_\Lambda
  \Big( \frac{\Lambda/\idealgenby{a}}{\rad \Lambda/\idealgenby{a}}
  \Big) < \infty$ and $(\gamma)\ \pd_\Lambda \Big(
  \frac{\Lambda/\idealgenby{a}}{\rad \Lambda/\idealgenby{a}} \Big) <
  \infty$.  Then $\pd_{\envalg{\Lambda}} \Lambda/\idealgenby{a} < \infty$.
\item\label{lem:stratifying:equiv} Assume that $\idealgenby{a}$ is a
  stratifying ideal in $\Lambda$. Then the following are equivalent.
\begin{enumerate}
\item $\pd_{\envalg{\Lambda}} \Lambda/\idealgenby{a} < \infty$.
\item The functor $e \colon \fmod \Lambda \lxr \fmod a{\Lambda}a$ is an
eventually homological isomorphism.
\end{enumerate}
\end{enumerate}
\end{lem}
\begin{proof}
\ref{lem:stratifying:pd} For two primitive idempotents $u$ and $v$ in
$\Lambda$, we have that
\[\Hom_{\envalg{\Lambda}}(\envalg{\Lambda}(u\tensor v), \Lambda/\idealgenby{a})
\iso u(\Lambda/\idealgenby{a})v.\] Then, if $u$ or $v$ occurs in
$a$, then this homomorphism set is zero.  Consequently we infer that
the composition factors of $\Lambda/\idealgenby{a}$ are direct
summands of the semisimple module $\Big(
\frac{\Lambda/\idealgenby{a}}{\rad \Lambda/\idealgenby{a}} \Big)
\tensor_k \Big( \frac{\opposite{\Lambda}/\idealgenby{a}}{\rad
   \opposite{\Lambda}/\idealgenby{a}} \Big)$.  By Lemma
\ref{lem:tensor-preserves-fin-projdim}
$\pd_{\envalg{\Lambda}} \Big( \frac{\Lambda/\idealgenby{a}}{\rad
   \Lambda/\idealgenby{a}} \Big) \tensor_k \Big(
\frac{\opposite{\Lambda}/\idealgenby{a}}{\rad
\opposite{\Lambda}/\idealgenby{a}}
\Big)$ is finite, hence the claim follows.

\ref{lem:stratifying:equiv} By Corollary~\ref{cor:algebra-ext-iso} and
part \ref{lem:stratifying:pd}, statement (b) implies (a).  

Conversely, assume (a).  For $j> \pd
_{\envalg{\Lambda}}\Lambda/\idealgenby{a}$ and any $\Lambda$-modules
$M$ and $N$ we have that 
\[
\Ext^j_{\envalg{\Lambda}}(\Lambda,\Hom_k(M,N))
\iso \Ext^j_{\envalg{\Lambda}}(\idealgenby{a}, \Hom_k(M,N))
\]
Using the isomorphism in the proof of Proposition 3.3 in \cite{Kon-Nag},
\[
\Ext^i_{\envalg{\Lambda}}(\idealgenby{a},X)
\iso \Ext^i_{\envalg{a\Lambda a}}(a\Lambda a, aXa),
\]
we obtain that
\begin{align*}
\Ext^i_{\envalg{\Lambda}}(\idealgenby{a},\Hom_k(M,N)) 
&\iso \Ext^i_{\envalg{a\Lambda a}}(a\Lambda a, a\Hom_k(M,N)a)\\
&\iso \Ext^i_{\envalg{a\Lambda a}}(a\Lambda a, \Hom_k(aM,aN))\\
&\iso \Ext^i_{a\Lambda a}(aM,aN))
\end{align*}
for all $\Lambda$-modules $M$ and $N$.  Since
$\Ext^i_{\envalg{\Lambda}}(\Lambda,\Hom_k(M,N)) \iso
\Ext^i_{\Lambda}(M,N)$, we obtain the isomorphism
\[
\Ext_\Lambda^j (M, N) \iso \Ext_{a\Lambda a}^j (aM, aN)
\]
for all $j > \pd _{\envalg{\Lambda}}\Lambda/\idealgenby{a}$ and all
$\Lambda$-modules $M$ and $N$.  Hence $e$ is an eventually homological
isomorphism.
\end{proof}

The following result gives a characterization of the condition
$(\gamma)$ when $\idealgenby{a}$ is a stratifying ideal.

\begin{lem}
  Let $\Lambda$ be an artin algebra and $a$ an idempotent in
  $\Lambda$.  Assume that $\idealgenby{a}$ is a stratifying ideal in
  $\Lambda$. Then we have $(\gamma)\ \pd_\Lambda \Big(
  \frac{\Lambda/\idealgenby{a}}{\rad \Lambda/\idealgenby{a}} \Big) <
  \infty$ if and only if $\gld{\Lambda/\idealgenby{a}}<\infty$ and
  $\pd_{\Lambda} \idealgenby{a}<\infty$. Moreover, if $(\gamma)$
  holds, then $(\beta)$ holds. 
\begin{proof}
  Assume that $(\gamma)$ $\pd_\Lambda \Big(
  \frac{\Lambda/\idealgenby{a}}{\rad \Lambda/\idealgenby{a}} \Big) <
  \infty$.  It is clear that $\pd_\Lambda\idealgenby{a} < \infty$ if
  and only if $\pd_\Lambda \Lambda/\idealgenby{a} < \infty$.  Since
  $\Lambda/\idealgenby{a}$ as a $\Lambda$-module is filtered in simple
  modules occurring as direct summands in
  $(\Lambda/\idealgenby{a})/(\rad \Lambda/\idealgenby{a})$, we infer
  that $\pd_\Lambda \Lambda/\idealgenby{a} < \infty$ by the property
  $(\gamma)$. Since $\idealgenby{a}$ is a stratifying ideal in
  $\Lambda$, we have that 
\[\Ext^j_{\Lambda/\idealgenby{a}}(X,Y) \iso \Ext^j_\Lambda(X,Y)\]
for all $j\ge 0$ and all modules $X$ and $Y$ in
$\fmod\Lambda/\idealgenby{a}$.  Using the above isomorphism and the
property $(\gamma)$ again, we obtain that
$\id_{\Lambda/\idealgenby{a}} Y \leq \pd_\Lambda
(\Lambda/\idealgenby{a})/(\rad \Lambda/\idealgenby{a})$ for all $Y$ in
$\fmod\Lambda/\idealgenby{a}$. Hence $\gld\Lambda/\idealgenby{a} <
\infty$.

Assume conversely that $\gld{\Lambda/\idealgenby{a}}<\infty$ and
$\pd_{\Lambda} \idealgenby{a}<\infty$. From \cite[Theorem
$3.9$]{Psaroud} we have a finite projective resolution $0\lxr \Lambda
a\otimes_{a\Lambda a}Q_n\lxr \cdots \lxr \Lambda a\otimes_{a\Lambda
  a}Q_0\lxr \idealgenby{a}\lxr 0$, where $Q_i$ are projective
$a\Lambda a$-modules.  Then applying the exact functor $e=a-$, it
follows from Proposition \ref{properties} that the sequence $0\lxr
Q_n\lxr \cdots \lxr Q_0\lxr a(\idealgenby{a})\lxr 0$ is exact.  We
infer that $(\beta)\ \pd_{a\Lambda a} a\Lambda < \infty$, since
$a\idealgenby{a} \iso a\Lambda$. Since $\gld\Lambda/\idealgenby{a} <
\infty$ and $\pd_{\Lambda} \Lambda/\idealgenby{a}<\infty$, we have
that $\pd_{\Lambda}X\leq \pd_{\Lambda/\idealgenby{a}}X +
\pd_{\Lambda}\Lambda/\idealgenby{a}$.  We infer that $(\gamma)$
$\pd_\Lambda \Big( \frac{\Lambda/\idealgenby{a}}{\rad
  \Lambda/\idealgenby{a}} \Big) < \infty$ holds.

The last claim follows immediately from the above. 
\end{proof}
\end{lem}


\begin{thebibliography}{99}

\bibitem{AKL} \textsc{L.~Angeleri H\"ugel, S.~Koenig and Q.~Liu}, {\em
    Recollements and tilting objects}, J.\ Pure Appl.\ Algebra {\bf 215}
  (2011), 420--438.

\bibitem{AKL2} \textsc{L.~Angeleri H\"ugel, S.~Koenig and Q.~Liu},
  {\em On the uniqueness of stratifications of derived module
    categories}, J.\ Algebra {\bf 359} (2012), 120--137.

\bibitem{AKL3} \textsc{L.~Angeleri H\"ugel, S.~Koenig and Q.~Liu}, {\em
    Jordan-H\"older theorems for derived module categories of
    piecewise hereditary algebras}, J.\ Algebra {\bf 352} (2012),
  361--381.

\bibitem{AKLY4} \textsc{L.~Angeleri H\"ugel, S.~Koenig, Q.~Liu and
    D.~Yang}, {\em Derived simple algebras and restrictions of
    recollements of derived module categories}, arXiv:1310.3479
  (2013).

\bibitem{AR:applications}
\textsc{M.~Auslander and I.~Reiten}, \textit{Applications of Contravariantly Finite
subcategories}, Adv.\ in Math.\ {\bf 86} (1991), 111--152.

\bibitem{AR:cm}
\textsc{M.~Auslander and I.~Reiten}, \textit{Cohen-Macaulay and Gorenstein
Algebras}, Progress in Math.\ {\bf 95} (1991), 221--245.


\bibitem{ARS}
\textsc{M.~Auslander, I.~Reiten and S.~Smal{\o}}, \textit{Representation
Theory of Artin Algebras}, Cambridge University Press, (1995).

\bibitem{BBD}
\textsc{A.~Beilinson, J.~Bernstein and P.~Deligne}, {\em Faisceaux
Pervers}, (French) [Perverse sheaves], Analysis and topology on
singular spaces, I (Luminy, 1981), 5--171, Asterisque, {\bf 100},
Soc.\ Math.\ France, Paris, (1982).


\bibitem{BR}
\textsc{A.~Beligiannis and I.~Reiten}, {\em Homological and homotopical
aspects of torsion theories}, Mem.\ Amer.\ Math.\ Soc.\ {\bf 188}
(2007), no.\ 883, viii+207 pp.


\bibitem{Bongartz}
\textsc{K.~Bongartz},
{\em Algebras and quadratic forms},
J.\ London Math.\ Soc.\ (2) {\bf 28} (1983), no.\ 3, 461--469.


\bibitem{BKSS} 
\textsc{A.\ B.\ Buan, H.\ Krause, N.\ Snashall and \O.\ Solberg},
  \emph{Support varieties -- an axiomatic approach}, preprint 2013.

\bibitem{Buchweitz:unpublished}
\textsc{R.-O.~Buchweitz},
{\em Maximal Cohen-Macaulay modules and Tate-Cohomology over Gorenstein rings},
unpublished manuscript, (1987), 155 pp.

\bibitem{Ca} J.\ F., Carlson, \emph{The complexity and varieties of
    modules}, in Lecture Notes in Mathematics {\bf 882}, 415--422,
  Springer-Verlag, 1981.

\bibitem{CartanEilenberg}
\textsc{H.~Cartan and S.~Eilenberg}, {\em Homological algebra},
Princeton University Press, 1956.

\bibitem{Chen:schurfunctors}
\textsc{X.-W.~Chen},
{\em Singularity categories, Schur functors and triangular matrix rings},
Algebr.\ Represent.\ Theory, {\bf 12} (2009), 181--191.


\bibitem{Chen:tworesultsofOrlov}
\textsc{X.-W.~Chen},
{\em Unifying two results of Orlov on singularity categories},
Abh.\ Math.\ Semin.\ Univ.\ Hambg., {\bf 80} (2010), 207--212.

\bibitem{Chen:radicalsquarezero}
\textsc{X.-W.~Chen},
{\em The singularity category of an algebra with radical square zero},
Doc.\ Math., {\bf 16} (2011), 921--936.

\bibitem{Chen:relativesing}
\textsc{X.-W.~Chen},
{\em Relative singularity categories and Gorenstein-projective modules},
Math.\ Nachr., {\bf 284} (2011), 199--212.

\bibitem{Chen:Singular equivalences}
\textsc{X.-W.~Chen},
{\em Singular equivalences induced by homological epimorphisms},
arXiv:1107.5922, to appear in Proc.\ Amer.\ Math.\ Soc.

\bibitem{Chen:Singular equivalences-trivial extensions}
\textsc{X.-W.~Chen},
{\em Singular equivalences of trivial extensions},
arXiv:1110.5955.

\bibitem{CurtisReiner}
\textsc{C.~W.~Curtis, I.~Reiner},
{\em Methods of representation theory. Vol.\ I. With applications to finite groups and orders}.
Pure and Applied Mathematics. A Wiley-Interscience Publication.
John Wiley \& Sons, Inc., New York, 1981. xxi+819 pp.

\bibitem{Xi:1} 
\textsc{H.~Chen, C.C.~Xi},
{\em Good tilting modules and recollements of derived module categories},
Proc.\ London Math.\ Soc.\ (2012) {\bf 104} (5): 959-996.\ (2012).

\bibitem{Xi:2} 
\textsc{H.~Chen, C.C.~Xi},
{\em Homological ring epimorphisms and recollements I: Exact pairs},
arXiv:1203.5168 (2012).

\bibitem{Xi:3} 
\textsc{H.~Chen, C.C.~Xi},
{\em Homological ring epimorphisms and recollements II: Algebraic K-theory},
arXiv:1212.1879 (2012).


\bibitem{CPS}
\textsc{E.~Cline, B.~Parshall, and L.~Scott},
{\em Finite dimensional algebras and highest weight categories},
J.\ Reine Angew.\ Math.\ {\bf 391} (1988), 85--99.

\bibitem{EHSST}
\textsc{K.~Erdmann, M.~Holloway, N.~Snashall, \O{}.~Solberg and R.~Taillefer},
{\em Support varieties for selfinjective algebras},
K-Theory {\bf 33} (2004), 67--87.

\bibitem{EH}
\textsc{K.~Erdmann and T.~Holm},
{\em Twisted bimodules and Hochschild cohomology
  for self-injective algebras of class $A_n$},
Forum Math.\ {\bf 11} (1999), no.\ 2, 177--201.

\bibitem{EHS}
\textsc{K.~Erdmann, T.~Holm and N.~Snashall}
{\em Twisted bimodules and Hochschild cohomology
  for self-injective algebras of class $A_n$. II},
Algebr.\ Represent.\ Theory {\bf 5} (2002), no.\ 5, 457--482.

\bibitem{E}
\textsc{L.~Evens},
\emph{The cohomology ring of a finite group},
Trans.\ Amer.\ Math.\ Soc., \textbf{101} (1961), 224--239.

\bibitem{FGR}
\textsc{R.~Fossum, P.~Griffith and I.~Reiten},
{\em Trivial Extensions of Abelian Categories with Applications to Ring Theory},
Lecture Notes in Mathematics, vol.\ {\bf 456}, Springer, Berlin, (1975).

\bibitem{Go}
\textsc{E.~Golod},
\emph{The cohomology ring of a finite $p$-group},
Dokl.\ Akad.\ Nauk SSSR \textbf{125} (1959) 703--706.

\bibitem{Pira}
\textsc{V.~Franjou and T.~Pirashvili}, {\em Comparison of abelian
categories recollements}, Documenta Math.\ {\bf 9} (2004), 41--56.

\bibitem{Gerstenhaber:1}
\textsc{M.~Gerstenhaber},
{\em The cohomology structure of an associative ring},
Ann.\ of Math.\ {\bf 78}, (1963), 267--288.

\bibitem{GMM} E.\ L.\ Green, D.\ Madsen, E.\ Marcos, \emph{Private
    communication}.

\bibitem{Han}
\textsc{Y.~Han}, {\em Recollements and Hochschild theory}, J.\ Algebra {\bf 397} (2014) 535--547.

\bibitem{Happel}
\textsc{D.~Happel}, {\em Partial tilting modules and recollement},
Proceedings of the International Conference on Algebra, Part 2
(Novosibirsk, 1989), 345--361, Contemp.\ Math., {\bf 131}, Part 2,
Amer.\ Math.\ Soc., Providence, RI, 1992.

\bibitem{Happel:3}
\textsc{D.~Happel},
{\em On Gorenstein algebras}.
Representation theory of finite groups and finite-dimensional algebras (Bielefeld, 1991),
389--404, Progr.\ Math., {\bf 95}, Birkh\"auser, Basel, 1991.

\bibitem{Happel:4}
\textsc{D.~Happel},
{\em Triangulated categories in the representation theory of finite-dimensional algebras}.
London Mathematical Society Lecture Note Series, {\bf 119}.
Cambridge University Press, Cambridge, 1988. x+208 pp.

\bibitem{Hochschild}
\textsc{G.~Hochschild},
{\em On the cohomology groups of an associative algebra},
Ann.\ of Math.\ {\bf 46}, (1945), 58--67.

\bibitem{Keller:cyclic}
\textsc{B.~Keller},
{\em On the cyclic homology of exact categories},
J.\ Pure Appl.\ Algebra {\bf 136} (1999), 1--56.

\bibitem{Kon-Nag} 
\textsc{S.~K\"onig and H.~Nagase},
\emph{Hochschild cohomology and stratifying ideals},
J.\ Pure Appl.\ Algebra {\bf 213} (2009), no.\ 5, 886--891.


\bibitem{Kontsevich}
\textsc{M.~Kontsevich},
{\em Homological algebra of mirror symmetry},
Proceedings of ICM, Zurich 1994, Basel, 1995, Birkhauser, 120--139.


\bibitem{Krause:Localization}
\textsc{H.~Krause},
{\em Localization theory for triangulated categories}.
Triangulated categories, 161--235, London Math.\ Soc.\ Lecture Note Ser., {\bf 375},
Cambridge Univ.\ Press, Cambridge, 2010.

\bibitem{Kuhn}
\textsc{N.J.~Kuhn},
{\em The generic representation theory of finite fields: a survey of basic structure}.
Infinite length modules (Bielefeld 1998), 193--212, Trends Math., Birkhauser, Basel, 2000.

\bibitem{MV}
\textsc{R.~MacPherson and K.~Vilonen},
{\em Elementary construction of perverse sheaves},
Invent.\ Math.\ {\bf 84} (1986), 403--436.

\bibitem{Miyachi}
\textsc{J.I.~Miyachi},
{\em Localization of Triangulated Categories and Derived Categories},
J.\ Algebra {\bf 141} (1991) 463--483.

\bibitem{N} 
H.\ Nagase, \emph{Hochschild cohomology and Gorenstein
    Nakayama algebras}, (English summary) Proceedings of the 43rd
  Symposium on Ring Theory and Representation Theory, 37--41,
  Symp.\ Ring Theory Represent.\ Theory Organ.\ Comm., Soja, 2011.

\bibitem{gradedrings:book}
\textsc{C.~N\u{a}st\u{a}sescu and F.~Van Oystaeyen},
{\em Methods of graded rings},
Lecture Notes in Mathematics, {\bf 1836}. Springer-Verlag, Berlin, 2004. xiv+304 pp


\bibitem{Neeman:book} \textsc{A.~Neeman}, {\em Triangulated
    categories}. Annals of Mathematics Studies, {\bf 148}. Princeton
  University Press, Princeton, NJ, 2001. viii+449 pp.

\bibitem{Orlov} \textsc{D.O.~Orlov}, {\em Triangulated categories of
    singularities and D-branes in Landau-Ginzburg models},
  Tr.\ Mat.\ Inst.\ Steklova {\bf 246} (2004), Algebr.\ Geom.\ Metody,
  Svyazi i Prilozh., 240--262; translation in Proc.\ Steklov
  Inst.\ Math.\ {\bf 246} (2004), 227--248.

\bibitem{Orlov:2} \textsc{D.O.~Orlov}, {\em Triangulated categories of
    singularities, and equivalences between Landau-Ginzburg
    models}. (Russian) Mat.\ Sb.\ {\bf 197} (2006), 117--132;
  translation in Sb.\ Math.\ {\bf 197} (2006), 1827--1840.
 
\bibitem{Orlov:3} \textsc{D.O.~Orlov}, {\em Derived categories of
    coherent sheaves and triangulated categories of
    singularities}. Algebra, arithmetic, and geometry: in honor of
  Yu.\ I.\ Manin.\ Vol.\ II, 503--531, Progr.\ Math., {\bf 270},
  Birkh\"auser Boston, Inc., Boston, MA, 2009.

\bibitem{Parshall}
\textsc{B.J.~Parshall and L.L.~Scott}, {\em Derived
categories, quasi-hereditary algebras, and algebraic groups},
Proceedings of the Ottawa-Moosonee Workshop in Algebra (1987), 105
pp., Carleton Univ., Ottawa, ON, 1988.

\bibitem{Psaroud}
\textsc{C.~Psaroudakis},
{\em Homological Theory of Recollements of Abelian Categories},
J.\ Algebra {\bf 398} (2014) 63--110.

\bibitem{PsaroudVitoria}
\textsc{C.~Psaroudakis and J.~Vit\'oria},
{\em Recollements of Module Categories},
accepted in Applied Categorical Structures, DOI 10.1007/s10485-013-9323-x.

\bibitem{Rickard}
\textsc{J.~Rickard},
{\em Derived categories and stable equivalence},
J.\ Pure Appl.\ Algebra {\bf 61} (1989) 303--317.

\bibitem{Rotman1}
\textsc{J.J~Rotman},
{\em An introduction to homological algebra}.
Academic Press, 2009.

\bibitem{triangular}
\textsc{S.~Smal{\o}},
{\em Functorial finite subcategories over triangular matrix rings},
Proc.\ Amer.\ Math.\ Soc.\ {\bf 111} (1991), 651--656.

\bibitem{SO}
\textsc{N.~Snashall and \O{}.~Solberg},
{\em Support varieties and Hochschild cohomology rings},
Proc.\ London Math.\ Soc.\ {\bf 88} (2004), 705--732.


\bibitem{Solberg:1}
\textsc{\O{}.~Solberg},
{\em Support varieties for modules and complexes},
Trens in representation theory of algebras and related topics, 239--270,
Contemp.\ Math., {\bf 406}, Amer.\ Math.\ Soc., Providence, RI, 2006.

\bibitem{V}
B.\ B., Venkov,
\emph{Cohomology algebras for some classifying spaces},
Dokl.\ Akad.\ Nauk SSSR \textbf{127} (1959)
943--944.

\bibitem{Suarez}
\textsc{M.~Suarez-Alvarez},
{\em The Hilton-Eckmann argument for the anti-commutativity of cup products},
Proc.\ Amer.\ Math.\ Soc.\ {\bf 132} (2004), 2241--2246.

\bibitem{Verdier}
\textsc{J.-L.~Verdier},
{\em Des cat\'egories d\'eriv\'ees des cat\'egories ab\'eliennes},
Ast\'erisque No.\ 239 (1996), xii+253 pp. (1997).

\bibitem{XiongZhang}
\textsc{B.~L.~Xiong and P.~Zhang},
{\em Gorenstein-projective modules over triangular matrix Artin algebras},
J.\ Algebra Appl.\ {\bf 11} (2012), 14 pp.

\end{thebibliography}
\end{document}